%% file: supermartTransport_arxiv.tex
\documentclass[11pt,reqno]{amsart}

\usepackage[T1]{fontenc}
\usepackage{amsfonts}
\usepackage{amsmath}
\usepackage{amssymb}
\usepackage{amsthm}
\usepackage{bbm}
\usepackage{bm}
\usepackage{mathrsfs}
\usepackage{verbatim}
\usepackage{color}
\usepackage{pdfsync}
\usepackage{enumitem}\setlist[enumerate]{leftmargin=*}
\usepackage{graphicx}
\usepackage{stmaryrd}
\usepackage[font=small]{caption}
\usepackage[letterpaper, textwidth=4.98in, textheight=7.61in, hmarginratio={1:1}, vmarginratio={1:1}, heightrounded]{geometry}

\usepackage{tikz}
\usetikzlibrary{arrows}
\usetikzlibrary{decorations.pathreplacing}

\theoremstyle{plain}
\newtheorem{theorem}{Theorem}[section]
\newtheorem{proposition}[theorem]{Proposition}
\newtheorem{lemma}[theorem]{Lemma}
\newtheorem{corollary}[theorem]{Corollary}

\theoremstyle{definition}
\newtheorem{definition}[theorem]{Definition}
\newtheorem{remark}[theorem]{Remark}

\newtheorem{example}[theorem]{Example}

\theoremstyle{remark}
\newcommand{\eps}{\varepsilon}

\newcommand{\N}{\mathbb{N}}

\newcommand{\R}{\mathbb{R}}

\newcommand{\Z}{\mathbb{Z}}

\newcommand{\cB}{\mathcal{B}}

\newcommand{\cD}{\mathcal{D}}

\newcommand{\cM}{\mathcal{M}}

\newcommand{\cP}{\mathcal{P}}

\newcommand{\cS}{\mathcal{S}}

\newcommand{\bI}{\mathbf{I}}

\newcommand{\bS}{\mathbf{S}}

\newcommand{\fM}{\mathfrak{M}}

\DeclareMathOperator{\bary}{bary}

\DeclareMathOperator{\supp}{supp}

\DeclareMathOperator{\Lip}{Lip}

\newcommand{\as}{\mbox{-a.s.}}
\newcommand{\qs}{\mbox{-q.s.}}
\newcommand{\1}{\mathbf{1}}

\newcommand{\rP}{{\overset{\;{}_{\rightarrow}}{P}}}
\newcommand{\lP}{{\overset{\;{}_{\leftarrow}}{P}}}
\newcommand{\shadow}[2]{S^{#2}(#1)}
\newcommand{\casts}[2]{\left\llbracket #1,#2 \right\rrbracket}
\newcommand{\smallcasts}[2]{\llbracket #1,#2 \rrbracket}

\numberwithin{equation}{section}

\usepackage[pdfborder={0 0 0}]{hyperref}
\hypersetup{
  urlcolor = black,
  pdfauthor = {Marcel Nutz, Florian Stebegg},
  pdfkeywords = {Coupling; Optimal Transport; Spence--Mirrlees Condition; Duality},
  pdftitle = {Canonical Supermartingale Couplings},
  pdfsubject = {Canonical Supermartingale Couplings},
  pdfpagemode = UseNone
}

\begin{document}
\title{Canonical Supermartingale Couplings}

\author{Marcel Nutz}
\thanks{M.\ Nutz is supported by an Alfred P.\ Sloan Fellowship and NSF Grant DMS-1512900. The authors would like to thank Mathias Beiglb\"ock, Nicolas Juillet, Jan Ob{\l}{\'o}j, Nizar Touzi and an anonymous referee for encouragement and advice.}

\author{Florian Stebegg}



\date{\today}

\begin{abstract}
Two probability distributions $\mu$ and $\nu$ in second stochastic order can be coupled by a supermartingale, and in fact by many. Is there a canonical choice? We construct and investigate two couplings which arise as  optimizers for constrained Monge--Kantorovich optimal transport problems where only supermartingales are allowed as transports. Much like the Hoeffding--Fr\'echet coupling of classical transport and its symmetric counterpart, the antitone coupling, these can be characterized by order-theoretic minimality properties, as simultaneous optimal transports for certain classes of reward (or cost) functions, and through no-crossing conditions on their supports; however, our two couplings have asymmetric geometries. Remarkably, supermartingale optimal transport decomposes into classical and martingale transport in several ways.\end{abstract}

\maketitle

\vspace{-0em}
\section{Introduction}\label{se:intro}

Let $\mu$ and $\nu$ be probability measures on the real line. A measure~$P$ on~$\R^{2}$ whose first and second marginals are $\mu$ and $\nu$, respectively, is called a coupling (or transport) of $\mu$ and $\nu$, and the set of all such measures is denoted by $\Pi(\mu,\nu)$. We shall be interested in couplings that are supermartingales; that is, if $(X,Y)$ denotes the identity on $\R^{2}$, then $E^{P}[Y|X]\leq X$ $P$-a.s. Thus, we assume throughout that $\mu$ and $\nu$ have a finite first moment, and denote by $\cS(\mu,\nu)$ the set of supermartingale couplings. A classical result of Strassen (cf.\ Proposition~\ref{pr:convexOrder}) shows that $\cS(\mu,\nu)$ is nonempty if and only if $\mu$ and $\nu$ are in convex-decreasing (or second stochastic) order, denoted $\mu\leq_{cd}\nu$ and defined by the requirement that $\mu(\phi)\leq\nu(\phi)$ for any convex and decreasing function $\phi$, where $\mu(\phi):=\int \phi\,d\mu$. Given $\mu\leq_{cd}\nu$, there are typically infinitely many supermartingale couplings. Our question: are there some special, ``canonical'' choices? The aim of this paper is to introduce and describe two such couplings, called the \emph{Increasing} and the \emph{Decreasing Supermartingale Transport} and denoted $\rP$ and $\lP$, respectively. They have remarkable properties that are, in several ways, analogous to the Hoeffding--Fr\'echet and antitone couplings which can be considered canonical choices in $\Pi(\mu,\nu)$ but typically are not supermartingales. The study undertaken is also a model problem of optimal transport under inequality constraints. We shall see that the supermartingale transport problem decomposes into unconstrained and equality (martingale) constrained transport, in multiple and sometimes unexpected ways.

\subsection{Synopsis}

The couplings $\rP$ and $\lP$ will be characterized in three different ways: an order-theoretic minimality property, optimality for a specific class of transport reward (or cost) functions, and a geometric property of the support stating that certain paths do or do not intersect.

Let us begin with the order-theoretic characterization. To explain the idea, suppose that~$\mu$ consists of finitely many atoms at $x_{1},\dots,x_{n}\in\R$, then a coupling of $\mu$ and $\nu$ can be defined by specifying a ``destination'' measure for each atom. We know from Strassen's result that the convex-decreasing order plays a special role, so it is natural to rank all possible destination measures for the first atom (as allowed by the given marginal $\nu$ and the supermartingale constraint) according to that order. A minimal element $\shadow{\mu|_{x_{1}}}{\nu}$ called the \emph{shadow}  will be shown to exist; essentially, it maximizes the barycenter of the destination measure and minimizes the variance. The procedure can be iterated after subtracting $\shadow{\mu|_{x_{1}}}{\nu}$ from $\nu$, and that determines a supermartingale coupling of $\mu$ and $\nu$. Depending on the order in which the atoms are processed, the coupling will have a very different structure. Two obvious choices are the increasing and the decreasing order of the $x_{k}$, and that gives rise to $\rP$ and $\lP$ (the arrows representing the order of processing). In the general, continuum version of the construction, we instead specify the destination of $\mu|_{(-\infty,x]}$ and $\mu|_{[x,\infty)}$ for each $x\in\R$. The following is taken from Theorem~\ref{th:canonicalCouplings} in the body of the paper; the precise definition of the shadow can be found in Lemma~\ref{le:ecdSandwich}.

\begin{theorem}\label{th:canonicalCouplingsIntro}
  There exists a unique measure $\rP$ on $\R^{2}$ which transports $\mu|_{(-\infty,x]}$ to its shadow $\shadow{\mu|_{(-\infty,x]}}{\nu}$ for all $x\in\R$.
    Similarly, there exists a unique measure $\lP$ which transports $\mu|_{[x,\infty)}$ to $\shadow{\mu|_{[x,\infty)}}{\nu}$ for all $x\in\R$. Moreover, these two measures are elements of $\cS(\mu,\nu)$. 
\end{theorem}

While the shadow construction illuminates the local order-theoretic nature of the couplings, it does not reveal the global geometric structure that is apparent in Figures~\ref{fi:simulationIncr} and~\ref{fi:simulationDecr}   (rendered on page~\pageref{fi:simulationDecr}). The figures show simulations of $\rP$ and $\lP$ for piecewise uniform marginals and discrete marginals; the mass is transported from the $x$-axis (top) to the $y$-axis (bottom).

\begin{figure}[p]
\begin{center}
  \emph{Increasing Supermartingale Transport} $\rP$

  \vspace{.5em}

  \input{simulationIncr2_pic}

  \vspace{.5em}

	\resizebox{3.8in}{!}{
	\begin{tikzpicture}[every node/.style={draw,circle,inner sep=0pt,minimum size=3pt}]
	\draw[thick] (-3,0) -- (6,0) node[right, draw = none] {\scriptsize $x$};
	\draw[thick] (-3,-2) -- (6,-2) node[right, draw = none] {\scriptsize $y$};
	\node[fill = white] (x0) at (1.0,0) {};
	\node[fill = white] (x1) at (0.5,0) {};
	\node[fill = white] (x2) at (0.0,0) {};
	\node[fill = white] (x4) at (3.0,0) {};
	\node[fill = white] (x5) at (3.5,0) {};
	\node[fill = white] (x6) at (4.0,0) {};
	\node[fill = white] (x7) at (4.5,0) {};
	\node[fill=black] (y0) at (-1.5,-2) {};
	\node[fill=black] (y1) at (-1.0,-2) {};
	\node[fill=black] (y2) at (-0.5,-2) {};
	\node[fill=black] (y3) at (0.5,-2) {};
	\node[fill=black] (y4) at (1.0,-2) {};
	\node[fill=black] (y5) at (1.5,-2) {};
	\node[fill=black] (y6) at (2.0,-2) {};
	\node[fill=black] (y7) at (2.5,-2) {};
	\node[fill=black] (y9) at (-2,-2) {};
	\node[fill=black] (y10) at (-2.5,-2) {};
	\draw[thick] (x0) -- (y0);
	\draw[thick] (x0) -- (y5);
	\draw[thick] (x1) -- (y1);
	\draw[thick] (x1) -- (y4);
	\draw[thick] (x2) -- (y2);
	\draw[thick] (x2) -- (y3);
	\draw[thick,densely dashed] (x4) -- (y7);
	\draw[thick,densely dashed] (x5) -- (y6);
	\draw[thick,densely dashed] (x6) -- (y9);
	\draw[thick,densely dashed] (x7) -- (y10);
	\end{tikzpicture}}
\end{center}
\vspace{.0em}\caption{Simulations of the Increasing Supermartingale Transport. We observe an interval of Left-Curtain kernels (black/continuous) on the left and an interval of antitone kernels (gray/dashed) on the right. The destinations of the right interval are on both sides of the destinations of the left one. (The definitions of $x^*$ and $M$ are given in Sections~\ref{se:BarriersAndPolarSets} and~\ref{se:monotonicityPrinciple}, respectively.)
}
\label{fi:simulationIncr}
\vspace{3.5em}
\begin{center}
  \emph{Decreasing Supermartingale Transport} $\lP$

  \vspace{0.5em}

  \input{simulationDecr2_pic}

  \vspace{.5em}
	
	\resizebox{3.8in}{!}{
	\begin{tikzpicture}[every node/.style={draw,circle,inner sep=0pt,minimum size=3pt}]
	\draw[thick] (-3,0) -- (6,0) node[right, draw = none] {\scriptsize $x$};
	\draw[thick] (-3,-2) -- (6,-2) node[right, draw = none] {\scriptsize $y$};
	\node[fill = white] (x0) at (-1.0,0) {};
	\node[fill = white] (x1) at (-0.5,0) {};
	\node[fill = white] (x2) at (0.0,0) {};
	\node[fill = white] (x3) at (2.5,0) {};
	\node[fill = white] (x4) at (3.0,0) {};
	\node[fill = white] (x5) at (3.5,0) {};
	\node[fill = white] (x6) at (4.0,0) {};
	\node[fill = white] (x7) at (4.5,0) {};
	\node[fill=black] (y0) at (-1.5,-2) {};
	\node[fill=black] (y1) at (-1.0,-2) {};
	\node[fill=black] (y2) at (-0.5,-2) {};
	\node[fill=black] (y3) at (0.5,-2) {};
	\node[fill=black] (y4) at (1.0,-2) {};
	\node[fill=black] (y5) at (1.5,-2) {};
	\node[fill=black] (y6) at (2.0,-2) {};
	\node[fill=black] (y7) at (2.5,-2) {};
	\node[fill=black] (y8) at (3.0,-2) {};
	\node[fill=black] (y9) at (3.5,-2) {};
	\node[fill=black] (y10) at (4.0,-2) {};
	\node[fill=black] (y11) at (5.0,-2) {};
	\node[fill=black] (y12) at (5.5,-2) {};
	\draw[thick] (x0) -- (y0);
	\draw[thick] (x0) -- (y5);
	\draw[thick] (x1) -- (y1);
	\draw[thick] (x1) -- (y4);
	\draw[thick] (x2) -- (y2);
	\draw[thick] (x2) -- (y3);
	\draw[thick,densely dashed] (x3) -- (y6);
	\draw[thick,densely dashed] (x4) -- (y7);
	\draw[thick,densely dashed] (x5) -- (y8);
	\draw[thick] (x6) -- (y9);
	\draw[thick] (x6) -- (y12);
	\draw[thick] (x7) -- (y10);
	\draw[thick] (x7) -- (y11);
	\end{tikzpicture}}
\end{center}
\vspace{0.0em}\caption{Simulations of the Decreasing Supermartingale Transport. We observe an interval of Right-Curtain kernels on the left, followed by an interval of Hoeffding--Fr\'echet kernels and another interval of Right-Curtain kernels. The destinations of these intervals preserve the original order.
}
\label{fi:simulationDecr}
\end{figure}

The Monge--Kantorovich optimal transport problem is a framework that enables a geometric description for its optimal transports, and thus it is desirable to represent $\rP$ and $\lP$ as corresponding solutions. More precisely, we shall introduce the \emph{supermartingale optimal transport} problem
\begin{equation}\label{eq:optTranspIntro}
  \sup_{P\in\cS(\mu,\nu)} P(f)
\end{equation}
where transports are required to be supermartingales, and then $\rP,\lP$ will be optimizers for reward functions $f$ satisfying certain geometric properties. To make the connection with other texts on optimal transport, notice that $P(f)=E^{P}[f(X,Y)]$ in our notation, and that $f$ can be seen as a cost function by a change of sign. We shall say that $f : \R^2 \to \R$ is \emph{supermartingale Spence--Mirrlees} if
\begin{equation}\label{eq:superMartSpenceIntro}
  f(x_{2},\cdot) - f(x_{1},\cdot)\mbox{ is strictly decreasing and strictly convex for all } x_{1}<x_{2}.
\end{equation}
If $f$ is smooth, this can be expressed through the cross-derivatives conditions $f_{xy}<0$ and $f_{xyy}>0$; the first one is the negative of the classical Spence--Mirrlees condition and the second is the so-called martingale Spence--Mirrlees condition. The following is a slightly simplified statement of Corollary~\ref{co:optTranspCharact}.

\begin{theorem}\label{th:optTranspCharactIntro}
  Let $f : \R^2 \to \R$ be Borel, supermartingale Spence--Mirrlees and suppose that there exist $a\in L^{1}(\mu)$, $b\in L^{1}(\nu)$ such that 
  $$
    |f(x,y)|\leq a(x) + b(y),\quad x,y\in\R.
  $$
  Then, $\rP$ is the unique solution of the supermartingale optimal transport problem~\eqref{eq:optTranspIntro}. Similarly, $\lP$ is the unique solution of $\inf_{P\in\cS(\mu,\nu)} P(f)$, or equivalently of~\eqref{eq:optTranspIntro} if instead  $-f$ is supermartingale Spence--Mirrlees.
\end{theorem}

Since $\rP$ and $\lP$ correspond to the combinations $f_{xy}<0, f_{xyy}>0$ and $f_{xy}>0, f_{xyy}<0$ of known conditions, it is natural to ask for the remaining two combinations, $f_{xy}>0,f_{xyy}>0$ and $f_{xy}<0,f_{xyy}<0$. While the associated optimal transports also have interesting features, they turn out to depend on the function $f$ within that class and hence, \emph{cannot} be called canonical; cf.\ Section~\ref{se:noncanonicalCouplings}.

The third characterization of $\rP$ and $\lP$ is through their supports. A point $(x,y)$ in the support can be thought of as a path that the transport is using, and the conditions are expressed as crossing or no-crossing conditions between the paths of the transport. While this characterization is an incarnation of the \emph{$c$-cyclical monotonicity} of classical transport, the supermartingale constraint requires a novel distinction of the origins $x$ into a set $M$ of ``martingale points'' and their complement. Intuitively, the supermartingale constraint is binding at points of $M$ and absent elsewhere---this will be made precise later on (Corollary~\ref{co:extremalDecomp}). Thus, we work with a Borel set $\Gamma\in\cB(\R^{2})$ that should be thought of as a support, and a second set $M\in\cB(\R)$. 
Consider arbitrary paths $(x_{1},y_{1}),(x_{2},y_{2})\in\Gamma$ with $x_{1}<x_{2}$; then, we call the pair $(\Gamma,M)$
 \begin{enumerate}
 \item  first-order left-monotone if $y_{1}\leq y_{2}$ whenever $x_{2}\notin M$,
 \item  first-order right-monotone if $y_{2}\leq y_{1}$ whenever $x_{1}\notin M$.
 \end{enumerate}
We also need the following properties of $\Gamma$ alone: considering three paths  
$(x,y_{1}),(x,y_{2}),(x',y')\in\Gamma$ with $y_{1}<y_{2}$, the set $\Gamma$ is second-order left-monotone (right-monotone) if $y'\notin (y_{1},y_{2})$ whenever $x<x'$ ($x>x'$). The latter two properties are taken from~\cite{BeiglbockJuillet.12} where they are simply called left- and right-monotonicity, and all four properties are summarized in Figure~\ref{fi:forbiddenConfigs}.

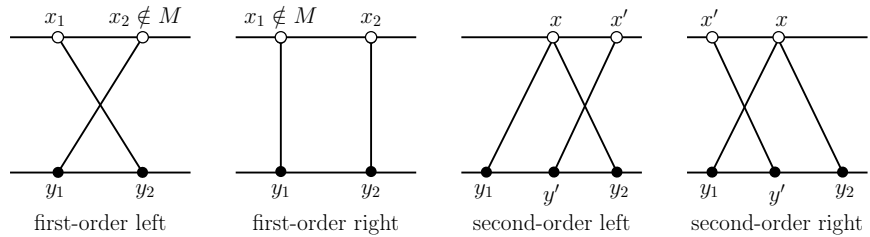
\begin{figure}[ht]
\begin{center}
\scalebox{0.6}{
\begin{tikzpicture}
\draw[very thick] (0,0) -- (4,0) {};
\draw[very thick] (0,3) -- (4,3) {};
\draw[very thick, o-*] (1,3.1) node[above] {\Large $x_1$} -- (3,-0.1) node[below] {\Large $y_2$};
\draw[very thick,o-*] (3,3.1) node[above] {\Large $x_2\notin M$} -- (1,-0.1) node[below] {\Large $y_1$};
\node[draw,thick,circle,fill=white,minimum size=1pt,inner sep=2.5pt] at (1.06,3) {};
\node[draw,thick,circle,fill=white,minimum size=1pt,inner sep=2.5pt] at (2.94,3) {};
\node[draw = none] at (2,-1.1) {\Large first-order left};

\draw[very thick] (5,0) -- (9,0) {};
\draw[very thick] (5,3) -- (9,3) {};
\draw[very thick,o-*] (6,3.1) node[above] {\Large $x_1\notin M$} -- (6,-0.1) node[below] {\Large $y_1$};
\draw[very thick,o-*] (8,3.1) node[above] {\Large $x_2$} -- (8,-0.1) node[below] {\Large $y_2$};
\node[draw,thick,circle,fill=white,minimum size=1pt,inner sep=2.5pt] at (6,3) {};
\node[draw,thick,circle,fill=white,minimum size=1pt,inner sep=2.5pt] at (8,3) {};
\node[draw = none] at (7,-1.15) {\Large first-order right};

\draw[very thick] (10,0) -- (14,0) {};
\draw[very thick] (10,3) -- (14,3) {};
\draw[very thick,o-*] (12.1,3.1) node[above] {\Large $x$} -- (10.5,-0.1) node[below] {\Large $y_1$};
\draw[very thick,-*] (12,3) -- (13.5,-0.1) node[below] {\Large $y_2$};
\draw[very thick,o-*] (13.5,3.1) node[above] {\Large $x'$} -- (12,-0.1) node[below] {\Large $y'$};
\node[draw,thick,circle,fill=white,minimum size=1pt,inner sep=2.5pt] at (12.03,3) {};
\node[draw,thick,circle,fill=white,minimum size=1pt,inner sep=2.5pt] at (13.44,3) {};
\node[draw = none] at (12,-1.1) {\Large second-order left};

\draw[very thick] (15,0) -- (19,0) {};
\draw[very thick] (15,3) -- (19,3) {};
\draw[very thick,o-*] (17.1,3.1) node[above] {\Large $x$} -- (15.5,-0.1) node[below] {\Large $y_1$};
\draw[very thick,-*] (17,3) -- (18.5,-0.1) node[below] {\Large $y_2$};
\draw[very thick,o-*] (15.5,3.1) node[above] {\Large $x'$} -- (17,-0.1) node[below] {\Large $y'$};
\node[draw,thick,circle,fill=white,minimum size=1pt,inner sep=2.5pt] at (17.03,3) {};
\node[draw,thick,circle,fill=white,minimum size=1pt,inner sep=2.5pt] at (15.56,3) {};
\node[draw = none] at (17,-1.15) {\Large second-order right};
\end{tikzpicture}} 
\vspace{-1em}\end{center}
\caption{\emph{Forbidden} configurations in the monotonicity properties}
\label{fi:forbiddenConfigs}
\end{figure}

The following result is the summary of Theorem~\ref{th:geomCharactRP} and Corollary~\ref{co:geomCharactRPConverse} in the body of the paper.

\begin{theorem}\label{th:geomCharactRPIntro}
  There exist nondegenerate\footnote{This is a minor notion detailed in Definition~\ref{de:nondegenerate}.} $(\Gamma,M)\in\cB(\R^{2})\times\cB(\R)$ which are first-order right-monotone and second-order left-monotone such that $\rP$ is concentrated on $\Gamma$ and $\rP|_{M\times\R}$ is a martingale.
  Conversely, if $(\Gamma,M)$ have those properties and $P\in\cS(\mu,\nu)$ is a transport which is concentrated on $\Gamma$ and $P|_{M\times\R}$ is a martingale, then $P=\rP$.
  
  The analogous statement, interchanging left and right, holds for $\lP$.
\end{theorem}

With some additional work, these theorems will allow us to explain the geometric features apparent in Figures~\ref{fi:simulationIncr} and~\ref{fi:simulationDecr}. To that end, let us first recall two pairs of related couplings.

Our characterizations highlight the analogies between $\rP,\lP$ and the classical Hoeffding--Fr\'echet and antitone couplings $P_{HF}, P_{AT}\in\Pi(\mu,\nu)$; see, e.g., \cite[Section~3.1]{RachevRuschendorf.98a}. Indeed, the latter have a minimality property similar to Theorem~\ref{th:canonicalCouplingsIntro}, but for the first stochastic order instead of the convex-decreasing one. Moreover, they are optimal transports for reward functions satisfying the classical Spence--Mirrlees condition $f_{xy}>0$ and its reverse, and they are characterized by what we called the first-order left- and right-monotonicity of their supports $\Gamma$ (with $M=\R$). 

The second pair of related couplings is given by the Left- and Right-Curtain couplings $P_{LC}, P_{RC}$ introduced in \cite{BeiglbockJuillet.12} where martingale transport is studied; that is, the given marginals are in \emph{convex} order and the transports are martingales. Indeed, these couplings are special cases of $\rP$ and $\lP$ that arise when the marginals $\mu\leq_{cd}\nu$ have the same barycenter---this corresponds to the fact that a supermartingale 
with constant mean is a martingale and vice versa. In that case, the first-order properties turn out to be irrelevant: in the shadow construction, the barycenter is constant and hence only the variance needs to be minimized; the condition for the reward functions is $f_{xyy}>0$ (or $<0$), and the second-order monotonicity property of $\Gamma$ alone describes the support. As we shall see, it is the \emph{interaction} between the first and second-order properties as well as the set $M$ that generates the rich structure of $\rP$ and $\lP$.

Turning to $\rP$ in Figure~\ref{fi:simulationIncr}, the first observation is that there are only two types of transport kernels. On the left, $\rP$ uses martingale kernels of the Left-Curtain type: each point on the $x$-axis is mapped to two points on the $y$-axis, and any two points $x,x'$ satisfy the condition of second-order left-monotonicity. On the right, the transport is of Monge-type (each point $x$ is mapped to a single point~$y$) and has the structure of an antitone coupling: any two paths intersect, which is the first-order right-monotonicity property. On the strength of the same  property, points $x$ in the portion to the right (thus not in $M$) can further be divided into two groups---the left group is mapped to points $y$ to the right of the destinations of the martingale points, and vice versa. These facts about $\rP$ are true not only in our example, but for arbitrary atomless marginals $\mu\leq_{cd}\nu$; see Remark~\ref{rk:propertiesFromIntro}.

Let us now turn to $\lP$ in Figure~\ref{fi:simulationDecr}. Similarly as before, we observe two types of paths; the Right-Curtain and the Hoeffding--Fr\'echet kernels. However, the intervals of martingale and  deterministic transport alternate twice---there is no longer a unique phase transition; in general, there can be countably many transitions. On the other hand, the order of the intervals is now preserved by the transport---this corresponds to the combination of the first- and second-order properties. These two differences show that the geometries of $\lP$ and $\rP$ differ fundamentally and suggest that one cannot obtain one coupling from the other by a transformation of the marginals. By contrast, it is well known that $P_{AT}$ can be constructed from $P_{HF}$ via the transformation $(x,y)\mapsto (x,-y)$, whereas $P_{RC}$ can be obtained from $P_{LC}$ via $(x,y)\mapsto (-x,-y)$.

One common feature of $\rP$ and $\lP$ is that each consists of an optimal martingale transport and an optimal (unconstrained) Monge--Kantorovich transport. That turns out to be a general fact: a result that we call Extremal Decomposition (Corollary~\ref{co:extremalDecomp}) states that given an optimal supermartingale transport~$P$ for an arbitrary reward function $f$, the restriction of $P$ to $M\times\R$ is an optimal martingale transport and the restriction to $M^{c}\times\R$ is an optimal Monge--Kantorovich transport between its own marginals. (These marginals, however, are not easily determined a priori.)

\subsection{Methodology and Literature}

Most of our results are based on the study of the optimal transport problem~\eqref{eq:optTranspIntro}. We analyze this problem for general, Borel-measurable reward functions $f$, formulate a corresponding dual problem and establish strong duality; i.e., absence of a duality gap and existence of dual optimizers. A formal application of Lagrange duality suggests to consider triplets $\varphi\in L^{1}(\mu)$, $\psi\in L^{1}(\mu)$, $h:\R\to\R_{+}$ such that
\begin{equation}\label{eq:formalDualIntro}
  \varphi(x)+\psi(y)+h(x)(y-x)\geq f(x,y),\quad (x,y)\in\R^{2}
\end{equation}
and define the dual value as $\inf \{\mu(\varphi)+\nu(\psi)\}$, where the infimum is taken over all triplets. Indeed, $\varphi$ and $\psi$ are Lagrange multipliers for the constraints~$\mu$ and~$\nu$, whereas $h(x)(y-x)$ with $h\geq0$ represents the supermartingale constraint $E^{P}[Y|X]\leq X$. We refer to \cite[Section~5]{HobsonNeuberger.12} for an intuitive discussion of the Lagrangian approach. While the corresponding duality for standard transport (without $h$) is valid by the celebrated result of \cite{Kellerer.84}, the dual problem for the supermartingale case needs to be relaxed in three ways to avoid a duality gap and ensure dual existence  (Theorem~\ref{th:dualityGlobal}). Namely, the range of $h$ needs to be widened on parts of the state space, the integrability of $\varphi$ and $\psi$ needs to be loosened, and the inequality~\eqref{eq:formalDualIntro} needs to be relaxed on paths $(x,y)$ that are not used by any transport (see Section~\ref{se:counterexDuality} for  pertinent counterexamples). In particular, it is important to classify all obstructions to supermartingale couplings; i.e., ``barriers'' that cannot be crossed (Proposition~\ref{pr:barriers}). Remarkably, there are no barriers beyond a specific point as soon as the barycenters of the marginals are not identical: the state space decomposes into one half-plane behaving as in the martingale case and another half-plane behaving as in classical transport.

For the martingale transport, a related duality theory was provided in~\cite{BeiglbockNutzTouzi.15}. In that case, the barycenters of the marginals agree and the compactness arguments underlying the duality focus on controlling the convexity of certain functions. While we shall greatly benefit from those ideas, the supermartingale case requires us to control simultaneously first and second order properties (slope and convexity) which gives rise to substantial differences on the technical side; in fact, it turns out that controlling the slope necessitates a nontrivial increment between the barycenters of $\mu$ and $\nu$. The above-mentioned decomposition of the state space is instrumental here.

Strong duality results in a monotonicity principle (Theorem~\ref{th:monotonicityPrinciple}) along the lines of the $c$-cyclical monotonicity condition  of classical transport (e.g., \cite[Theorem 2.13]{AmbrosioGigli.13}): a variational result linking the optimality of a transport to the pointwise properties of its support. This principle is our main tool to study the couplings $\rP$ and $\lP$, parallel to the celebrated variational principle for the martingale case in \cite{BeiglbockJuillet.12} which has pioneered the idea that concepts similar to cyclical monotonicity can be developed beyond the classical transport setting. In the supermartingale transport problem, the monotonicity principle has a novel form describing a pair $(\Gamma,M)$ of sets as in Theorem~\ref{th:geomCharactRPIntro} rather than the support $\Gamma$ alone. The set $M$ enters the variational formulation by determining the class of competitors, much like it determines which paths are subject to the first-order monotonicity condition, and turns out to be fundamental in determining the geometries of $\rP$ and $\lP$.

As a variational result, the monotonicity principle necessitates knowing a priori that an optimal transport exists. We show that a supermartingale Spence--Mirrlees function $f$ is automatically continuous (Proposition~\ref{pr:spenceMirrlessCont}) in a tailor-made topology that is coarse enough to preserve weak compactness of $\cS(\mu,\nu)$, and that yields the required existence. This result, together with the purely geometric formulation of the Spence--Mirrlees conditions (Definition~\ref{de:spenceMirrlees}), also improves the literature on martingale transport \cite{BeiglbockJuillet.12, HenryLabordereTouzi.13, Juillet.14} where a range of assumptions is imposed on~$f$ both to ensure existence and to express the Spence--Mirrlees condition in terms of partial derivatives or a specific functional form; cf.\ Corollary~\ref{co:optTranspCharact}. A second generalization is that Theorem~\ref{th:optTranspCharactIntro} remains true if the Spence--Mirrlees condition~\eqref{eq:superMartSpenceIntro} is satisfied in the non-strict sense, except that the optimizer need not be unique.

With the appropriate notions in place, the proofs of Theorems~\ref{th:optTranspCharactIntro} and~\ref{th:geomCharactRPIntro} use the monotonicity principle to analyze the interplay between the geometry of the set $M$ and the first- and second-order monotonicity and Spence--Mirrlees conditions. The construction of~$\rP$ and~$\lP$ with the minimality property of Theorem~\ref{th:canonicalCouplingsIntro} rests on the precise understanding of the shadow of a single atom (Lemma~\ref{le:shadowOfDirac}) and compactness arguments;
a novel phenomenon is that the barycenter of the shadow needs to be found through an optimization rather than being known a priori as in the martingale case.

To the best of our knowledge, supermartingale couplings have not been specifically studied in the extant literature. However, as indicated above, martingale optimal transport has received considerable attention since it was introduced in \cite{BeiglbockHenryLaborderePenkner.11} and \cite{GalichonHenryLabordereTouzi.11}. In particular, \cite{BeiglbockJuillet.12,HenryLabordereTouzi.13, HobsonKlimmek.15,HobsonNeuberger.12,Juillet.14} study optimal martingale transports between two marginals for specific cost functions; the martingale Spence--Mirrlees condition in the form $f_{xyy}>0$ appears for the first time in \cite{HenryLabordereTouzi.13}, generalizing the functional form used in \cite{BeiglbockJuillet.12}. The non-strict condition, as well as the geometric definition, are novelties of this paper. We also remark that some of the technical developments in Sections~\ref{se:spenceMirrleesAndGeom}--\ref{se:geomCharact} provide simplifications with respect to previous works, when specialized to the martingale case.

While martingale (equality) constraint and classical (unconstrained) case can occur as special cases of supermartingale transport, the more surprising discovery is that the latter can be ``built'' from these two ingredients: the supermartingale (inequality) constraint is decomposed into two extremal cases. This forms a common thread in this paper, starting with the analysis of the maximal barrier which splits the plane into half-planes behaving like in these two cases and thus allows us to apply the compactness result of Proposition~\ref{pr:closednessIrred}. The variational principle decomposes the domain into points where the supermartingale constraint is felt as an equality constraint and points where it is not felt at all, and correspondingly, the Extremal Decomposition shows that any optimal supermartingale transport can be split into a martingale-optimal one and an  optimizer of an unconstrained problem. Conversely, our study of optimal supermartingale transports for Spence--Mirrlees reward functions in Sections~\ref{se:spenceMirrleesAndGeom}--\ref{se:geomCharact} shows how the geometric properties of Left-Curtain and Fr\'echet--Hoeffding couplings interact to create the patterns of the canonical supermartingale couplings.

Martingale optimal transport is motivated by considerations of model uncertainty in financial mathematics.
If, in the financial context, dynamic hedging is restricted by a no-shorting constraint, the dual problem  is supermartingale transport. Thus, it can be seen as a special case of the dual problem in \cite{FahimHuang.14} where general portfolio constraints are studied.
For background on Monge--Kantorovich  transport, we refer to \cite{AmbrosioGigli.13,RachevRuschendorf.98a,RachevRuschendorf.98b,Villani.03,Villani.09}. Recently, a rich literature has emerged around martingale  transport and model uncertainty; see \cite{Hobson.11, Obloj.04, Touzi.14} for surveys and, e.g., \cite{AcciaioBeiglbockPenknerSchachermayer.12, BeiglbockNutz.14, BouchardNutz.13, BurzoniFrittelliMaggis.15, CampiLaachirMartini.14, CheriditoKupperTangpi.14, DeMarcoHenryLabordere.15, FahimHuang.14, GhoussoubKimLim.15, GozlanRobertoSamsonTetali.14, Griessler.16, Nutz.13, NutzStebeggTan.17, Zaev.14} for models in discrete time, \cite{BeiglbockCoxHuesmannPerkowskiPromel.15, BiaginiBouchardKardarasNutz.14, CoxHouObloj.14, CoxObloj.11, DolinskyNeufeld.15, 
DolinskySoner.12, DolinskySoner.14, HenryLabordereOblojSpoidaTouzi.12, HenryLabordereTanTouzi.14, Hobson.98, NeufeldNutz.12, Nutz.14, 
Stebegg.14, TanTouzi.11} for continuous time, and \cite{BeiglbockCoxHuesman.14, BeiglbockHenryLabordereTouzi.15, BeiglbockHuesmannStebegg.15, Cox.08, CoxOblojTouzi.15, GuoTanTouzi.15, GuoTanTouzi.15a, GuoTanTouzi.15b, HirschProfetaRoynetteYor.11, Hobson.15, KallbladTanTouzi.15, OblojSpoida.13}
for related Skorokhod embedding and mimicking problems.

The remainder  of this paper is organized as follows. While Section~\ref{se:prelimSecondStochOrder} recalls basic facts related to the convex-decreasing order, Section~\ref{se:BarriersAndPolarSets} contains a complete description of the barriers to supermartingale couplings and more precisely, the structure of $\cS(\mu,\nu)$-polar sets. After these preparations, Section~\ref{se:duality} presents a complete duality theory for Borel reward functions, and Section~\ref{se:monotonicityPrinciple} formulates the resulting monotonicity principle. Section~\ref{se:shadowConstruction} introduces the couplings $\rP$ and $\lP$ via the shadow construction. In Section~\ref{se:spenceMirrleesAndGeom}, we propose the Spence--Mirrlees conditions for reward functions and show via the monotonicity principle that the associated optimal transports are supported on sets $(\Gamma,M)$ satisfying corresponding monotonicity properties. Section~\ref{se:geomCharact} continues the analysis by showing that any coupling supported on such sets must coincide with $\rP$ or $\lP$, respectively. In Section~\ref{se:regularityOfSpenceMirrlees}, we close the circle:  Spence--Mirrlees functions are shown to admit optimal transports and on the strength of the duality theory, that allows us to conclude the existence of suitable sets $(\Gamma,M)$. The main theorems stated in the Introduction then follow. The concluding Section~\ref{se:counterex} collects a number of counterexamples.

\section{Preliminaries}\label{se:prelimSecondStochOrder}

It will be useful to consider finite measures, not necessarily normalized to be probabilities. Let $\mu,\nu$ be finite measures on $\R$ with finite first moment. Extending the notation from the Introduction, we write $\Pi(\mu,\nu)$ for the set of all couplings; i.e., measures $P$ on $\R^{2}$ such that $P\circ X^{-1}=\mu$ and $P\circ Y^{-1}=\nu$, where $(X,Y):\R^{2}\to\R^{2}$ is the identity. Moreover, $\cS(\mu,\nu)$ is the subset of all $P\in\Pi(\mu,\nu)$ which are supermartingales; i.e., $\int Y\1_{A}(X)\,dP \leq \int X\1_{A}(X)\,dP$ for all $A\in\cB(\R)$, and finally $\cM(\mu,\nu)$ consist of all $P\in\Pi(\mu,\nu)$ satisfying this condition with equality.

We say that~$\mu$ and~$\nu$ are in \emph{convex-decreasing order}, or \emph{second stochastic order}, denoted $\mu\leq_{cd}\nu$, if $\mu(\phi)\leq \nu(\phi)$ for any convex, nonincreasing function $\phi: \R\to \R$. It then follows that $\mu$ and $\nu$ have the same total mass; moreover, we shall use repeatedly that it suffices to check the inequality for functions~$\phi$ of linear growth. An alternative characterization of this order refers to the put (price) function, defined by
$$
  p_{\mu}: \R\to\R,\quad   p_{\mu}(t) := \int (t-s)^{+}\, \mu(ds).
$$
Writing $\bary(\mu) := (\int x \,d\mu)/\mu(\R)$ for the barycenter (with $\bary(\mu):=0$ if $\mu=0$) and $\partial^{\pm}p_\mu$ for the right and left derivatives, the following properties are easily verified:

\begin{enumerate} 
  \item $p_\mu$ is nonnegative, increasing\footnote{Throughout this paper, increasing means nondecreasing.},
  convex, and $\partial^{+}p_\mu(t) - \partial^{-}p_\mu(t) = \mu(\{t\})$,  
  \item $\lim_{t \to -\infty} p_\mu(t) = 0$ and $\lim_{t \to \infty} p_\mu(t) = \infty\1_{\mu\neq0}$,
  \item $\lim_{t \to \infty} \{p_\mu(t) - \mu(\R)[t-\bary(\mu)]\} = 0$.
\end{enumerate}

In particular, we may extend $p_\mu$ continuously to $\overline{\R}=[-\infty,\infty]$. The following result is classical; see, e.g., \cite[Theorem~2.58]{FollmerSchied.11}.

\begin{proposition}\label{pr:convexOrder} Let $\mu,\nu$ be finite measures on $\R$ with finite first moment and $\mu(\R)=\nu(\R)$. The following are equivalent:
  \begin{enumerate}
	  \item  $\mu\leq_{cd}\nu$,
	  \item $p_{\mu}\leq p_{\nu}$,
	  \item $\cS(\mu,\nu)\neq \emptyset$,
		\item there exists a stochastic kernel $\kappa(x,dy)$ with finite mean such that $\int y \,\kappa(x,dy)\leq x$ for all $x\in\R$ and $\nu= (\mu\otimes \kappa) \circ Y^{-1}$, where $\mu\otimes \kappa$ denotes the product. 
	\end{enumerate}
\end{proposition}

In all that follows, the statement $\mu\leq_{cd}\nu$ implicitly means that $\mu,\nu$ are finite measures on $\R$ with finite first moment. Moreover, such a pair and the corresponding supermartingale optimal transport problem will be called \emph{proper} if the barycenters of $\mu$ and $\nu$ do \emph{not} coincide. In the improper case, the problem degenerates to a martingale optimal transport problem because any supermartingale with constant mean is a martingale. Indeed, let us convene that~$\mu$ and~$\nu$ are in \emph{convex order}, denoted $\mu\leq_{c}\nu$, if $\mu(\phi)\leq \nu(\phi)$ for any convex function $\phi: \R\to \R$, and introduce the symmetric potential function 
$u_{\mu}: \R\to\R$ by $u_{\mu}(t) := \int |t-s|\, \mu(ds)$.
Given $\mu\leq_{cd}\nu$, the following are then equivalent:
(a)~$\bary(\mu)=\bary(\nu)$,
(b)~$\mu\leq_{c}\nu$,
(c)~$u_{\mu}\leq u_{\nu}$,
(d)~$\cM(\mu,\nu)\neq \emptyset$,
(e)~the kernel $\kappa$ in (iv) can be chosen with $\int y \,\kappa(x,dy)= x$ for all $x\in\R$.

\section{Barriers and Polar Sets}\label{se:BarriersAndPolarSets}

We fix $\mu\leq_{cd}\nu$ throughout this section. Our first aim is to characterize all points $x\in\overline\R$ which cannot be crossed by any supermartingale transport $P\in\cS(\mu,\nu)$.

\begin{definition}\label{de:barrier}
  A point $x\in\overline\R$ is called a \emph{barrier} if $Y\leq x$ $P$-a.s.\ on $\{X\leq x\}$ and $Y\geq x$ $P$-a.s.\ on $\{X\geq x\}$, for all $P\in\cS(\mu,\nu)$.
\end{definition}

We may note that $\pm\infty$ are always barriers. The following result not only shows how barriers can be described as points where the put functions touch, but also introduces a particular barrier $x^{*}$ which divides the real line into two parts: To the left of $x^{*}$, the supermartingale transport problem is in fact just a martingale transport problem. To the right of $x^{*}$, we have a proper supermartingale transport problem and there are no non-trivial barriers. For example, in Figure~\ref{fi:simulationIncr}, the point $x^*$ is the left boundary of
the support of $\nu$, whereas in Figure~\ref{fi:simulationDecr} it indeed splits
the supports of $\mu$ and $\nu$ into two parts.
The convention $\sup\emptyset = -\infty$ is used.

\begin{proposition}\label{pr:barriers}
  Define
  $
    x^{*} := \sup \{x\in\R: \, p_{\mu}(x)=p_{\nu}(x) \} \in \overline{\R}.
  $
  Then 
  \begin{enumerate}
  \item $x^{*}$ is a barrier and $p_{\mu}(x^{*})=p_{\nu}(x^{*})$,
  \item a point $x\in [-\infty,x^{*})$ is a barrier if and only if $p_{\mu}(x)=p_{\nu}(x)$,
  \item if $x\in (x^{*},\infty]$ is a barrier then $\mu(x,\infty)=\nu(x,\infty)=0$.
  \end{enumerate}
  Moreover, $x^{*}$ is the maximal barrier $x\in\overline\R$ such that $P|_{\{X< x\}}$ is a martingale transport for some (and then all) $P\in\cS(\mu,\nu)$.
\end{proposition}

The reverse implication in (iii) is almost true: a point $x$ with $\mu(x,\infty)=\nu(x,\infty)=0$ is not crossed by any transport. However, if $\mu$ has an atom at~$x$, this mass may be transported to $(-\infty,x)$ and then $x$ does not satisfy our definition of a barrier which is chosen so that any mass at the barrier remains invariant.

Before reporting the proof in Section~\ref{se:proofOfBarriers}, we use the above result to characterize the polar sets and the irreducible components.%

\begin{definition}\label{de:irred}
  The pair $\mu\leq_{cd}\nu$ is \emph{irreducible} if the set $I=\{p_{\mu}< p_{\nu}\}$ is connected and $\mu(I)=\mu(\R)$. In this situation, let~$J$ be the union of~$I$ and any endpoints of $I$ that are atoms of $\nu$; then~$(I,J)$ is the \emph{domain} of $(\mu,\nu)$.
\end{definition} 

This definition coincides with the notion of \cite{BeiglbockJuillet.12, BeiglbockNutzTouzi.15} in the context of martingale transport. More precisely, for $x<x^{*}$, we have $p_{\mu}(x)=p_{\nu}(x)$ if and only if $u_{\mu}(x)=u_{\nu}(x)$.

In the general case, the supermartingale transport problem will be decomposed into at most countably many irreducible components. 
We recall that a set is called polar for a family $\cP$ of measures if it is $P$-null for all $P\in\cP$.

\begin{proposition}\label{pr:decomp}
  Let $\mu\leq_{cd}\nu$, let $I_{0}=(x^{*},\infty)$ and let $(I_{k})_{1\leq k \leq N}$ be the (open) components of $\{p_{\mu}<p_{\nu}\}\cap (-\infty,x^{*})$, where $N\in \{0,1,\dots,\infty\}$.
 
 (i) Set $I_{-1}=\R\setminus \cup_{k\geq0} I_{k}$ and $\mu_{k}=\mu|_{I_{k}}$ for $k\geq -1$, so that $\mu=\sum_{k\geq-1} \mu_{k}$.
	  Then, there exists a unique decomposition $\nu=\sum_{k\geq-1} \nu_{k}$ such that
	  $$
	    \mu_{-1} = \nu_{-1} \quad\;\; \mbox{and}\quad\;\; \mu_{0}\leq_{cd} \nu_{0}  \quad\;\; \mbox{and}\quad\;\; \mu_{k}\leq_{c} \nu_{k} \quad \mbox{for all} \quad k\geq1.
	  $$  
	  Moreover, this decomposition satisfies $I_{k}=\{p_{\mu_{k}}<p_{\nu_{k}}\}$ for all $k\geq0$; i.e., each such pair $(\mu_{k},\nu_{k})$ is irreducible.  
	  Finally, any $P\in\cS(\mu,\nu)$ admits a unique decomposition
	  $
	    P=\sum_{k\geq-1} P_{k}
	  $
	  such that $P_{0}\in\cS(\mu_{0},\nu_{0})$ and  $P_{k}\in\cM(\mu_{k},\nu_{k})$ for all $k\neq 0$. %
	  
	 (ii) Let $B\subseteq \R^{2}$ be a Borel set. Then $B$ is $\cS(\mu,\nu)$-polar if and only if there exist a $\mu$-nullset $N_{\mu}$ and a $\nu$-nullset $N_{\nu}$ such that
	  $$
	     B \subseteq (N_{\mu}\times \R) \cup (\R \times N_{\nu}) \cup \bigg(\Delta \cup \bigcup_{k\geq0} I_{k}\times J_{k}\bigg)^{c},
	  $$
	  where $\Delta=\{(x,x)\in\R^{2}:\, x\in\R\}$ is the diagonal and $J_{k}$ is constructed from~$I_{k}$ as in Definition~\ref{de:irred}.
\end{proposition}

\subsection{Proofs of Propositions~\ref{pr:barriers} and~\ref{pr:decomp}}\label{se:proofOfBarriers}

We begin with the proof of Proposition~\ref{pr:barriers}, stated through a sequence of lemmas. We may assume that $\mu$ and $\nu$ are probability measures.

\begin{lemma}\label{le:barrier}
 Let $x\in\overline\R$. If $p_{\mu}(x)=p_{\nu}(x)$, then $x$ is a barrier and the equality $E^{P}[X\1_{X<x}] = E^{P}[Y\1_{X<x}]$ holds for all $P\in\cS(\mu,\nu)$.
\end{lemma}

\begin{proof}
  Let $p_{\mu}(x)=p_{\nu}(x)$ and let $E[\,\cdot\,]$ be the expectation associated with an arbitrary $P\in\cS(\mu,\nu)$. Using $E[Y|X] \leq X$ and Jensen's inequality,
  $
    (x-X)^{+} \leq (x-E[Y|X])^{+} \leq E[(x-Y)^{+}|X],
  $
  and since $p_{\mu}(x)=p_{\nu}(x)$ means that $E[(x-X)^{+}]=E[(x-Y)^{+}]$, it follows that 
  $
    (x-X)^{+} = E[(x-Y)^{+}|X].
  $
  
  As a first consequence, 
  $
    E[(x-Y)^{+}\1_{X\geq x}] = E[(x-X)^{+}\1_{X\geq x}]=0
  $
  and hence $Y\geq x$ $P$-a.s.\ on $\{X\geq x\}$.
  A second consequence is
  $
    E[(x-Y)\1_{X\leq x}] \leq   E[(x-Y)^{+}\1_{X\leq x}] = E[(x-X)^{+}\1_{X\leq x}].
  $
  Since $E[Y|X] \leq X$ implies that
  $
    E[(x-Y)\1_{X\leq x}]\geq E[(x-X)\1_{X\leq x}] = E[(x-X)^{+}\1_{X\leq x}],
  $
  it follows that 
  $
    E[(x-Y)\1_{X\leq x}] = E[(x-Y)^{+}\1_{X\leq x}]
  $
  and thus $Y\leq x$ $P$-a.s.\ on $\{X\leq x\}$. This completes the proof of the barrier property.

  The above inequalities also show that $E[(x-Y)\1_{X \leq x}] = E[(x-X)\1_{X \leq x}]$ and hence $E[Y\1_{X \leq x}] = E[X\1_{X \leq x}]$. To infer the second assertion, it remains to note that the barrier property implies that $Y = x$ $P$-a.s.\ on $\{X = x\}$.
\end{proof}

\begin{corollary}\label{co:threshold}
  We have $p_{\mu}(x^{*})=p_{\nu}(x^{*})$ and $E^{P}[X\1_{X<x^{*}}] = E^{P}[Y\1_{X<x^{*}}]$ for all $P\in\cS(\mu,\nu)$.
\end{corollary}

\begin{proof}
  The claim is trivial if $x^{*}=-\infty$. Otherwise, the first claim follows from the fact that $p_{\mu}$ and $p_{\nu}$ are continuous, and then the second claim follows
  from Lemma~\ref{le:barrier}.
\end{proof}

\begin{lemma}\label{le:martRestrict}
  Let $x\in\overline{\R}$ be a barrier. The following are equivalent:
  \begin{enumerate}
    \item $E^{P}[X\1_{X<x}] = E^{P}[Y\1_{X<x}]$ for some (and then all) $P\in\cS(\mu,\nu)$.
    \item $P|_{\{X<x\}}$ is a martingale transport for some (and then all) $P\in\cS(\mu,\nu)$.
    \item[(ii')] $P|_{\{X\leq x\}}$ is a martingale transport for some (and then all) $P\in\cS(\mu,\nu)$.
  \end{enumerate}
\end{lemma}

\begin{proof}
  If (i) holds for some $P\in\cS(\mu,\nu)$, then (ii) holds for the same $P$ since a supermartingale with constant mean is a martingale, and the converse holds as any martingale has constant mean. We complete the equivalence of~(i) and~(ii) by showing that if (i) holds for one $P\in\cS(\mu,\nu)$, it necessarily holds for all elements of $\cS(\mu,\nu)$. The cases $x=\pm\infty$ are clear, so let $x\in\R$.
  Let $P\in\cS(\mu,\nu)$ and let $\nu'$ be the second marginal of $P':=P|_{\{X<x\}}$. As~$x$ is a barrier, we have $\nu'=\nu$ on $(-\infty,x)$. If $\bar P\in\cS(\mu,\nu)$ is arbitrary and $\bar P',\bar \nu '$  are defined analogously, we have 
 $\bar \nu'=\nu=\nu'$ on $(-\infty,x)$ by the same reasoning.
  But then also $\nu'(\{x\})=\bar\nu'(\{x\})$, since this is the remaining mass transported from $(-\infty,x)$:
  we have $
    \nu'(\{x\}) = \mu(-\infty,x) - \nu'(-\infty,x) = \mu(-\infty,x) - \bar\nu'(-\infty,x) = \bar\nu'(\{x\}).
  $
  As a result, $\bar \nu'=\nu'$ on $(-\infty,x]$, and $\bar P'$ satisfies (i) whenever~$P$ does. Finally, (ii) implies (ii') because $x$ is a barrier, and the reverse is clear.
\end{proof}

\begin{lemma}\label{le:thresholdReverse}
  Let $x\in\overline{\R}$ be a barrier such that $P|_{\{X<x\}}$ is a martingale transport for some $P\in\cS(\mu,\nu)$. Then $p_{\mu}(x)=p_{\nu}(x)$.
\end{lemma}

\begin{proof}
  The cases $x=\pm\infty$ are clear, so let $x\in\R$. The martingale property yields that
  $
    p_{\mu}(x) = E[(x-X)\1_{X<x}] = E[(x-Y)\1_{X<x}].
  $
  Since $Y\leq x$ $P$-a.s.\ on $\{X<x\}$ and $\{Y<x\}\subseteq \{X< x\}$ $P$-a.s.,
  $
    E[(x-Y)\1_{X<x}] = E[(x-Y)^{+}\1_{X<x}] \geq E[(x-Y)^{+}\1_{Y<x}] = p_{\nu}(x).
  $
  Thus, $p_{\mu}(x)\geq p_{\nu}(x)$. As the converse inequality is always true, we deduce that $p_{\mu}(x)= p_{\nu}(x)$.
\end{proof}

The following completes the proof of Proposition~\ref{pr:barriers}(ii). %

\begin{corollary}\label{co:thresholdReverse}
  Let $x\in[-\infty,x^{*}]$ be a barrier. Then $p_{\mu}(x)=p_{\nu}(x)$.
\end{corollary}

\begin{proof}
  We may assume that $x\in\R$ which entails that $x^{*}>-\infty$.
  Lemma~\ref{le:barrier}, Corollary~\ref{co:threshold} and Lemma~\ref{le:martRestrict} show that the restriction of any $P\in\cS(\mu,\nu)$ to $\{X<x^{*}\}$ is a martingale transport. As $x\leq x^{*}$, the same holds for the restriction to $\{X<x\}$, and now Lemma~\ref{le:thresholdReverse} applies.
\end{proof}

\begin{lemma}\label{le:noBarrier}
  If $\bar x\in (x^{*},\infty]$ is a barrier, then $\mu(\bar x,\infty)=\nu(\bar x,\infty)=0$. 
\end{lemma}

\begin{proof}
  The case $\bar x = \infty$ is clear. 
  Let $\bar x\in (x^{*},\infty)$ be a barrier and suppose for contradiction that $\mu(\bar x,\infty)>0$ or $\nu(\bar x,\infty)>0$.
  
  \emph{Case 1: $\nu(\bar x,\infty)>0$.}  We contradict the barrier property with an element of $\cS(\mu,\nu)$ transporting mass from $(-\infty,\bar x)$ to $(\bar x,\infty)$, and vice versa.

  Let $P\in\cS(\mu,\nu)$ be arbitrary and let $P=\mu\otimes \kappa$ be a disintegration such that for all $x<\bar x$, we have $\bary(\kappa(x))\leq x$ and $\kappa(x,dy)$ is concentrated on $(-\infty,\bar x]$ but not on $\{\bar x\}$; these choices are possible due to the barrier  and the supermartingale property.
  
  For each $x\in (-\infty,\bar x)$, let $\eps(x)\in[0,1]$ be the largest number such that
  $$
    \kappa'(x) :=  (1-\eps(x))\kappa(x)|_{(-\infty,\bar x)} + \tilde\eps(x) \nu|_{(\bar x,\infty)} + \kappa(x)|_{\{\bar x\}}
  $$  
  satisfies $\bary(\kappa'(x)) \leq x$; here $\tilde\eps(x)$ is the unique constant such that $\kappa'(x)$ is a probability measure. This defines a stochastic kernel with the properties
  $$
    \kappa'(x)\{\bar x\} = \kappa(x)\{\bar x\}\quad \mbox{for all }x,\quad \kappa'(x)[\bar x,\infty) > \kappa(x)[\bar x,\infty)\quad \mbox{if }\eps(x)>0.
  $$
  Moreover, $\eps>0$ on a set of positive $\mu$-measure, as otherwise $P|_{\{X<\bar x\}}$ is a martingale transport which would contradict $\bar x>x^{*}$ (Lemma~\ref{le:thresholdReverse}). Let $\nu_{2}$ be the restriction to $(\bar x,\infty)$ of the second marginal of 
	$
	   \mu|_{(-\infty,\bar x)} \otimes \kappa'.
	$
	By truncating the above function $\eps(\cdot)$ at some positive constant $\bar \eps$, we may assume that $\nu_{2}\leq \nu$ while retaining the other properties. Thus, we can define a measure $\mu_{2}\leq \mu$ by taking the preimage of $\nu_{2}$ under $P$ (obtained by disintegrating $P=\nu(dy)\otimes \hat\kappa(y,dx)$ and taking  $\mu_{2}$ to be the first marginal of $\nu_{2}(dy)\otimes \hat\kappa(y,dx)$). Moreover, let
	$\nu_{1}$ be the restriction to $(-\infty,\bar x)$ of the second marginal of 
	$
	   \mu|_{(-\infty,\bar x)} \otimes \kappa \; - \; \mu|_{(-\infty,\bar x)} \otimes \kappa'.
	$
  Then $c:=\nu_{1}(\R)=\mu_{2}(\R)$ and by construction,
	$$
	 \mu|_{(-\infty,\bar x)} \otimes \kappa' \;\;+\;\;(\mu|_{[\bar x,\infty)} - \mu_{2} ) \otimes \kappa \;\;+\;\; c^{-1}\mu_{2}\otimes \nu_{1}
	$$
	is an element of $\cS(\mu,\nu)$. Since $\mu(-\infty,\bar x)>0$ and $\nu(\bar x,\infty)>0$, it transports mass across $\bar x$, contradicting that $\bar x$ is a barrier.

	\emph{Case 2: $\mu(\bar x,\infty)>0$ and $\nu(\bar x,\infty)=0$.} Note that in this case, $\nu|_{[\bar x,\infty)}$ is concentrated at $\bar x$ and the entire mass $\mu(\bar x,\infty)>0$ is transported to that atom by any $P\in\cS(\mu,\nu)$, in addition to any mass coming from $(-\infty,\bar x]$.   We shall contradict the barrier property by constructing an element of $\cS(\mu,\nu)$ which transports mass from $(\bar x,\infty)$ to $(-\infty,\bar x)$; this will be balanced by moving appropriate mass from $(-\infty,\bar x)$ to $\{\bar x\}$.
	
	Let $P\in\cS(\mu,\nu)$ be arbitrary and let $\kappa$ be as above. For each $x\in (-\infty,\bar x)$, let $\eps(x)\in[0,1]$ be the largest number such that
  $$
    \kappa'(x) :=  (1-\eps(x))\kappa(x)|_{(-\infty,\bar x)} + \tilde\eps(x) \nu|_{\{\bar x\}} 
  $$  
  satisfies $\bary(\kappa'(x))\leq x$; again, $\tilde\eps(x)$ is the unique constant such that $\kappa'(x)$ is a probability measure. This defines a stochastic kernel with 
  $$
    \kappa'(x)\{\bar x\} \geq \kappa(x)\{\bar x\}\quad \mbox{for all }x,\quad \kappa'(x)\{\bar x\} > \kappa(x)\{\bar x\}\quad \mbox{if }\eps(x)>0,
  $$
  and again, $\eps>0$ on a set of positive $\mu$-measure. Let $\nu_{2}$ be the restriction to~$\{\bar x\}$ of the second marginal of 
	$
	   \mu|_{(-\infty,\bar x)} \otimes \kappa' \;-\; \mu|_{(-\infty,\bar x)} \otimes \kappa.
	$
	After truncating $\eps(\cdot)$ we again have $\nu_{2}\leq\nu$; recall that $P$ transports the mass $\mu(\bar x,\infty)>0$ to $\bar x$. Continuing the construction as above, the latter property shows that $\mu_{2}(\bar x,\infty)>0$, and the barrier property is again contradicted.
\end{proof}

\begin{proof}[Proof of Proposition~\ref{pr:barriers}]
   Proposition~\ref{pr:barriers} is now a consequence of Lemma~\ref{le:barrier}, Corollary~\ref{co:threshold}, Corollary~\ref{co:thresholdReverse} and Lemma~\ref{le:noBarrier}.
\end{proof}

\begin{proof}[Proof of Proposition~\ref{pr:decomp}(i)]
  According to Proposition~\ref{pr:barriers}, we face a pure martingale transport problem on $(-\infty,x^{*}]$; in particular, we may apply the decomposition result of \cite[Theorem~8.4]{BeiglbockJuillet.12} on this part of the state space to obtain $\nu_{k}$ and $P_{k}$ for $k\geq1$. Since $x^{*}$ is itself a barrier by Proposition~\ref{pr:barriers}, the only possible choice for $\nu_{0}$ is
  $
    \nu_{0}=\nu|_{(x^*,\infty)} + [\mu(x^*,\infty)-\nu(x^*,\infty)]\delta_{x^{*}},
  $
  and this measure satisfies $\mu_{0}\leq_{cd}\nu_{0}$.
\end{proof}

We proceed towards the proof of the second part of Proposition~\ref{pr:decomp}.

\begin{lemma}\label{le:polarSetsIrred}
  If $\mu\leq_{cd}\nu$ is irreducible, the $\Pi(\mu,\nu)$-polar sets and the $\cS(\mu,\nu)$-polar sets coincide.
\end{lemma}

\begin{proof}
  If $\mu$ and $\nu$ have the same barycenter, then $\cS(\mu,\nu)=\cM(\mu,\nu)$ and this is the result of \cite[Corollary~3.4]{BeiglbockNutzTouzi.15}. Thus, we may assume that $(\mu,\nu)$ is proper. By Proposition~\ref{pr:barriers}, the associated domain $(I,J)$ satisfies $I=(x^{*},\infty)$ for some $x^{*}\in[-\infty,\infty)$, while $J=I$ if $\nu(\{x^{*}\})=0$ (including the case $x^{*}=-\infty$)  and $J=[x^{*},\infty)$ if $\nu(\{x^{*}\})>0$. 
  
  Since $\cS(\mu,\nu)\subseteq \Pi(\mu,\nu)$, it suffices to show that for any $\pi\in\Pi(\mu,\nu)$ there exists $P\in\cS(\mu,\nu)$ such that $P\gg\pi$. Let us show more generally that 
  \begin{center}for any measure $\pi$ on $\R^{2}$ with marginals $\pi_{1}\leq \mu$ and $\pi_{2}\leq \nu$\\there exists $P\in\cS(\mu,\nu)$ such that $P\gg\pi$.
  \end{center} 
  While $\pi$ is necessarily supported by $I\times J$, we prove the claim under the additional condition that $\pi$ is concentrated on a compact rectangle $K\times L \subset I\times J$. This entails no loss of generality: a general $\pi$ may be decomposed into a sum $\pi=\sum_{n}\pi^{n}$ of measures satisfying this condition, and if $P^{n}$ are the corresponding supermartingales, $P=\sum_{n} 2^{-n} P^{n}$ satisfies the claim.
  
  The definition of $(I,J)$ implies that $\nu$ assigns positive mass to any neighborhood of the lower endpoint $x^{*}$ of $J$. More precisely, we can find a compact set $B\subset J$, located entirely to the left of $K\subset I$, such that $\nu(B)>0$. (If $\nu(\{x^{*}\})>0$ we can simply take $B=\{x^{*}\}$.) Consider a disintegration $\pi=\pi_{1}\otimes \kappa$ where $\kappa(x,dy)$ is concentrated on $L$ for all $x\in K$. We introduce another stochastic kernel $\kappa'$ of the form
  $$
    \kappa'(x,dy) = \frac{\kappa(x,dy) + \eps(x) \nu(dy)|_{B} }{ c(x)}.
  $$
  Here $c(x)\geq1$ is the normalizing constant such that $\kappa'(x,dy)$ is a stochastic kernel. Moreover, $\eps(x):=0$ for $x$ such that $\bary(\kappa (x))\leq x$, whereas for $x$ with $\bary(\kappa (x))> x$ we let $\eps(x)$ be the unique positive number such that $\bary(\kappa' (x))= x$---this number exists by the intermediate value theorem; note that $B$ is located to the left of $x\in K$. 
  By construction, 
  $$
    \pi':= \nu(B) \pi_{1}\otimes \kappa'
  $$
  is a supermartingale with $\pi'\gg \pi$ and its marginals satisfy $\pi'_{1}\leq \pi_{1}\leq\mu$ as well as $\pi'_{2}\leq \nu$; the latter is due to $\pi_{1}(\R)\leq\mu(\R)=1$ and
  $
  \kappa'(x) \leq \nu(B)^{-1}\, \nu|_{B} + \kappa(x)
  $
  and $\kappa(x)$ being concentrated on $B^{c}$. We also note that 
  \begin{equation}\label{eq:piPrimeConcentrated}
  \mbox{$\pi'$  is concentrated on a quadrant of the form $[k,\infty)^{2}$}
  \end{equation}
  with $[k,\infty)\subseteq J$; here $k\in\R$ is determined by the lower bound of the set $B$.
  We shall complete the proof by constructing $P\in\cS(\mu,\nu)$ such that $P\gg\pi'$. 
  
  (i) We first consider the case where $\nu(\{x^{*}\})=0$ and hence $I=J=(x^{*},\infty)$ and $k>x^{*}$. Since $p_{\nu}-p_{\mu}$ is continuous, strictly positive on $I$ and  
  $
    \lim_{t\to\infty} p_{\nu}(t)-p_{\mu}(t) =  \mu(\R)[\bary(\nu)-\bary(\mu)]>0,
  $
  we see that $p_{\nu}-p_{\mu}$ is uniformly bounded away from zero on $[k,\infty)$. On the other hand, $p_{\pi'_{2}}-p_{\pi'_{1}}$ is uniformly bounded on $[k,\infty)$ since
  $$
    \lim_{t\to\infty} p_{\pi'_{2}}(t)-p_{\pi'_{1}}(t) =  \pi'_{1}(\R)[\bary(\pi'_{1})-\bary(\pi'_{2})] < \infty.
  $$
  As a result, there exists $\epsilon >0$ such that 
  $
    p_{\mu} - \epsilon p_{\pi'_{1}} \leq p_{\nu} - \epsilon p_{\pi'_{2}}
  $
  on $[k,\infty)$, but then also on $\R$ because $p_{\pi'_{1}}=p_{\pi'_{2}}=0$ outside of $[k,\infty)$ due to~\eqref{eq:piPrimeConcentrated}. Noting that this inequality may also be stated as
  $
    p_{\mu - \epsilon \pi'_{1}} \leq  p_{\nu - \epsilon \pi'_{2}},
  $
  Proposition~\ref{pr:convexOrder} shows that there exists some $P'\in\cS(\mu - \epsilon \pi'_{1},\nu - \epsilon \pi'_{2})$, and we complete the proof by setting $P:=P'+ \epsilon\pi'_{1}(\R)^{-1} \pi'$.
  
  (ii) In the case $\nu(\{x^{*}\})>0$ we need to argue differently that there exists $\epsilon >0$ such that 
  $
    p_{\mu} - \epsilon p_{\pi'_{1}} \leq p_{\nu} - \epsilon p_{\pi'_{2}}
  $
  on $[k,\infty)$. By enlarging $[k,\infty)$, we may assume that $k=x^{*}$ is the left endpoint of $J$. As $\mu(I)=\mu(\R)=\nu(J)$, we have
  $
    \partial^{+}p_\mu(x^{*}) = \partial^{+}p_\mu(x^{*}) - \partial^{-}p_\mu(x^{*}) = \mu(\{x^{*}\}) =0
  $
  and similarly
  $$
    \partial^{+}p_{\pi'_{1}}(x^{*})=0,\quad \partial^{+}p_{\pi'_{2}}(x^{*})=\pi'_{2}(\{x^{*}\}), \quad \partial^{+}p_\nu(x^{*}) = \nu(\{x^{*}\})>0.
  $$  
  Since $\nu(\{x^{*}\}) \geq \pi'_{2}(\{x^{*}\})$, it then follows that
  $
   0\neq \partial^{+}(p_{\nu}-p_{\mu})(x^{*}) \geq \partial^{+}(p_{\pi'_{2}}-p_{\pi'_{1}})(x^{*}).
  $
  The existence of the desired $\epsilon >0$ follows and the rest of the argument is as in (i).
\end{proof}

\begin{proof}[Proof of Proposition~\ref{pr:decomp}(ii)]
  By the decomposition in Proposition~\ref{pr:decomp}(i) and Lemma~\ref{le:polarSetsIrred}, a Borel set $B\subseteq \R^{2}$ is $\cS(\mu,\nu)$-polar if and only if $B\cap (I_{k}\times J_{k})$ is $\Pi(\mu_{k},\nu_{k})$-polar for all $k\geq0$ and $B\cap \Delta$ is $P_{-1}$-null. It remains to apply the result of \cite[Proposition~2.1]{BeiglbockGoldsternMareschSchachermayer.09} for each $k\geq0$: a Borel set $B_{k}$ is $\Pi(\mu_{k},\nu_{k})$-polar if and only if $B_{k} \subseteq (N_{\mu_{k}}\times \R) \cup (\R \times N_{\nu_{k}})$ for nullsets $N_{\mu_{k}}$ and $N_{\nu_{k}}$.
\end{proof}

\section{Duality Theory}\label{se:duality}

In this section, we introduce and analyze a dual problem for supermartingale optimal transport. We shall prove that this problem admits an optimizer and that there is no duality gap.

\subsection{Integration on a Proper Irreducible Component}\label{se:generalizedIntegral}

We first introduce the notion of integrability that will be used for the dual elements. Let $\mu\leq_{cd}\nu$ be proper and irreducible with domain $(I,J)$, and let $\chi: J\to \R$ be a concave increasing function. Since $\chi^{+}$ has linear growth, $\mu(\chi)$ and $\nu(\chi)$ are well defined in $[-\infty,\infty)$. In what follows, we give a meaningful definition of the difference $\mu(\chi)-\nu(\chi)$ in cases where both terms are infinite. We write $\chi'$ for the \emph{left} derivative of $\chi$, with the convention that $\chi'(\infty) := \lim_{t \to \infty} \chi'(t) = \inf_{t\in I} \chi'(t)$, and $-\chi''$ for the second derivative measure of the convex function $-\chi$ on $I$. Finally, recall that $I=(x^{*},\infty)$. If $\nu$ has an atom at $x^{*}$, then $\chi$ may have a jump at $x^{*}$ and we denote its magnitude by 
$
  \Delta\chi(x^{*}) := \chi(x^{*}+) - \chi(x^{*}) \in \R_{+}.
$

\begin{lemma}\label{le:intConcaveWelldef}
  Let $\mu\leq_{cd}\nu$ be proper and irreducible with domain $(I,J)$, let $\chi: J\to \R$ be a concave increasing function, and let $P=\mu \otimes \kappa$ be an arbitrary element of $\cS(\mu,\nu)$. Then
  \begin{align*}
    (\mu-\nu)(\chi) 
    &:=
    \int_I \left[\chi(x) - \int_J \chi(y) \kappa(x,dy) \right]\, \mu(dx) \\\
    &= \chi'(\infty) [\bary(\mu)-\bary(\nu)]  + \! \int_I (p_\mu-p_\nu)\, d\chi'' +\Delta\chi(x^{*})\nu(\{x^{*}\}).
  \end{align*}
  In particular, the definition of $(\mu-\nu)(\chi)\in [0,\infty]$ does not depend on $P$.
\end{lemma}

The proof follows the lines of \cite[Lemma~4.1]{BeiglbockNutzTouzi.15} and is omitted.
Our next aim is to define expressions of the form $\mu(\varphi)+\nu(\psi)$ in a situation where the individual integrals are not necessarily finite. We continue to assume that $\mu\leq_{cd}\nu$ is proper and irreducible with domain $(I,J)$.

\begin{definition}\label{de:ModConcaveIntegral}
  Let $\varphi: I\to \overline \R$ and $\psi: J\to \overline \R$ be Borel functions. If there exists a concave increasing function $\chi: J\to \R$ such that $\varphi-\chi \in L^{1}(\mu)$ and  $\psi+\chi \in L^{1}(\nu)$, we say that $\chi$ is a \emph{moderator} for $(\varphi,\psi)$ and set
  $$
  \mu(\varphi)+\nu(\psi) := \mu(\varphi-\chi)+\nu(\psi+\chi) + (\mu-\nu)(\chi) \,\in (-\infty,\infty].
  $$
  (This value is independent of the choice of $\chi$ by an argument similar to~\cite[Remark~4.8]{BeiglbockNutzTouzi.15}.) 
  We denote by $L^{ci}(\mu,\nu)$ the space of all pairs $(\varphi,\psi)$ which admit a moderator $\chi$ such that $(\mu-\nu)(\chi)<\infty$.
\end{definition}

\subsection{Closedness on a Proper Irreducible Component}\label{se:closednessOnIrred}

In this section, we introduce the dual problem for a proper and irreducible pair $\mu\leq_{cd}\nu$ with domain $(I,J)$. It will be convenient to work with a nonnegative reward function $f$ and alleviate this restriction  later on (Remark~\ref{rk:lowerBoundDualityGlobal}).

\begin{definition}\label{de:dualDomainPointw}
  Let $f: I\times J \to [0,\infty]$. We denote by $\cD^{ci,pw}_{\mu,\nu}(f)$ the set of all Borel functions $(\varphi,\psi,h): \R\to\overline\R \times \overline\R \times \R_{+}$ with $(\varphi,\psi)\in L^{ci}(\mu,\nu)$ and 
  $$
    \varphi(x) + \psi(y) + h(x)(y-x) \geq f(x,y),\quad (x,y)\in I\times J.
  $$
\end{definition}

We emphasize that the above inequality is stated in the pointwise (``pw'') sense. 
The following is the key result of this section. We stress that its assertion fails if $(\mu,\nu)$ is not proper; cf.\ Section~\ref{se:counterexDuality} for a counterexample.

\begin{proposition}\label{pr:closednessIrred}
   Given $f,f_{n}: I\times J \to [0,\infty]$ such that $f_{n}\to f$ pointwise and $(\varphi_{n},\psi_{n},h_{n})\in \cD^{ci,pw}_{\mu,\nu}(f_{n})$ satisfying $\sup_{n} \mu(\varphi_{n})+\nu(\psi_{n})<\infty$, there exist
   $$
     (\varphi,\psi,h)\in \cD^{ci,pw}_{\mu,\nu}(f) \quad\mbox{such that}\quad \mu(\varphi)+\nu(\psi)\leq \liminf_{ n}\mu(\varphi_{n})+\nu(\psi_{n}).
   $$
\end{proposition}

For the course of the proof, we abbreviate 
$
  \cD^{ci}(f):=\cD^{ci,pw}_{\mu,\nu}(f).
$

\begin{lemma}\label{le:passageToChi}
  Let $(\varphi, \psi,h)\in \cD^{ci}(0)$. There exists a moderator $\chi: J\to \R$ for $(\varphi, \psi)$ such that $\chi\leq\varphi$ on $I$ and $-\chi \leq \psi$ on $J$. In particular, we have $(\mu-\nu)(\chi)\leq \mu(\varphi)+\nu(\psi)$.
\end{lemma}

\begin{proof}
  Let  $P = \mu \otimes \kappa$ be a disintegration of some $P\in\cS(\mu,\nu)$ and let $(\varphi, \psi,h)\in \cD^{ci}(0)$. With a careful application of Fubini's theorem, one can verify that
  \begin{equation}\label{eq:hFiniteIteratedIntegral}
     \!\iint h(x)(y-x) \,\kappa(x,dy)\,\mu(dx) =\!  \int  h(x)(\bary(\kappa(x))-x) \,\mu(dx) > -\infty.
  \end{equation}
  (This is quite different from the property that $h(X)(Y-X)\in L^{1}(P)$ which may fail.)
   A second application of Fubini's theorem then yields that 
  \begin{align}\label{eq:PintegralIsIteratedIntegral}
    P[\varphi&(X)+\psi(Y)+h(X)(Y-X)]\nonumber\\
    &=\mu(\varphi) + \nu(\psi)+ \iint h(x)(y-x) \,\kappa(x,dy)\,\mu(dx)\in\R.
  \end{align}
  The concave and increasing function 
  $$
    \chi(y) := \inf_{x\in I} \,[\varphi(x) + h(x)(y-x)],\quad y\in J
  $$
  satisfies $\chi\leq\varphi$ on $I$ and $-\chi \leq \psi$ on $J$,  
  and one can check that $\chi$ is finite-valued as a consequence of $(\varphi,\psi)\in L^{ci}(\mu,\nu)$.
  Set $\bar\varphi:=\varphi - \chi \geq 0$ and $\bar\psi:=\psi+\chi \geq 0$. By the first part of the proof, the iterated integral of 
  $\varphi(x) + \psi(y) + h(x)(y-x)$ with respect to $\kappa$ and $\mu$ 
  is finite. The function
  \begin{equation}\label{eq:exprForTermByTermInt}
    \bar\varphi(x) + \bar\psi(y) + [\chi(x)-\chi(y)] + h(x)(y-x)
  \end{equation}
  is identical to the former; therefore, the iterated integral of~\eqref{eq:exprForTermByTermInt} is again finite.  For fixed $x\in I$, all four terms in~\eqref{eq:exprForTermByTermInt} are bounded from below by linearly growing functions. It follows that for $\mu$-a.e.\ $x\in I$, the integral with respect to $\kappa(x,dy)$ can be computed term-by-term, which yields 
  $$
    \bar\varphi(x) + \int \bar\psi(y) \,\kappa(x,dy) + \int [\chi(x)- \chi(y)] \,\kappa(x,dy) + h(x)(\bary(\kappa(x))-x).
  $$
  The first three terms are nonnegative, and the last term is known to be $\mu$-integrable by the first part of the proof. Thus, we may again integrate term-by-term with respect to $\mu$. In conclusion, the iterated integral of~\eqref{eq:exprForTermByTermInt}, which was already determined to be finite, may also be computed term-by-term. In particular, we deduce that $\mu(\bar\varphi)<\infty$, $\nu(\bar\psi) < \infty$ and $(\mu-\nu)(\chi)<\infty$, 
  showing that $(\bar\varphi,\bar\psi)\in L^{ci}(\bar\mu,\bar\nu)$ with concave moderator $\chi$, and 
  $
    \mu(\varphi)+\nu(\psi) = \mu(\bar\varphi) + \nu(\bar\psi)+ (\mu - \nu)(\chi) \geq (\mu - \nu)(\chi)
  $
  as desired.
\end{proof}

Our last tool for the proof of Proposition~\ref{pr:closednessIrred} is a compactness principle for concave increasing  functions. We mention that the conclusion fails if the pair $\mu\leq_{cd}\nu$ is not proper (see also Section~\ref{se:counterexDuality}): a nontrivial difference between the barycenters is crucial to control the first derivatives.

\begin{proposition}\label{pr:concaveCompactness}
  Let $a=\bary(\mu)$ and let $\chi_{n}: J\to \R$ be concave increasing functions such that 
  $ \chi_{n}(a)=0$ and $\sup_{n\geq 1} (\mu-\nu)(\chi_{n})<\infty$.
  There exists a subsequence $\chi_{n_{k}}$ which converges pointwise on $J$ to a concave increasing function $\chi: J\to \R$ such that $(\mu-\nu)(\chi) \leq \liminf_{k} (\mu-\nu)(\chi_{n_{k}})$.
\end{proposition}

\begin{proof}
  By our assumption, $(\mu-\nu)(\chi_{n})$ is bounded uniformly in $n$. Since $\bary(\mu)>\bary(\nu)$, Lemma~\ref{le:intConcaveWelldef} yields a constant $C$ such that 
  $0\leq \chi'_{n}(\infty)\leq C$ and $0\leq \int_{I} (p_{\mu}-p_{\nu})\, d\chi''_{n} \leq C$, as well as $0\leq \Delta\chi_{n}(x^{*})\leq C$
  in the case where $\nu(\{x^{*}\})>0$.
  For a suitable subsequence $\chi_{n_{k}}$, we have
  \begin{align}
  \lim_{k} \chi'_{n_{k}}(\infty) &= \liminf_{n} \chi'_{n}(\infty),\label{eq:firstDerivLimit}\\
  \mbox{and similarly} \quad \lim_{k} \Delta\chi_{n_{k}}(x^{*}) &= \liminf_{n} \Delta\chi_{n}(x^{*}) \quad \mbox{if }\nu(\{x^{*}\})>0.\label{eq:zeroDerivLimit}
  \end{align}

  Without loss of generality we assume that $n_{k}=k$. Given $y_{0}\in I$, we recall from the proof of Lemma~\ref{le:polarSetsIrred} that $p_{\mu}-p_{\nu}$ is strictly negative and uniformly bounded away from zero on $[y_{0},\infty)\subseteq (x^{*},\infty)=I$, and deduce that 
  $
   0\leq -\chi''_{n}[y_{0},\infty) \leq C'
  $
  for a constant $C'$. Since the (left) derivative $\chi'_{n}$ is decreasing, it follows that 
  $
    \chi'_{n}(y)  = -\chi''_{n}[y,\infty) + \chi'_{n}(\infty) \leq C' + C$ for all $y\in [y_{0},\infty).
  $
  Thus, the Lipschitz constant of $\chi_{n}$ is bounded on compact subsets of $I$, uniformly in $n$. Recalling that $\chi_{n}(a)=0$, the Arzela--Ascoli theorem then yields a function $\chi: I\to\R$ such that $\chi_{n}\to\chi$ locally uniformly, after passing to a subsequence. Clearly $\chi$ is concave and increasing, and integration by parts shows that $-\chi''_{n}$ converges to the second derivative measure $-\chi''$ associated with $\chi$, in the sense of weak convergence relative to the compactly supported continuous functions on $I$. Approximating $p_{\mu}-p_{\nu}$ from above with compactly supported continuous functions $g_{n}$, we then see that
  $
    \int_{I} (p_{\mu}-p_{\nu})\, d\chi'' = \lim_{m} \lim_{n}\int_{I} g_{m}\, d\chi_{n}''  \leq \liminf_{n\to\infty} \int_{I} (p_{\mu}-p_{\nu})\, d\chi''_{n}.
  $
  Using also~\eqref{eq:firstDerivLimit}, \eqref{eq:zeroDerivLimit} and the representation in Lemma~\ref{le:intConcaveWelldef}, we obtain  $(\mu-\nu)(\chi) \leq  \liminf_{n\to\infty} (\mu-\nu)(\chi_{n})$ as desired.
\end{proof}

\begin{proof}[Proof of Proposition~\ref{pr:closednessIrred}]
 We may assume that $\liminf_{ n}\mu(\varphi_{n})+\nu(\psi_{n})=\limsup_{ n}\mu(\varphi_{n})+\nu(\psi_{n})$, by passing to a subsequence. 
 Since $(\varphi_{n},\psi_{n},h_{n})\in \cD^{ci}(f_{n})$ and $f_{n}\geq0$, we can introduce the associated moderators $\chi_{n}$ as in Lemma~\ref{le:passageToChi}. We may assume that $\chi_{n}(a)=0$, where $a:=\bary(\mu)\in I$, by translating $\varphi_{n}$ and $\psi_{n}$ appropriately. After passing to a subsequence, Proposition~\ref{pr:concaveCompactness} then yields a pointwise limit $\chi: J\to\R$. Now, $(\varphi,\psi,h)\in \cD^{ci}(f)$ can be constructed using Komlos' lemma and concave-increasing envelopes, following the ideas in  the proof of \cite[Proposition~5.2]{BeiglbockNutzTouzi.15}.
\end{proof}

\subsection{Duality on a Proper Irreducible Component}\label{se:dualityOnIrred}

Recall that the pair $\mu\leq_{cd}\nu$ is proper and irreducible. Next, we define the primal and dual values.

\begin{definition}\label{de:primalAndDualOnIrred}
  Let $f: \R^{2}\to [0,\infty]$ and write $P(f)$ for the outer integral. The \emph{primal} and \emph{dual} problems are respectively given by 
  \begin{align*}
    \bS_{\mu,\nu}(f) :=\sup_{P\in\cS(\mu,\nu)} P(f),\quad %
    \bI^{pw}_{\mu,\nu}(f) :=\inf_{(\varphi,\psi,h)\in \cD^{ci,pw}_{\mu,\nu}(f)}  \mu(\varphi) + \nu(\psi).%
  \end{align*}
\end{definition}

A function $f: \R^{2}\to [0,\infty]$ is \emph{upper semianalytic} if the sets $\{f\geq c\}$ are analytic for all $c\in\R$, where a subset of $\R^{2}$ is called analytic if it is the (forward) image of a Borel subset of a Polish space under a Borel mapping. Any Borel function is upper semianalytic; we refer to \cite{BertsekasShreve.78} for further background.

\begin{proposition}\label{pr:dualityIrred}
  Let $\mu\leq_{cd}\nu$ be proper and irreducible, $f: \R^{2}\to [0,\infty]$.
  \begin{enumerate}
  \item If $f$ is upper semianalytic, then $\bS_{\mu,\nu}(f)=\bI^{pw}_{\mu,\nu}(f) \in [0,\infty]$.
  \item If\, $\bI^{pw}_{\mu,\nu}(f)<\infty$, there exists a dual optimizer $(\varphi,\psi,h)\in \cD^{ci,pw}_{\mu,\nu}(f)$.
  \end{enumerate}
\end{proposition} 

\begin{proof}
   The inequality ``$\leq$'' in (i) follows from~\eqref{eq:hFiniteIteratedIntegral} and~\eqref{eq:PintegralIsIteratedIntegral}. The converse inequality as well as (ii) follow from Proposition~\ref{pr:closednessIrred}, using the Hahn--Banach and Choquet theorems along the lines of \cite[Theorem~6.2]{BeiglbockNutzTouzi.15}.
\end{proof}

\subsection{Global Duality}\label{se:globalDuality}

In this section, we formulate a global duality result. We shall be brief since it is little more than the combination of the preceding results for the proper irreducible case and the known martingale case; however, it requires some notation.

Let $\mu\leq_{cd}\nu$ be probability measures and let $f: \R^{2}\to [0,\infty]$ be Borel. As in the irreducible case, the primal problem is
$
    \bS_{\mu,\nu}(f) :=\sup_{P\in\cS(\mu,\nu)} P(f).
$
For the dual problem, we first recall from Proposition~\ref{pr:decomp} the decompositions $\mu=\sum_{k\geq-1} \mu_{k}$ and $\nu=\sum_{k\geq -1} \nu_{k}$, where $\mu_{k}\leq_{cd}\nu_{k}$ is irreducible with domain $(I_{k},J_{k})$ for $k\geq0$ and $\mu_{-1}=\nu_{-1}$; moreover, $P_{-1}$ is the unique element of $\cS(\mu_{-1},\mu_{-1})$. So far, we have focused on a proper pair $(\mu_{0},\nu_{0})$ and its dual problem. The pairs $(\mu_{k},\nu_{k})$ for $k\geq1$ are in convex order ($\mu_{k}$ and $\nu_{k}$ have the same barycenter) and the corresponding martingale optimal transport has an analogous duality theory. While the arguments are  different, the preceding results hold true if ``convex-increasing'' is replaced by ``convex'' and the function $h$ is allowed to take values in $\R$ instead of $\R_{+}$; we refer to \cite{BeiglbockNutzTouzi.15} for the proofs. The spaces corresponding to $L^{ci}(\mu,\nu)$ and $\cD^{ci}_{\mu,\nu}(f)$ are denoted $L^{c}(\mu,\nu)$ and $\cD^{c}_{\mu,\nu}(f)$, respectively.

Let $(\varphi,\psi,h): \R\to\overline\R \times \overline\R \times \R$ be Borel. Since $P_{-1}$ is concentrated on the diagonal $\Delta$, the  dual problem associated to $(\mu_{-1},\nu_{-1})$ is trivially solved, for instance, by setting $\varphi(x)=f(x,x)$ and $\psi=h=0$. To simplify the notation below, we set 
$
  L^{c}_{\mu_{-1},\nu_{-1}} := \{(\varphi,\psi): \, \varphi+\psi \in L^{1}(\mu_{-1}) \} 
$
and $\mu_{-1}(\varphi)+\nu_{-1}(\psi) := \mu_{-1}(\varphi+\psi)$ for $(\varphi,\psi)\in L^{c}_{\mu_{-1},\nu_{-1}}$. Moreover,  $\cD^{c,pw}_{\mu_{-1},\nu_{-1}}(f)$ is the set of all $(\varphi,\psi,h)$ with $(\varphi,\psi)\in L^{c}_{\mu_{-1},\nu_{-1}}$ and 
$
  \varphi(x)+\psi(x)\geq f(x,x)$ for all $x\in I_{-1}.
$
Finally, we set 
$
  \bS_{\mu_{-1},\nu_{-1}}(f) := P_{-1}(f) \equiv \mu_{-1}(f(X,X)).
$

We can now introduce the domain for the global dual problem which will be stated in the quasi-sure sense. A property is said to hold $\cS(\mu,\nu)$-quasi surely, or $\cS(\mu,\nu)$-q.s.\ for short, if it holds $P$-a.s.\ for all $P\in\cS(\mu,\nu)$.%

\begin{definition}\label{de:globalIntegrability}
  Let $L(\mu,\nu)$ be the set of all Borel functions $\varphi,\psi: \R\to\overline\R$ such that $(\varphi,\psi)\in L^{ci}(\mu_{0},\nu_{0})$ and $(\varphi,\psi)\in L^{c}(\mu_{k},\nu_{k})$ for all $k\neq0$ and 
  $
    \sum_{k\geq-1} |\mu_{k}(\varphi)+\nu_{k}(\psi)| <\infty.
  $
  For $(\varphi,\psi)\in L(\mu,\nu)$, we define
  $$
    \mu(\varphi)+\nu(\psi) := \sum_{k\geq-1} \mu_{k}(\varphi)+\nu_{k}(\psi) <\infty,
  $$
  and $\cD_{\mu,\nu}(f)$ is the set of all Borel functions $(\varphi,\psi,h): \R\to\overline\R \times \overline\R \times \R$   
  such that $(\varphi,\psi)\in L(\mu,\nu)$, $h=0$ on $I_{-1}$, $h\geq0$ on $I_{0}$ and 
  $$
    \varphi(X)+\psi(Y) + h(X)(Y-X)\geq f(X,Y)\quad \cS(\mu,\nu)\qs
  $$
  Finally, 
  $
  \bI_{\mu,\nu}(f) :=\inf_{(\varphi,\psi,h)\in \cD_{\mu,\nu}(f)}  \mu(\varphi) + \nu(\psi) \;\in [0,\infty].
  $
\end{definition}

We emphasize that $h$ is required to be nonnegative on $I_{0}$ but can take arbitrary real values outside of $I_{0}$ and $I_{-1}$. It is shown in Section~\ref{se:counterexDuality} that nonnegativity cannot be enforced everywhere.

Before making precise the correspondence between this quasi-sure formulation and the individual components, recall that the intervals $J_{k}$ may overlap at their endpoints, so we have to avoid counting certain things twice. Indeed, let $(\varphi_{k},\psi_{k},h_{k})\in \cD^{ci,pw}_{\mu_{k},\nu_{k}}(f)$ for $k=0$ and $(\varphi_{k},\psi_{k},h_{k})\in \cD^{c,pw}_{\mu_{k},\nu_{k}}(f)$ for $k\geq1$; we claim that $\psi_{k}$ can be normalized such that
\begin{equation}\label{eq:psiNormalization}
    \psi_{k}=0 \quad \mbox{on}\quad J_{k} \setminus I_{k}.
\end{equation}
Indeed, if $J_{k}$ contains one of its endpoints, it is an atom of~$\nu$ and hence~$\psi_{k}$ is finite on~$J_{k}\setminus I_{k}$. If $k\geq1$, we can translate $\psi_{k}$ by an affine function and translate $\varphi_{k}$ and $h_{k}$ accordingly. In the supermartingale case $k=0$, we recall from Proposition~\ref{pr:decomp} that $I_{0}=(x^{*},\infty)$, so that $J_{0}$ can have at most one endpoint. As a result, we may obtain the normalization by translating $\psi_{0}$ with a constant, which can be compensated by translating $\varphi_{0}$ alone and thus respecting the requirement that $h_{0}\geq0$ on $I_{0}$.

\begin{lemma}\label{le:reductionToComponentsDual}
  Let $f: \R^{2}\to [0,\infty]$ be Borel, let $\mu\leq_{cd}\nu$ and let $\mu_{k},\nu_{k}$ be as in Proposition~\ref{pr:decomp}.

  (i) Let $(\varphi_{0},\psi_{0},h_{0})\in \cD^{ci,pw}_{\mu_{0},\nu_{0}}(f)$ and  $(\varphi_{k},\psi_{k},h_{k})\in \cD^{c,pw}_{\mu_{k},\nu_{k}}(f)$ for \mbox{$k\geq 1$} be normalized as in~\eqref{eq:psiNormalization}, and let $\varphi_{-1}(x)=f(x,x)$ and $\psi_{-1}=0$. If $\sum_{k\geq-1} \mu(\varphi_{k})+ \nu(\psi_{k})<\infty$, then
  $$
    \varphi:=\sum_{k\geq-1} \varphi_{k} \1_{I_{k}} ,\quad \psi:=\sum_{k\geq0} \psi_{k} \1_{J_{k}}, \quad h:=\sum_{k\geq0} h_{k} \1_{I_{k}}
  $$
  satisfies $(\varphi,\psi,h)\in \cD_{\mu,\nu}(f)$ and 
  $
    \mu(\varphi)+ \nu(\psi) = \sum_{k\geq-1} \mu_{k}(\varphi_{k})+ \nu_{k}(\psi_{k}).
  $

  (ii) Conversely, let $(\varphi,\psi,h)\in \cD_{\mu,\nu}(f)$. After changing $\varphi$ on a $\mu$-nullset and $\psi$ on a $\nu$-nullset, we have $(\varphi,\psi,h)\in \cD^{ci,pw}_{\mu_{0},\nu_{0}}(f)$ and $(\varphi,\psi,h)\in \cD^{c,pw}_{\mu_{k},\nu_{k}}(f)$ for $k\neq 0$, and 
  $
    \sum_{k\geq-1} \mu_{k}(\varphi)+ \nu_{k}(\psi) = \mu(\varphi)+ \nu(\psi) <\infty.
  $
\end{lemma}

This follows from Proposition~\ref{pr:decomp}; the details of the proof are analogous to  \cite[Lemma~7.2]{BeiglbockNutzTouzi.15}.
We can now state the global duality result.

\begin{theorem}\label{th:dualityGlobal}
  Let $f: \R^{2}\to [0,\infty]$ be Borel and let $\mu\leq_{cd}\nu$. Then
  $$
    \bS_{\mu,\nu}(f) = \bI_{\mu,\nu}(f) \in [0,\infty].
  $$
  If\, $\bI_{\mu,\nu}(f)<\infty$, there exists an optimizer $(\varphi,\psi,h)\in \cD_{\mu,\nu}(f)$ for $\bI_{\mu,\nu}(f)$.
\end{theorem}

This is a consequence of Proposition~\ref{pr:dualityIrred} and the corresponding result in the martingale case; the arguments are as in \cite[Theorem~7.4]{BeiglbockNutzTouzi.15}.

\begin{remark}\label{rk:lowerBoundDualityGlobal}
  The lower bound on $f$ in Theorem~\ref{th:dualityGlobal} can easily be relaxed. Indeed, let $f: \R^{2}\to \overline\R$ be Borel and suppose that there exist $a\in L^{1}(\mu)$, $b\in L^{1}(\nu)$ such that 
  $f(x,y)\geq a(x) + b(y)$ for all $x,y\in\R$.
  Then, we may apply Theorem~\ref{th:dualityGlobal} to
  $
    \bar f := [f(X,Y) - a(X)-b(Y)]^{+}
  $
  and deduce the duality result for $f$ as well.
\end{remark}

\section{Monotonicity Principle}\label{se:monotonicityPrinciple}

An important consequence of the duality theorem is a monotonicity principle describing the support of optimal transports; it can be seen as a substitute for the cyclical monotonicity from classical transport theory. The following notion will be useful for our study of the canonical couplings.

\begin{definition}\label{de:competitor}
  Let $\pi$ be a finite measure on $\R^{2}$ with finite first moment and let $M_{0},M_{1}\subseteq \R$ be Borel. Denote by $\pi_{1}$ its first marginal and by $\pi=\pi_{1}\otimes \kappa$ a disintegration. A measure $\pi'$ is an \emph{$(M_{0},M_{1})$-competitor} of $\pi$ if it has the same marginals and if its disintegration $\pi'=\pi_{1}\otimes \kappa'$ satisfies
  \begin{align*}
    \bary(\kappa'(x)) &\leq \bary(\kappa(x))\quad \mbox{for $\pi_{1}$-a.e.\ $x\in M_{0}$,}\\
    \bary(\kappa'(x)) &= \bary(\kappa(x))\quad \mbox{for $\pi_{1}$-a.e.\ $x\in M_{1}$.}
  \end{align*}
\end{definition}

This definition extends a concept of \cite{BeiglbockJuillet.12} where the barycenters are required to be equal on the whole real line. In our context, we need to distinguish three regimes for the applications in the subsequent sections: equality, inequality, and no constraint on the barycenters. One consequence of the inequality is that the notion of competitors is no longer symmetric.

Given $\mu\leq_{cd}\nu$, we recall from Proposition~\ref{pr:decomp} the sets $I_{k},J_{k}$, where the labels $k\geq1$ correspond to the martingale components, $k=0$ is the supermartingale component, and $k=-1$ is the complement (where any transport from $\mu$ to $\nu$ is the identity). Moreover, any element of $\cS(\mu,\nu)$ is necessarily supported by the set 
\begin{equation}\label{eq:defSigma}
    \Sigma:=\Delta \cup \bigcup_{k\geq0} I_{k}\times J_{k}.
\end{equation}

\begin{theorem}[Monotonicity Principle]\label{th:monotonicityPrinciple}
  Let $f: \R^{2}\to [0,\infty]$ be Borel, let $\mu\leq_{cd}\nu$ be probability measures and suppose that $\bS_{\mu,\nu}(f)<\infty$. There exist Borel sets $\Gamma\subseteq \R^{2}$ and $M\subseteq \R$ with the following properties.
  \begin{enumerate}
  
  \item A measure $P\in\cS(\mu,\nu)$ is optimal for $\bS_{\mu,\nu}(f)$ if and only if it is concentrated on $\Gamma$ and $P|_{M\times\R}$ is a martingale.
  
  \item Let $\bar\mu\leq_{cd}\bar\nu$ be  probabilities on $\R$. If $\bar P\in\cS(\bar\mu,\bar\nu)$ is concentrated on $\Gamma$  and $\bar P|_{M\times\R}$ is a martingale, then $\bar P$ is optimal for $\bS_{\bar\mu, \bar\nu}(f)$. %
  
  \item Let $M_{0}=M\cap I_{0}$ and $M_{1}=M\setminus M_{0}$, and let 
  $\pi$ be a finitely supported probability on $\R^{2}$ which is concentrated on~$\Gamma$. Then $\pi(f)\geq \pi'(f)$ for any $(M_{0},M_{1})$-competitor $\pi'$ of $\pi$ that is concentrated on $\Sigma$.
  \end{enumerate}
  If $(\varphi,\psi,h)\in \cD_{\mu,\nu}(f)$ is a suitable\footnote{chosen as in Lemma~\ref{le:reductionToComponentsDual}\,(ii)} version of the optimizer from Theorem~\ref{th:dualityGlobal}, then we can take 
  \begin{align*}
    M  &:= (I_{0}\cap\{h>0\}) \cup (\cup_{k\neq0} I_{k}),\\
   \Gamma &:=\big\{(x,y)\in \R^{2}:\, \varphi(x)+\psi(y)+h(x)(y-x) = f(x,y)\big\} \cap \Sigma.
  \end{align*}
  Moreover, the assertion in (iii) remains true if $\pi$ is not finitely supported, as long as $(\varphi,\psi)\in L(\pi_{1},\pi_{2})$, where $\pi_{1}$ and $\pi_{2}$ are the marginals of $\pi$.
\end{theorem}

Before proving the theorem, let us draw a corollary stating that the supermartingale optimal transport can be decomposed as follows. On $M$, an optimizer $P\in\cS(\mu,\nu)$ is also an optimizer of a martingale optimal transport problem. Thus, we think of $M$ as the set where the supermartingale constraint is ``binding,'' and in fact it acts like the seemingly stronger martingale constraint (thus $M$ as in \emph{martingale}). Whereas on $N:=\R\setminus M$, the measure~$P$ is also an optimizer of a (Monge--Kantorovich) optimal transport problem with no constraint at all on the dynamics ($N$ as in \emph{no} constraint).

\begin{corollary}[Extremal Decomposition]\label{co:extremalDecomp}
  Let $f: \R^{2}\to [0,\infty]$ be Borel and let $\mu\leq_{cd}\nu$ be probability measures such that $\bS_{\mu,\nu}(f)<\infty$. There exists a Borel set $M\subseteq \R$ with the following property.
  
  Given an optimizer $P\in\cS(\mu,\nu)$ for $\bS_{\mu,\nu}(f)$, let $\mu_{M}=\mu|_{M}$ and let $\nu_{M}$ be the image\footnote{
  If $P=\mu\otimes\kappa$, the image of $\mu_{M}$ under $P$ is defined as the second marginal of $\mu_{M}\otimes\kappa$.
  }
   of $\mu_{M}$ under $P$. Moreover, let $\mu_{N}=\mu|_{\R\setminus M}$ and let $\nu_{N}$ be the image of $\mu_{N}$ under $P$. Then
   for the same reward function $f$,
  \begin{enumerate}
  \item $P|_{M\times\R}$ is an optimal \emph{martingale} transport from $\mu_{M}$ to $\nu_{M}$,
  \item $P|_{N\times\R}$ is an optimal \emph{Monge--Kantorovich} transport from $\mu_{N}$ to $\nu_{N}$.
  \end{enumerate}
\end{corollary}

A word of caution is in order: while the set $M$ is defined without reference to $P$, the second marginals $\nu_{M},\nu_{N}$ in the extremal  problems do depend on~$P$. In that sense, the decomposition is non-unique---which, however, is quite natural given that the optimizer $P$ is non-unique as well, for general $f$.

\begin{remark}\label{rk:lowerBoundDualityMonPrinciple}
  The lower bound on $f$ in Theorem~\ref{th:monotonicityPrinciple} and Corollary~\ref{co:extremalDecomp} can be relaxed as follows. Instead of $f$ being nonnegative, suppose that there exist real functions  $a\in L^{1}(\mu)$, $b\in L^{1}(\nu)$ such that 
  $f(x,y)\geq a(x) + b(y)$ for all $x,y\in\R$.
  Then, Theorem~\ref{th:monotonicityPrinciple}\,(i),\,(iii) as well as Corollary~\ref{co:extremalDecomp} hold as above, using Remark~\ref{rk:lowerBoundDualityGlobal} but otherwise the same proofs. Moreover, Theorem~\ref{th:monotonicityPrinciple}\,(ii) as well as the last statement in Theorem~\ref{th:monotonicityPrinciple} hold under the condition that $a,b$ are integrable for $\bar\mu,\bar\nu$ and $\pi_{1},\pi_{2}$, respectively.
\end{remark}

\begin{example}\label{ex:decompositionNoncanonicalCase}
  In the context of Corollary~\ref{co:extremalDecomp}, suppose that $\mu$ has no atoms and that $f$ is smooth, of linear growth, and satisfies the Spence--Mirrlees condition $f_{xy}>0$ and the martingale Spence--Mirrlees condition $f_{xyy}>0$ (this is \emph{not} one of the canonical cases studied later). Then an optimizer $P$ exists and the corollary implies that $P|_{M\times\R}$ is the Left-Curtain coupling~\cite{BeiglbockJuillet.12} between its marginals and $P|_{N\times\R}$ is the Hoeffding--Fr\'echet coupling~\cite[Section~3.1]{RachevRuschendorf.98a} between its marginals. In particular, writing $P=\mu\otimes \kappa$ and using the results of the indicated references, we immediately deduce the possible forms of the kernel:  at almost every $x$, $\kappa(x)$ is either deterministic (the Hoeffding--Fr\'echet kernel) or a martingale kernel concentrated at two points (the Left--Curtain kernel). 
  In particular, $\kappa(x)$ is never what might seem to be the typical case---a truly random process with downward drift.

  We mention that the coupling $P$ is nevertheless not canonical in the sense of the Introduction: even if uniqueness holds, the optimal coupling may change substantially if we replace $f$ by a different function satisfying the same Spence--Mirrlees conditions; cf.\ Section~\ref{se:noncanonicalCouplings} for a counterexample.
\end{example}

\begin{proof}[Proof of Corollary~\ref{co:extremalDecomp}]
  Let $M$ and $(\varphi,\psi,h)$ be as in Theorem~\ref{th:monotonicityPrinciple}, and note that $P(f)<\infty$.
  
  (i) We have $P_{M}:=P|_{M\times\R}\in\cM(\mu_{M},\nu_{M})$ by (i) of the theorem. Moreover, setting $P_{N}:=P|_{N\times\R}$, any $\bar P_{M}\in\cM(\mu_{M},\nu_{M})$ induces an element of $\cS(\mu,\nu)$ via $\bar P:=\bar P_{M} +P_{N}$. Thus, $\bar P_{M}(f)\leq P_{M}(f)$ by the optimality of~$P$.
  
  (ii) This part is less direct because elements of $\Pi(\mu_{N},\nu_{N})$ are not supermartingales in general; we shall invoke Theorem~\ref{th:monotonicityPrinciple}(iii) with $\pi:=P$. By~(i) of the theorem, $\pi$ is concentrated on $\Gamma$, and of course $(\varphi,\psi)\in L(\mu,\nu)=L(\pi_{1},\pi_{2})$. Moreover, 
  for any $\pi'_{N}\in \Pi(\mu_{N},\nu_{N})$, the measure $\pi'=P_{M}+\pi'_{N}$ is an $(M_{0},M_{1})$-competitor of $\pi=P_{M}+P_{N}$ which is concentrated on $\Sigma$ as $P_N$ is concentrated on $I_{0}\times J_{0}\subseteq \Sigma$; note that $N\subseteq I_{0}$. Now the extension of~(iii) at the end of the theorem yields $\pi(f)\geq \pi'(f)$ and hence $P_{N}(f)\geq \pi'_{N}(f)$.
\end{proof}

\begin{proof}[Proof of Theorem~\ref{th:monotonicityPrinciple}]
  As $\bI_{\mu,\nu}(f)=\bS_{\mu,\nu}(f)<\infty$, Theorem~\ref{th:dualityGlobal} yields a dual optimizer $(\varphi,\psi,h)\in \cD_{\mu,\nu}(f)$ and we can define $\Gamma$ and $M$ as stated. 
  
  (i) Let $P\in\cS(\mu,\nu)$ and let $P=\mu\otimes \kappa$ be a disintegration. Recalling~\eqref{eq:hFiniteIteratedIntegral} and~\eqref{eq:PintegralIsIteratedIntegral} and the analogous facts for the martingale case~\cite{BeiglbockNutzTouzi.15}, we have
  \begin{align*}
    P(f)
    \leq P[\varphi(X) + \psi(Y) + h(X)(Y-X)]
    \leq \mu(\varphi)+\nu(\psi).
  \end{align*}%
  Since $\bS_{\mu,\nu}(f) = \mu(\varphi)+\nu(\psi)$, $P$ is optimal if and only if both inequalities are equalities. As $P(f)<\infty$, the first inequality is an equality if and only if $P$ is concentrated on $\Gamma$. Moreover, the second inequality is an equality if and only if $\int (y-x) \,\kappa(x,dy)=0$ $\mu$-a.e.\ on $\{h>0\}$; note that the condition on $\kappa$ holds automatically on the martingale components $I_{k}$, $k\geq1$. In particular, this is equivalent to $P|_{M\times\R}$ being a martingale. 
  
  (ii) We choose a version of $(\varphi,\psi,h)\in \cD_{\mu,\nu}(f)$ as in Lemma~\ref{le:reductionToComponentsDual}\,(ii); moreover, we may assume that $\bar P(f)<\infty$. We need to show that $(\varphi,\psi,h)\in \cD_{\bar\mu,\bar\nu}(f)$; once this is established, optimality can be argued as in~(i) above.
  
  (a) On the one hand, we need to show that
  \begin{equation}\label{eq:monPrinciplePolarSets}
    \varphi(X) + \psi(Y) + h(X)(Y-X) \geq f(X,Y) \quad \cS(\bar\mu,\bar\nu)\qs
  \end{equation}   
  For this, it suffices to prove that the domains of the irreducible components of $\bar\mu\leq_{cd}\bar\nu$ are subsets of the ones of $\mu\leq_{cd}\nu$; i.e., that $p_{\mu}(x)=p_{\nu}(x)$ implies $p_{\bar\mu}(x)=p_{\bar\nu}(x)$, for any $x\in\R$. Indeed, let $p_{\mu}(x)=p_{\nu}(x)$. Since $\bar P$ is concentrated on $\Gamma \subseteq \Sigma$, we know that $Y\leq x$ $\bar P$-a.s.\ on $\{X\leq x\}$ and $Y\geq x$ $\bar P$-a.s.\ on $\{X\geq x\}$. Writing $E[\,\cdot\,]$ for the expectation under $\bar P$, it follows that
  $
    p_{\bar\nu}(x) = E[(x-Y)^{+}]=E[(x-Y)\1_{X\leq x}].
  $
  Note that $p_{\mu}(x)=p_{\nu}(x)$ implies $x\leq x^{*}$, cf.\ Proposition~\ref{pr:barriers}. Recalling that $(-\infty,x^{*})\subseteq M$, our assumption on $\bar P$ then yields that $\bar P|_{\{X<x\}}$ is a martingale. Thus,
  $
    E[(x-Y)\1_{X\leq x}] = E[(x-X)\1_{X\leq x}]=E[(x-X)^{+}]=p_{\bar\mu}(x)
  $
  and part~(a) is complete.
  
  (b) On the other hand, we need to show that $(\varphi,\psi)\in L(\bar\mu,\bar\nu)$. By reducing to the components, we may assume without loss of generality that~$(\bar\mu,\bar\nu)$ is irreducible with domain $(I,J)$. Moreover, the argument for the martingale case is contained in the proof of \cite[Corollary~7.8]{BeiglbockNutzTouzi.15}, so we shall assume that~$(\bar\mu,\bar\nu)$ is proper.
  Let
  $
    \chi(y) := \inf_{x\in I} \,[\varphi(x) + h(x)(y-x)].%
  $
  As $(\varphi,\psi,h)\in \cD^{ci,pw}_{\mu,\nu}(f)$, the arguments below~\eqref{eq:PintegralIsIteratedIntegral} yield that
  $\chi: J\to \R$ is concave and increasing, that $\bar\varphi:=\varphi - \chi \geq 0$ and $\bar\psi:=\psi+\chi \geq 0$, and that the expectation $\bar P[ \varphi(X) +  \psi(Y)  + h(X)(Y-X)]$ can be computed as the $\bar\mu(dx)$-integral of
  $$
    \bar\varphi(x) + \int\bar\psi(y)\,\kappa(x,dy)+ \left[\chi(x) - \int \chi(y) \kappa(x,dy) \right] + h(x)(\bary(\kappa(x))-x),
  $$
  where $\bar P = \bar\mu\otimes \kappa$ for some kernel $\kappa$ (not necessarily the same as in~(i)). By the assumption that $\bar P|_{M\times \R}$ is a martingale and $\R\setminus M \,=\, \{h\leq 0\} \cap I_{0} \,\subseteq \, \{h=0\}$, either $h(x)=0$ or $\bary(\kappa(x))=x$, for $\bar\mu$-a.e.\ $x\in\R$. Using also that $\bar P$ is concentrated on $\Gamma$, we deduce that
  $
    \bar P(f)=\bar P[ \varphi(X) +  \psi(Y)  + h(X)(Y-X)] 
     = \bar\mu(\bar\varphi) + \bar\nu(\bar\psi)+ (\bar\mu - \bar\nu)(\chi),
  $
  where the last step is justified by the nonnegativity of the integrands. As $\bar P(f)<\infty$, we conclude that the three (nonnegative) terms on the right-hand side are finite; that is, $(\varphi,\psi)\in L(\bar\mu,\bar\nu)$ with moderator $\chi$.
  
  (iii) Again, we may assume that $\pi(f)<\infty$. Let $\pi'$ be an $(M_{0},M_{1})$-competitor of $\pi$, let $\bar\mu,\bar\nu$ be the common first and second marginals of $\pi,\pi'$ and let $\pi = \bar\mu\otimes \kappa$,  $\pi' = \bar\mu\otimes \kappa'$. If $(\varphi,\psi)\in L(\bar\mu,\bar\nu)$, using $h\geq0$ on $M_{0}\subseteq I_{0}$ and $\R\setminus M \subseteq \{h=0\}$ and the definition of the competitor yields
$
    \pi(f) 
    =\bar\mu(\varphi)+\bar\nu(\psi) + \int_{M} h(x) (\bary(\kappa(x)) - x) \, \bar\mu(dx)
    \geq \bar\mu(\varphi)+\bar\nu(\psi) + \int_{M} h(x) (\bary(\kappa'(x))-x) \, \bar\mu(dx) \geq \pi'(f).
$  
 Of course, $(\varphi,\psi)\in L(\bar\mu,\bar\nu)$ holds in particular if $\pi$ is finitely supported.
\end{proof}

\section{Shadow Construction}\label{se:shadowConstruction}

In this section, we introduce the Increasing and Decreasing Supermartingale Transports via an order-theoretic construction. 
Let $\fM^{1}(\R)$ be the set of all finite measures on $(\R,\cB(\R))$ which have a finite first moment, endowed with the weak convergence induced by the continuous functions of linear growth. We shall mainly use the restriction of this topology to subsets of measures of equal mass, and then it is equivalent to the Kantorovich or 1-Wasserstein distance $W(\nu,\nu')=\sup_{f} (\nu-\nu')(f)$, where $f$ ranges over all 1-Lipschitz functions.

\begin{definition}\label{de:extendedOrder}
  Let $\mu,\nu\in\fM^{1}(\R)$. We say that $\mu,\nu$ are in \emph{positive-convex-decreasing} order, denoted $\mu\leq_{pcd}\nu$, if
  $\mu(\phi)\leq\nu(\phi)$ for all nonnegative, convex, decreasing functions $\phi:\R\to\R$.
\end{definition}

We note that $\mu\leq_{pcd}\nu$ necessarily satisfy $\mu(\R)\leq\nu(\R)$. In fact, the case of strict inequality is the one of interest: if $\mu(\R)=\nu(\R)$, then $\mu\leq_{pcd}\nu$ is equivalent to $\mu\leq_{cd}\nu$.

\begin{lemma}\label{le:ecdSandwich}
  Let $\mu,\nu\in\fM^{1}(\R)$ satisfy $\mu\leq_{pcd}\nu$. Then the set%
  \footnote{We think of the elements of $\casts{\mu}{\nu}$ as lying between $\mu$ and $\nu$, as the notation suggests. However, we caution the reader that $\mu,\nu\notin \casts{\mu}{\nu}$ in general.}
  $$
    \casts{\mu}{\nu} := \{\theta\in\fM^{1}(\R):\, \mu\leq_{cd}\theta\leq\nu\}
  $$
  is nonempty and contains a unique least element $\shadow{\mu}{\nu}$ for the convex-decreasing order:
  $\shadow{\mu}{\nu} \leq_{cd} \theta$ for all $\theta \in \casts{\mu}{\nu}$.
  The measure $\shadow{\mu}{\nu}$ is called the \emph{shadow} of $\mu$ in $\nu$.
\end{lemma}

\begin{proof}
  Without loss of generality, $\nu$ is a probability measure.
  
  (i) We first show that $\casts{\mu}{\nu}$ contains some element $\theta$. Let $\lambda$ be the Lebesgue measure on $\R$ and let $G_{\nu}$ be the quantile function of~$\nu$; that is, the left-continuous inverse of the c.d.f.\ of $\nu$. We define
  $$
    \theta := \lambda|_{[0,k]} \circ G_{\nu}^{-1}\quad\mbox{where}\quad k:=\mu(\R)\in[0,1]. 
  $$
  This implies that $\theta\in\fM^{1}(\R)$, that $\theta(\R)=k$, and that $\theta\leq \nu$. Intuitively speaking, $\theta$ is the ``left-most'' measure $\theta\leq \nu$ of mass $k$ on $\R$; in particular, if~$\nu$ admits a density $f_{\nu}$, the density of $\theta$ is $f_{\theta}=f_{\nu}\1_{(-\infty,G_{\nu}(k)]}$.
  
  Let $\phi$ be a convex, decreasing function; we need to show that $\mu(\phi)\leq \theta(\phi)$. To this end, we may assume that $\phi(G_{\nu}(k))=0$ by translating $\phi$, and then 
  $
    \mu(\phi) \leq \mu(\phi^{+}) \leq \nu(\phi^{+}) = \theta(\phi^{+})=\theta(\phi)
  $
  since $\phi^{+}=\phi$ on $G_{\nu}([0,k])$. As a result, $\theta\in\casts{\mu}{\nu}\neq\emptyset$.
  
  (ii) Next, we show that $\casts{\mu}{\nu}$ is directed; i.e., given $\theta_{i}\in \casts{\mu}{\nu},$ $i=1,2$ there exists $\theta\in \casts{\mu}{\nu}$ such that $\theta\leq_{cd}\theta_{i}$. Indeed, let $p:\R\to\R$ be defined as the convex hull of the minimum of $p_{\theta_{1}}$ and $p_{\theta_{2}}$. Then $p$ is convex, and $p$ is increasing like $p_{\theta_{i}}$. Since the asymptotic slope of the functions $p_{\theta_{i}}$ is given by $\theta_{i}(\R)=\mu(\R)$, the same is true for $p$, and finally, $p_{\theta_{i}}\geq p_{\mu}$ yields $p\geq p_{\mu}$. These facts imply that $p$ is the put function associated with a measure $\theta$ satisfying $\mu\leq_{cd}\theta\leq_{cd}\theta_{i}$. It remains to show that $\theta\leq \nu$, which is equivalent to $p_{\nu}-p$ being convex. 
  Indeed, the fact that $p_{\nu}-p_{\theta_{i}}$ is convex for $i=1,2$ implies this property; cf.\ the proof of \cite[Lemma~4.6]{BeiglbockJuillet.12} for a detailed argument.
  
  (iii) The set $\casts{\mu}{\nu}\subseteq \fM^{1}(\R)$ consists of measures with common total mass $\mu(\R)$; we show that it is compact. Indeed, closedness is readily established.
Moreover, any $\theta\in\casts{\mu}{\nu}$ satisfies $\theta\leq\nu$. By Prokhorov's theorem, this immediately yields tightness in the weak topology induced by bounded continuous functions, and then using $\int |x|\, \nu(dx)<\infty$ yields relative compactness. %
  
  (iv) It follows from (iii) that for any convex, decreasing function $\phi$ of linear growth, the continuous functional $\theta\mapsto\theta(\phi)$ has a nonempty compact set $\Theta_{\phi}\subseteq\casts{\mu}{\nu}$ of minimizers. The directedness of $\casts{\mu}{\nu}$ from (ii) implies that a finite intersection $\Theta_{\phi_{1}} \cap\cdots\cap \Theta_{\phi_{n}}$ is still nonempty, and then compactness shows that $\theta\mapsto\theta(\phi)$ has a common minimizer $\shadow{\mu}{\nu}$ for all $\phi$. Uniqueness of the minimizer holds since $\theta_{1}\leq_{cd}\theta_{2}$ and $\theta_{2}\leq_{cd}\theta_{1}$ imply $\theta_{1}=\theta_{2}$.
\end{proof}

\begin{lemma}\label{le:shadowOfDirac}
  Let $\mu,\nu\in\fM^{1}(\R)$ satisfy $\mu\leq_{pcd}\nu$ and suppose that $\mu$ is concentrated at a single point $x\in\R$. Then, 
  the shadow $\shadow{\mu}{\nu}$ is of the form 
  $$
    \shadow{\mu}{\nu}=\nu|_{(a,b)}+ k_{a}\delta_{a}+k_{b}\delta_{b}.
  $$
  Among all measures $\theta\leq\nu$ with mass $\mu(\R)$ of this form, $\shadow{\mu}{\nu}$ is determined by maximizing $\bary(\theta)$ subject to the constraint $\bary(\theta)\leq x$. Moreover, $\shadow{\mu}{\nu}$ has minimal variance among all measures $\theta\leq\nu$ with mass $\mu(\R)$ and $\bary(\theta)=\bary(\shadow{\mu}{\nu})$.
  Finally, $a$ and $b$ can be chosen such that $a\leq x\leq b$.
  
  The map $\nu\mapsto \shadow{\mu}{\nu}$ is continuous when restricted to a set of measures $\nu\in\fM^{1}(\R)$ of equal total mass satisfying $\mu\leq_{pcd}\nu$.
\end{lemma}

\begin{proof}
  We may assume that $\nu(\R)$=1. Then, $\mu=k\delta_{x}$ for some $k\in[0,1]$, and we may focus on $k\in(0,1)$. Consider the family
  $$
    \theta_{s} = \lambda|_{[s,s+k]} \circ G_{\nu}^{-1},\quad s\in [0,1-k].
  $$
  Similarly as in the proof of Lemma~\ref{le:ecdSandwich}, we have $\theta_{s}\leq \nu$ for all $s$, whereas $\mu=k\delta_{x}\leq_{cd}\theta_{s}$ if and only if $\bary(\theta_{s})\leq x$. 
  As $\mu\leq_{pcd}\nu$, this inequality holds true in particular for $s=0$. The function 
  \begin{equation}\label{eq:barycenterFormula}
    s\mapsto\bary(\theta_{s}) = \frac{1}{k}\int_{0}^{k} G_{\nu}(s+t)\,\lambda(dt) = \frac{1}{k}\int_{0}^{k} G_{\nu}(s+t+)\,\lambda(dt)
  \end{equation}
  is increasing and continuous: $G_\nu(s)$ and its right limit $G_\nu(s+)$ differ only on a Lebesgue nullset, the fist representation shows left-continuity and the second shows right-continuity. Thus, we may define $s^{*}$ as the largest value in $[0,1-k]$ for which $\bary(\theta_{s})\leq x$, and then
  $\theta^{*}:=\theta_{s^{*}}$ is in $\casts{\mu}{\nu}$. We claim that $\theta^{*}$ is the least element in $\casts{\mu}{\nu}$.
  
  To show this, let $(a,b)=(G_{\nu}(s^{*}),G_{\nu}(s^{*}+k))$; then $\theta^{*}|_{(a,b)}=\nu|_{(a,b)}$ and $\theta^{*}$ is concentrated on $[a,b]$. Now let $\theta\in\casts{\mu}{\nu}$ be arbitrary. As $\theta\leq\nu$, we see that $\theta-(\theta^{*}\wedge\theta)$ is concentrated on $(a,b)^{c}$, whereas $\theta^{*}-(\theta^{*}\wedge\theta)$ is concentrated on $[a,b]$. Moreover, we must have $\bary(\theta)\leq \bary(\theta^{*})$. Indeed, this is clear if $\bary(\theta^{*})=x$. If not, the definition of $s^{*}$ implies that $\nu(b,\infty)=0$ and then $\theta^{*}$ clearly has the largest barycenter among all measures $\theta\leq\nu$ with mass~$\mu(\R)$. Thus, Lemma~\ref{le:orderWhenOutOfInterval} below implies that $\theta^{*}\leq_{cd}\theta$ and as a result, $\theta^{*}$ is the least element in $\casts{\mu}{\nu}$; i.e., $\shadow{\mu}{\nu}=\theta^{*}$.  
  
  As $\bary(\theta^{*})\leq x$, it is clear that $a\leq x$. With the above choice of $b$, it may happen that $b < x$ . However, by the definition of $s^{*}$, this is possible only if $\theta^{*}(\{b\})=\nu(\{b\})$ and $\nu(b,\infty)=0$. In that case, we may redefine $b:=x$ without invalidating the other assertions of the lemma, and then we have $a\leq x\leq b$ as required.
  
  Finally, the continuity of $\nu\mapsto \shadow{\mu}{\nu}$ can be shown by using~\eqref{eq:barycenterFormula} 
  and the property $W(\nu,\nu') = \int_{0}^{1} |G_{\nu}(t)-G_{\nu'}(t)|\,\lambda(dt)$ of the 1-Wasserstein distance; we omit the details.
\end{proof}

\begin{lemma}\label{le:orderWhenOutOfInterval}
  Let $\mu,\nu\in\fM^{1}(\R)$ satisfy $\mu(\R)=\nu(\R)$ and $\bary(\mu)\geq\bary(\nu)$. If there exists an interval $I=(a,b)$ such that $\mu$ is concentrated on $\bar{I}:=[a,b]\cap\R$ and $\nu$ is concentrated on $I^{c}$, then $\mu\leq_{cd}\nu$. The same is true if there exists an interval~$I$ such that $\mu - (\mu\wedge\nu)$ is concentrated on~$\bar{I}$ and $\nu - (\mu\wedge\nu)$ is concentrated on~$I^{c}$.
\end{lemma}

\begin{proof}
  The first claim implies the second, so we may focus on the former. We need to show that $\mu(\phi)\leq\nu(\phi)$ for any convex decreasing function $\phi$. To this end, we may assume that the left endpoint $a$ of the interval is finite and strictly smaller than the right endpoint $b$, as otherwise we must have $\mu=\nu=0$; moreover, we may reduce to the case $\phi(a)=0$. If $b$ is finite as well, we define 
  $
    \psi(x) := \phi(x) - \frac{\phi(b)}{b-a} (x-a)$ for $x\in\R,
  $
  whereas $\psi:=\phi$ if $b=\infty$. Then $\psi\leq0$ on $\bar{I}$ and $\psi\geq0$ on $I^{c}$, which yields
  $
    \mu(\phi)\leq \mu(\psi^{+}) + \frac{\phi(b)}{b-a} [\bary(\mu)-a]\leq \nu(\psi^{+}) + \frac{\phi(b)}{b-a} [\bary(\nu)-a] = \nu(\phi)
  $
  as desired.
\end{proof}

The following result is important to apply the shadow in an iterative fashion. The first assertion intuitively follows from the minimality of the shadow: if we transport part of a measure $\mu\leq_{pcd}\nu$ to its shadow in $\nu$, the remaining part $\mu_{2}$ of $\mu$ is still dominated by the remaining part of $\nu$. Moreover, if we then transport $\mu_{2}$ to its shadow in the remainder, the cumulative result is the same as the shadow of $\mu$ in $\nu$. 

\begin{proposition}\label{pr:shadowAdditive}
  Let $\mu_{1},\mu_{2}, \nu\in\fM^{1}(\R)$ satisfy $\mu_{1}+\mu_{2}\leq_{pcd}\nu$. Then $\mu_{2}\leq_{pcd}\nu - \shadow{\mu_{1}}{\nu}$ and 
  $
    \shadow{\mu_{1}+\mu_{2}}{\nu} = \shadow{\mu_{1}}{\nu} + \shadow{\mu_{2}}{\nu-\shadow{\mu_{1}}{\nu}}.
  $
\end{proposition}

\begin{proof}
  Using the result of Lemma~\ref{le:shadowOfDirac}, one can first establish the claim when $\mu_{1}$ is a single atom. Then, one can extend to the general case along the lines of \cite[Theorem~4.8]{BeiglbockJuillet.12}; we omit further details.
\end{proof}

Next, we shall use the shadow mapping to construct specific supermartingale transports. Let $\mu\leq_{cd}\nu$ and suppose first that $\mu=\sum_{i=1}^{n} k_{i} \delta_{x_{i}}$ is finitely supported. We may transport~$\mu$ to~$\nu$ by first mapping $k_{1}\delta_{x_{1}}$ to its shadow in~$\nu$, continue by mapping $k_{2}\delta_{x_{2}}$ to its shadow in the ``remainder'' $\nu-\shadow{k_{1}\delta_{x_{1}}}{\nu}$ of $\nu$, and so on. Proceeding until $i=n$, this constructs the kernel $\kappa$ corresponding to a supermartingale transport $\mu\otimes\kappa\in\cS(\mu,\nu)$. In fact, this recipe leads to a whole family of transports---the labeling of the atoms was arbitrary, and a different order in their processing will typically give rise to a different transport. There are two choices that seem canonical: left-to-right (increasing) and right-to-left (decreasing). We shall show in the subsequent sections that the corresponding transports~$\rP$ and~$\lP$ are indeed canonical in several ways.

\begin{theorem}\label{th:canonicalCouplings}
  Let $\mu\leq_{cd}\nu$. 
  \begin{enumerate}
    \item There exists a unique measure $\rP$ on $\R\times\R$ which transports $\mu|_{(-\infty,x]}$ to its shadow $\shadow{\mu|_{(-\infty,x]}}{\nu}$ for all $x\in\R$; that is, the first marginal of~$\rP$ equals~$\mu$ and  
      $\rP((-\infty,x] \times A) = \shadow{\mu|_{(-\infty,x]}}{\nu}(A)$ for $A\in\cB(\R)$.

	  \item Similarly, there exists a unique measure $\lP$ on $\R\times\R$ which transports $\mu|_{[x,\infty)}$ to its shadow $\shadow{\mu|_{[x,\infty)}}{\nu}$ for all $x\in\R$.
	  \end{enumerate}
	  Moreover, those two measures are elements of $\cS(\mu,\nu)$. We call $\rP$ and $\lP$ the Increasing and the Decreasing  Supermartingale Transport, respectively.
\end{theorem}

\begin{proof}
  The function
  $
    F(x,y):=\shadow{\mu|_{(-\infty,x]}}{\nu}(-\infty,y]
  $
  is clearly increasing and right-continuous in $y$. Moreover, 
  Proposition~\ref{pr:shadowAdditive} implies that 
  $$
   \shadow{\mu|_{(-\infty,x_{2}]}}{\nu} - \shadow{\mu|_{(-\infty,x_{1}]}}{\nu} = \shadow{\mu|_{(x_{1},x_{2}]}}{\nu-\shadow{\mu|_{(-\infty,x_{1}]}}{\nu}}\geq0, \quad x_{1}\leq x_{2}
  $$
  which yields the same properties for the variable $x$; note that the total mass of the right-hand side equals $\mu(x_{1},x_{2}]$. Noting also that $F$ has the proper normalization for a c.d.f., we conclude that $F$ induces a unique measure $\rP$ on $\cB(\R\times\R)$. It is clear that $\mu$ is the first marginal of $\rP$. The second marginal is $S^{\nu}(\mu)\leq \nu$, and this is in fact an equality because both measures have the same mass. To conclude that $\rP\in\cS(\mu,\nu)$, it suffices to show that $\rP[Y\phi(X)] \leq \rP[X\phi(X)]$ for all $\phi=\1_{(x_{1},x_{2}]}$ with $x_{1}<x_{2}$. Indeed, Proposition~\ref{pr:shadowAdditive} implies that
  $
    \rP[Y\phi(X)] 
     = \int y \, [\shadow{\mu|_{(-\infty,x_{2}]}}{\nu}-\shadow{\mu|_{(-\infty,x_{1}]}}{\nu}](dy) 
     = \bary(\shadow{\mu|_{(x_{1},x_{2}]}}{\nu-\shadow{\mu|_{(-\infty,x_{1}]}}{\nu}}) 
     \leq \bary(\mu|_{(x_{1},x_{2}]}) 
    = \rP[X\phi(X)]$.
  The arguments for (ii) are analogous.
\end{proof}

A different construction of $\rP$ and $\lP$ could proceed through an approximation of the marginals by discrete measures, for which the couplings can be defined explicitly by iterating Lemma~\ref{le:shadowOfDirac}, and a subsequent passage to the limit. We refer to \cite[Remark~2.18]{Juillet.14} for a sketch of such a construction in the martingale case.

\section{Spence--Mirrlees Functions and Geometry of their Optimal Transports}\label{se:spenceMirrleesAndGeom}

In this section, we relate monotonicity properties of the reward function $f$ to the geometry of the supports of the corresponding optimal supermartingale transports, where the support will be described by a pair $(\Gamma,M)$ as in Theorem~\ref{th:monotonicityPrinciple}. We first introduce the relevant properties of $f$.

\begin{definition}\label{de:spenceMirrlees}
  A function $f : \R^2 \to \R$ is \emph{first-order Spence--Mirrlees} if
  $$
    f(x_{2},\cdot) - f(x_{1},\cdot)\quad\mbox{is strictly increasing for all } x_{1}<x_{2}.
  $$
  Moreover, $f$ is \emph{second-order Spence--Mirrlees} if
  $$
    f(x_{2},\cdot) - f(x_{1},\cdot)\quad\mbox{is strictly convex for all } x_{1}<x_{2},
  $$
  and $f$ is \emph{supermartingale Spence--Mirrlees} if $f$ is second-order Spence--Mirrlees and $-f$ is first-order Spence--Mirrlees.
\end{definition}  

We note that if $f$ is smooth, the first and second order Spence--Mirrlees properties are equivalent to the  classical cross-derivative conditions $f_{xy}>0$ and $f_{xyy}>0$, respectively. The latter is also called martingale Spence--Mirrlees condition in the literature on martingale optimal transport---the above terminology will be more convenient in what follows.

\begin{remark}\label{rk:smoothSMfunctions}
   There exist smooth, linearly growing supermartingale Spence--Mirrlees functions on $\R^{2}$.
   Indeed, let $\varphi$ be a smooth, bounded, strictly increasing function on $\R$; e.g., $\varphi(x)=\tanh(x)$. Let $\psi$ be a smooth,  linearly growing, strictly decreasing, strictly convex function on $\R$; e.g., $\psi(y)=(1+y^{2})^{1/2}-y$. Then,
  $
    g(x,y):=\varphi(x)\psi(y)
  $
  satisfies $g_{xy}<0$ and $g_{xyy}>0$, while $|g(x,y)| \leq C(1+|y|)$ for some $C>0$.
\end{remark}

Next, we introduce the relevant geometric properties of the support.

\begin{definition}\label{de:firstOrderMonotone}
 Let $(\Gamma,M)\subseteq \R^{2}\times\R$ and consider $(x_{1},y_{1}),(x_{2},y_{2})\in\Gamma$ with $x_{1}<x_{2}$. The pair $(\Gamma,M)$ is
 \begin{enumerate}
 \item  \emph{first-order left-monotone} if $y_{1}\leq y_{2}$ whenever $x_{2}\notin M$,
 \item  \emph{first-order right-monotone} if $y_{2}\leq y_{1}$ whenever $x_{1}\notin M$.
 \end{enumerate}
\end{definition}

We will also need the following properties of $\Gamma$; they are taken from~\cite{BeiglbockJuillet.12} where they are simply called left- and right-monotonicity.

\begin{definition}\label{de:secondOrderMonotone}
 Let $\Gamma\subseteq \R^{2}$ and consider $(x,y_{1}),(x,y_{2}),(x',y')\in\Gamma$ with $y_{1}<y_{2}$. Then $\Gamma$ is
 \begin{enumerate}
 \item  \emph{second-order left-monotone} if $y'\notin (y_{1},y_{2})$ whenever $x<x'$,
 \item  \emph{second-order right-monotone} if $y'\notin (y_{1},y_{2})$ whenever $x'<x$.
 \end{enumerate}
\end{definition}

For convenience, we shall use the same terminology for a pair $(\Gamma,M)$ even though only $\Gamma$ is relevant for the second-order properties.
Yet another notion will be useful; we write
$
  \Gamma^{1}=\{x\in\R:\, (x,y)\in\Gamma~\mbox{for some }y\in\R\}
$
for the projection of $\Gamma$ onto the first coordinate.

\begin{definition}\label{de:nondegenerate}
 A pair $(\Gamma,M)\subseteq \R^{2}\times\R$ is \emph{nondegenerate} if
 \begin{enumerate}
 \item for all $x\in\Gamma^{1}$ such that $(x,y)\in\Gamma$ for some $y>x$, there exists $y'<x$ such that $(x,y')\in\Gamma$,
 \item for all $x\in\Gamma^{1}\cap M$ such that $(x,y)\in\Gamma$ for some $y<x$, there exists $y'> x$ such that $(x,y')\in\Gamma$.
 \end{enumerate}
  These two conditions imply that
\begin{enumerate}
 \item[(i')] for all $x\in\Gamma^{1}$ there exists $y\leq x$ such that $(x,y)\in\Gamma$,
 \item[(ii')] for all $x\in\Gamma^{1}\cap M$ there exists $y\geq x$ such that $(x,y)\in\Gamma$.
 \end{enumerate}
\end{definition}

Essentially, nondegeneracy postulates that there is a down-path at every $x\in\Gamma^{1}$, and also an up-path if $x\in M$. Thus, it is a natural requirement if we intend to consider supermartingales supported by~$\Gamma$ which are martingales on $M\times\R$. For later use, let us record that nondegeneracy can be assumed without loss of generality in our context.

\begin{remark}\label{rk:nondegWlog}
  Let $(\Gamma,M)\in\cB(\R^{2})\times\cB(\R)$, let $\mu\leq_{cd}\nu$ be probability measures and suppose there is $P\in\cS(\mu,\nu)$ with $P(\Gamma)=1$ such that $P|_{M\times\R}$ is a martingale. Then, there exists a Borel subset $\Gamma'\subseteq \Gamma$ with $P(\Gamma')=1$ such that $(\Gamma',M)$ is nondegenerate.
\end{remark}

\begin{proof}
  Let $N'_{1}$ be the set of all $x\in\Gamma^{1}$ such that Definition~\ref{de:nondegenerate}\,(i) fails. Then $N'_{1}$ is universally measurable and thus we can find a Borel set  $N_{1}\supseteq N'_{1}$ such that $N_{1}\setminus N'_{1}$ is $\mu$-null. The fact that $P$ is a supermartingale implies that $\Gamma_{1}:=\Gamma \cap \{Y>X\}\cap (N_{1}\times\R)$ is $P$-null. After defining similarly a set $N_{2}$ for Definition~\ref{de:nondegenerate}\,(ii), the martingale property of $P$ on $M\times\R$ shows that $\Gamma_{2}:=\Gamma \cap \{Y<X\}\cap (N_{2}\times\R)$ is $P$-null as well, and then we can set $\Gamma':=\Gamma\setminus (\Gamma_{1}\cup\Gamma_{2})$.
\end{proof}

The first-order properties turn out to be highly asymmetric when combined with nondegeneracy. The following observation will have far-reaching consequences regarding the geometry of the coupling $\rP$ and has no analogue in the left-monotone case.

\begin{remark}\label{rk:leftClosed}
  Let $(\Gamma,M)$ be first-order right-monotone and nondegenerate. Then, $M$ is a half-line unbounded to the left within $\Gamma^{1}$; that is,
  $$
    \mbox{if $x_{1},x_{2}\in\Gamma^{1}$ satisfy $x_{1}<x_{2}$ and $x_{2}\in M$, then $x_{1}\in M$.}
  $$
  Indeed, let $x_{1},x_{2}$ be as stated; then nondegeneracy yields $y_{1},y_{2}$ such that $y_{1}\leq x_{1}<x_{2}\leq y_{2}$ and $(x_{i},y_{i})\in\Gamma$. If we had $x_{1}\notin M$, this would contradict first-order right-monotonicity.
\end{remark}

With these definitions in place, we can use the monotonicity principle of Theorem~\ref{th:monotonicityPrinciple} to infer the geometry of $(\Gamma,M)$ from the properties of $f$. 

\begin{proposition}\label{pr:spenceImpliesGammaMonotone}
  Let\footnote{In fact, this result merely uses the general shape of $\Sigma$, not the specific marginals.} $\mu\leq_{cd}\nu$ and recall the corresponding intervals $I_{k},J_{k}$ of Proposition~\ref{pr:decomp} and the set $\Sigma$ of~\eqref{eq:defSigma}. Let $(\Gamma,M)\in\cB(\R^{2})\times\cB(\R)$ be nondegenerate, where $\Gamma\subseteq \Sigma$ and $M=M_{0}\cup M_{1}$ with Borel sets $M_{0}\subseteq I_{0}$ and $M_{1}= \cup_{k\neq0} I_{k}$, and let $f : \R^2 \to \R$. Suppose that the assertion of Theorem~\ref{th:monotonicityPrinciple}\,(iii) holds; that is, if $\pi$ is a finitely supported probability which is concentrated on~$\Gamma$, then $\pi(f)\geq \pi'(f)$ for any $(M_{0},M_{1})$-competitor $\pi'$ of $\pi$ that is concentrated on $\Sigma$.
  \begin{enumerate}
  \item \hspace{-.3em}If $f$ is first-order Spence--Mirrlees, \!$(\Gamma,M)$\! is first-order left-monotone.
  \item \hspace{-.3em}If $-\!f$ is first-order Spence--Mirrlees, \!$(\Gamma,M)$\! is first-order right-monotone.
  \item \hspace{-.3em} If $f$ is second-order Spence--Mirrlees, \!$\Gamma$\! is second-order left-monotone.\hspace{-3em}
  \item \hspace{-.3em}If $-\!f$ is second-order Spence--Mirrlees, \!$\Gamma$\! is second-order right-monotone.
  \end{enumerate}
\end{proposition}

\begin{proof}
  (i) Consider $(x_{1},y_{1}),(x_{2},y_{2})\in\Gamma$ with $x_{1}<x_{2}$ and suppose for contradiction that $y_{2}<y_{1}$.
  The measures
  $$
    \pi := \tfrac{1}{2}\delta_{(x_1,y_1)} + \tfrac{1}{2}\delta_{(x_2,y_2)},
\quad \pi' := \tfrac{1}{2}\delta_{(x_1,y_2)} + \tfrac{1}{2}\delta_{(x_2,y_1)}
  $$
  have the same first marginal $\pi_{1}=\tfrac{1}{2}\delta_{x_1} + \tfrac{1}{2}\delta_{x_2}$. Let $\pi=\pi_{1}\otimes\kappa$ and $\pi'=\pi_{1}\otimes\kappa'$, then 
  $
    \bary(\kappa'(x_{1})) < \bary(\kappa(x_{1}))$ and $\bary(\kappa'(x_{2})) > \bary(\kappa(x_{2})).
  $
  Suppose that $x_{1}\notin M_{1}$ and $x_{2}\notin M$. Then, $\pi'$ is an $(M_{0},M_{1})$-competitor of $\pi$. Moreover, $x_{i}\notin M_{1}$ implies that $x_{i}\in I_{0}$ and thus $y_{i}\in J_{0}$, $i=1,2$ which shows that $\pi'$ is supported on $\Sigma$. Thus, we must have $\pi(f) \geq \pi'(f)$.  However,
  $
  2(\pi(f) - \pi'(f)) = (f(x_2,y_2) - f(x_1,y_2)) - (f(x_2,y_1) - f(x_1,y_1)) < 0
  $
  as $f$ is first-order Spence--Mirrlees, so we have reached the desired contradiction.
  
  Let $x_{1}\in M_{1}$ and $x_{2}\notin M$. Recalling that $M_{1}= \cup_{k\neq0} I_{k}=(-\infty,x^{*}]$, we have $y_{1}\in J_{k}$ for some $k\neq0$, whereas $x_{2}\notin M$ implies $y_{2}\in J_{0}$. Since $J_{0}$ is located to the right of $J_{k}$ for $k\neq 0$, we must have $y_{1}\leq y_{2}$.

  (ii) Consider $(x_{1},y_{1}),(x_{2},y_{2})\in\Gamma$ with $x_{1}<x_{2}$ and suppose for contradiction that $y_{1}<y_{2}$. We define $\pi,\pi'$ as in~(i); then
  $
    \bary(\kappa'(x_{1})) > \bary(\kappa(x_{1}))$ and $\bary(\kappa'(x_{2})) < \bary(\kappa(x_{2})).
  $
  Let $x_{1}\notin M$. Then, $x_{1}\in I_{0}=(x^{*},\infty)$ and thus $x_{2}>x_{1}$ is in $I_{0}$ as well. In particular, $x_{2}\notin M_{1}$ and $y_{1},y_{2}\in J_{0}$. Thus $\pi'$ is an $(M_{0},M_{1})$-competitor of $\pi$ that is concentrated on~$\Sigma$ and we reach a contradiction to $-f$ being first-order Spence--Mirrlees, similarly as in~(i).

  (iii) Let $(x,y_{1}),(x,y_{2}),(x',y')\in\Gamma$ satisfy $x<x'$ and assume for contradiction that $y_{1}<y'<y_{2}$.
  Define $\lambda = \frac{y' - y_1}{y_2 - y_1}$ and
	\begin{align*}
	\pi = \tfrac{\lambda}{2}\delta_{(x,y_1)} + \tfrac{1-\lambda}{2}\delta_{(x,y_2)} + \tfrac{1}{2}\delta_{(x',y')}, \quad
	\pi' = \tfrac{\lambda}{2}\delta_{(x',y_1)} + \tfrac{1-\lambda}{2}\delta_{(x',y_2)} + \tfrac{1}{2}\delta_{(x,y')}.
	\end{align*}
	Then, $\pi$ and $\pi'$ have the same first marginal $\pi_{1}$ and if $\pi=\pi_{1}\otimes\kappa$ and $\pi'=\pi_{1}\otimes\kappa'$, then $\kappa(x)$, $\kappa'(x)$, $\kappa(x')$, $\kappa'(x')$ all have barycenter $y'$. Hence, $\pi'$ is an $(M_{0},M_{1})$-competitor of $\pi$, and since the shape of $\Gamma\subseteq\Sigma$ shows that~$\pi'$ is concentrated on $\Sigma$, we deduce that $\pi(f)\geq \pi'(f)$. However, $f$ being second-order Spence--Mirrlees implies that $\pi(f)< \pi'(f)$.
	
	(iv) The proof is symmetric to (iii).
\end{proof}

\section{Geometric Characterization of the Canonical Supermartingale Transports}\label{se:geomCharact}

In this section, we consider fixed probability measures $\mu\leq_{cd}\nu$ and show that the associated Increasing and Decreasing Supermartingale Transports $\rP,\lP$ (cf.\ Theorem~\ref{th:canonicalCouplings}) are characterized by geometric properties of their supports.

\begin{theorem}\label{th:geomCharactRP}
  Let $(\Gamma,M)\in\cB(\R^{2})\times\cB(\R)$ be nondegenerate and let $P\in\cS(\mu,\nu)$ be such that $P$ is concentrated on $\Gamma$ and $P|_{M\times\R}$ is a martingale.
  \begin{enumerate}
  \item If $(\Gamma,M)$ is first-order right-monotone and second-order left-monotone, then $P$ is the Increasing Supermartingale Transport $\rP$.
  \item If $(\Gamma,M)$ is first-order left-monotone and second-order right-monotone, then $P$ is the Decreasing Supermartingale Transport $\lP$.\end{enumerate}
\end{theorem}

Before stating the proof, we record two auxiliary lemmas. The first one follows directly from the fact that $\cS(\mu,\nu)\neq\emptyset$ by Proposition~\ref{pr:convexOrder}.

\begin{lemma}\label{le:supportsOrdered}
  Let $a\in\R$ and $\mu\leq_{cd}\nu$. If $\nu$ is concentrated on $[a,\infty)$, then so is $\mu$, and moreover $\nu(\{a\})\geq\mu(\{a\})$. If $\mu\leq_{c}\nu$, the same holds for $(-\infty,a]$.
\end{lemma}

\begin{lemma}[{\cite[Lemma~5.4]{BeiglbockJuillet.12}}]\label{le:signedSupp}
  Let $\sigma$ be a nontrivial signed measure on $\R$ with $\sigma(\R)=0$ and let $\sigma=\sigma^{+}-\sigma^{-}$ be its Hahn decomposition. There exist $a\in\supp(\sigma^{+})$ and $b>a$ such that $\int (b-y)^{+}\1_{[a,\infty)}(y)\,d\sigma(y)>0$.
\end{lemma}

\begin{proof}[Proof of Theorem~\ref{th:geomCharactRP}\,(i).]
  Given $x\in\R$, we set $\mu_{x}:=\mu|_{(-\infty,x]}$ and denote by $\nu^{P}_{x}$ the second marginal of 
  $P|_{(-\infty,x]\times \R}$; that is, the image of $\mu_{x}$ under the transport $P$. Since $P$ is concentrated on $\Gamma$ and has the same mass as~$\rP$, it suffices to show that 
  \begin{equation}\label{eq:partialMarginalEquality}
  \nu^{P}_{x}=\nu^{\rP}_{x}
  \end{equation}
  for all $x\in\Gamma^{1}$.
  In a first step we will show that \eqref{eq:partialMarginalEquality} holds for all $x \in \Gamma^1 \cap M$. In view of Remark \ref{rk:leftClosed} it then follows that 
  \begin{equation}\label{eq:equalOnM}
    P|_{M\times\R} = \rP|_{M\times\R}.
  \end{equation}
  After that we will show that \eqref{eq:partialMarginalEquality} holds for all $x \in \Gamma^1$ if $M = \emptyset$, and then the latter assumption will be
  removed in a final step.

  Let us first establish an auxiliary result that will be used in Steps~1 and~2. If \eqref{eq:partialMarginalEquality} is violated for some $x \in \Gamma^1$, then the signed measure 
  $$
    \sigma:=\nu^{\rP}_{x}-\nu^{P}_{x}
  $$
  is nontrivial and we can find $a\in\supp(\sigma^{+})$ and $b>a$ as in Lemma~\ref{le:signedSupp}.   
  Note that $\sigma^{+}\leq \nu-\nu^{P}_{x}$ and that $\nu-\nu^{P}_{x}$ is the image of $\mu|_{(x,\infty)}$ under $P$. Hence, $a\in\supp(\nu-\nu^{P}_{x})$ and as $P(\Gamma)=1$, there exists a sequence of points
  \begin{equation}\label{eq:xnApprox}
    (x_{n},a_{n})\in\Gamma \quad \mbox{with $x<x_{n}$ and $a_{n}\to a$}.
  \end{equation}
  
  \vspace{.5em}
    
  \emph{Step 1: Equality of the martingale parts.} We argue by contradiction and assume that there exists $x \in \Gamma^1 \cap M$ violating \eqref{eq:partialMarginalEquality}. We first establish that 
  \begin{equation}\label{eq:meansEqual}
    \nu^{\rP}_{x}\leq_{c} \nu^{P}_{x} \quad\mbox{and in particular}\quad \bary(\nu^{\rP}_{x})=\bary(\nu^{P}_{x}).
  \end{equation}
  Indeed, in view of $x\in M$, Remark~\ref{rk:leftClosed} shows that $(-\infty,x]\cap\Gamma^{1}\subseteq M$ and thus $P|_{(-\infty,x]\times\R}$ is a martingale. Therefore, $\bary(\nu^{P}_{x}) = \bary(\mu_{x})$, and moreover $\bary(\mu_{x})\geq \bary(\nu^{\rP}_{x})$ since $\rP$ is a supermartingale. Thus, $\bary(\nu^{P}_{x})\geq\bary(\nu^{\rP}_{x})$. On the other hand, $P\in\cS(\mu,\nu)$ implies $\nu^{P}_{x}\in\casts{\mu_{x}}{\nu}$ and hence $\nu^{\rP}_{x}\leq_{cd} \nu^{P}_{x}$ by the minimality property defining $\rP$; cf.\ Theorem~\ref{th:canonicalCouplings}. In view of Proposition~\ref{pr:convexOrder}, these two facts imply~\eqref{eq:meansEqual}.
  Next, we show that
  \begin{equation}\label{eq:leftConcentrated}
    \Gamma_{t}\cap (a,\infty) = \emptyset,\quad t\leq a\wedge x,\quad\mbox{where}\quad \Gamma_{t}:=\{y\in\R:\, (t,y)\in\Gamma\}.
  \end{equation}  
  Indeed, let $t\leq a\wedge x$ and suppose that $\Gamma_{t}\cap (a,\infty) \neq \emptyset$. Then in particular $\Gamma_{t}\cap (t,\infty) \neq \emptyset$ and thus nondegeneracy, more precisely Definition~\ref{de:nondegenerate}\,(i), yields that $\Gamma_{t}\cap (-\infty,t) \neq \emptyset$ and hence $\Gamma_{t}\cap (-\infty,a) \neq \emptyset$. But now we obtain a contradiction to the second-order left-monotonicity of $\Gamma$ by using $(x_{n},a_{n})$ from~\eqref{eq:xnApprox} for $(x',y')$ and $t$ for $x$ in Definition~\ref{de:secondOrderMonotone}, for some large enough $n$.
  
  \vspace{.5em}
  
    \emph{Case (a): $x\in M$ and $x\leq a$.} As $x\leq a$, \eqref{eq:leftConcentrated} applies to all $t\leq x$ and hence $P(\Gamma)=1$ implies that $\nu^{P}_{x}$ is concentrated on $(-\infty,a]$. In view of~\eqref{eq:meansEqual} and  Lemma~\ref{le:supportsOrdered}, it follows that $\nu^{\rP}_{x}$ is concentrated on $(-\infty,a]$ as well, and $\nu^{P}_{x}(\{a\})\geq \nu^{\rP}_{x}(\{a\})$. These three facts imply
  $\int (b-y)^{+}\1_{[a,\infty)}(y)\,\nu^{\rP}_{x}(dy) 
       =(b-a) \nu^{\rP}_{x}(\{a\})$ is dominated by $(b-a) \nu^{P}_{x}(\{a\}) 
       = \int (b-y)^{+}\1_{[a,\infty)}(y)\,\nu^{P}_{x}(dy)$;  
  that is, $\int (b-y)^{+}\1_{[a,\infty)}(y)\,\sigma(dy) \leq 0$. This contradicts the choice of~$a$ and~$b$; cf.\ Lemma~\ref{le:signedSupp}.
  
  \vspace{.5em}
  
  \emph{Case (b): $x\in M$ and $a<x$.} Since $a<x$, we can argue exactly as below~\eqref{eq:meansEqual} to obtain that
    \begin{equation}\label{eq:meansEqual2}
    \nu^{\rP}_{a}\leq_{c} \nu^{P}_{a} \quad\mbox{and in particular}\quad \bary(\nu^{\rP}_{a})=\bary(\nu^{P}_{a}).
  \end{equation}
  Moreover, \eqref{eq:leftConcentrated} and $P(\Gamma)=1$ now imply that $\nu^{P}_{a}$ is concentrated on $(-\infty,a]$, and then Lemma~\ref{le:supportsOrdered} shows that 
  \begin{equation}\label{eq:nuPConcR}
    \nu^{P}_{a}, \; \nu^{\rP}_{a} \mbox{ are concentrated on }(-\infty,a]\quad\mbox{and}\quad \nu^{\rP}_{a}(\{a\})\leq \nu^{P}_{a}(\{a\}).
  \end{equation}  
    
  Next, we establish that $\nu^{P}_{x}-\nu^{P}_{a}$ is concentrated on $[a,\infty)$. Let $a<t\leq x$ be such that $\Gamma_{t}\neq\emptyset$. Since $x\in M$, Remark~\ref{rk:leftClosed} yields that $t\in M$ and now nondegeneracy, cf.\ Definition~\ref{de:nondegenerate}\,(ii'), shows that $\Gamma_{t} \cap [t,\infty)\neq \emptyset$. Then, using~\eqref{eq:xnApprox} and the second-order left-monotonicity of $\Gamma$ yield that $\Gamma_{t}\cap (-\infty,a)=\emptyset$, and therefore, $\nu^{P}_{x}-\nu^{P}_{a}$ is indeed concentrated on $[a,\infty)$. We shall prove below that 
  \begin{equation}\label{eq:auxConvexOrder}
    \nu^{\rP}_{x}-\nu^{\rP}_{a} \leq_{cd} \nu^{P}_{x}-\nu^{P}_{a}
  \end{equation}
  and thus Lemma~\ref{le:supportsOrdered} shows that $\nu^{\rP}_{x}-\nu^{\rP}_{a}$ is concentrated on $[a,\infty)$ as well. Using these facts, \eqref{eq:nuPConcR} and that $y\mapsto (b-y)^{+}\1_{[a,\infty)}(y)$ is convex decreasing on $[a,\infty)$, yields that
  $
    \int (b-y)^{+}\1_{[a,\infty)}(y)\,\nu^{\rP}_{x}(dy) 
       = \int (b-y)^{+}\1_{[a,\infty)}(y)\,(\nu^{\rP}_{x}-\nu^{\rP}_{a})(dy) + (b-a) \nu^{\rP}_{a}(\{a\})  
       $ is dominated by $\int (b-y)^{+}\1_{[a,\infty)}(y)\,(\nu^{P}_{x}-\nu^{P}_{a})(dy) + (b-a) \nu^{P}_{a}(\{a\})  
       = \int (b-y)^{+}\1_{[a,\infty)}(y)\,\nu^{P}_{x}(dy).
  $
  This again contradicts the choice of~$a$ and~$b$; cf.\ Lemma~\ref{le:signedSupp}.

  It remains to show~\eqref{eq:auxConvexOrder}. Indeed, using again that $\nu^{P}_{x}-\nu^{P}_{a}$ is concentrated on $[a,\infty)$ as well as~\eqref{eq:nuPConcR}, we have
  $$
    \nu_x^P - \nu_a^P = (\nu_x^P - \nu_a^P)|_{[a,\infty)} \leq (\nu - \nu_a^P)|_{[a,\infty)} \leq (\nu - \nu_a^\rP)|_{[a,\infty)} \leq \nu - \nu_a^\rP.
  $$
  On the other hand, we have $\mu|_{(a,x]} \leq_{cd} \nu_x^P - \nu_a^P$ by the supermartingale property, and thus
  $
    \nu_x^P - \nu_a^P \in \smallcasts{\mu|_{(a,x]}}{\nu - \nu_a^\rP}.
  $
  Since $\nu_{x}^{\rP} - \nu_{a}^{\rP} = \shadow{\mu|_{(a,x]}}{\nu - \nu_{a}^{\rP}}$ is the minimal element of the above set by the definition of $\rP$ and the additivity of the shadow (Proposition~\ref{pr:shadowAdditive}), we conclude that~\eqref{eq:auxConvexOrder} holds, and that completes the proof of Step~1.

   \vspace{.5em}
    
  \emph{Step 2: $M = \emptyset$.} Again, suppose there exists $x \in \Gamma^1$ such that
  \eqref{eq:partialMarginalEquality} is violated. Define
  $
    y_{x}:=\inf \Gamma_{x}.
  $
  If $(x',y)\in\Gamma$ and $x'<x$, first-order right-monotonicity implies that $y\geq y_{x}$ (since $M=\emptyset$), and the latter holds trivially for $x'=x$. Conversely, if $(x',y)\in\Gamma$ and $x< x'$, first-order right-monotonicity implies that $y\leq y_{x}$. As a result, $P$ is concentrated on the set
  $
    (-\infty,x]\times[y_{x},\infty) \;\cup\;  (x,\infty) \times (-\infty,y_{x}]
  $
  and as $P\in\cS(\mu,\nu)$, this implies that
  $
    \nu^{P}_{x} = \nu|_{(y_{x},\infty)} + k\delta_{y_{x}}$ where $k:=\mu((-\infty,x])- \nu((y_{x},\infty)).
  $ 
  This is the minimal element of $\smallcasts{\mu|_{(-\infty,x]}}{\nu}$ by Lemma~\ref{le:orderWhenOutOfInterval}, and thus $\nu^{P}_{x}=\nu^{\rP}_{x}$.
  
  \vspace{.5em}
    
  \emph{Step 3: $M \neq \emptyset$.}
  In the general case, let $\mu_{M}=\mu|_{M}$ and let $\nu^{P}_{M}$ denote the second marginal of $P|_{M\times\R}$. We note that $x\notin M$ yields $M\subseteq (-\infty,x]$ by Remark~\ref{rk:leftClosed} and hence $\mu_{M}\leq \mu_{x}$.
  
  We may apply the result proved in Step~2 to $\Gamma'=\Gamma\cap (M^{c}\times\R)$, $M'=\emptyset$ and the marginals $\mu'=\mu-\mu_{M}$, $\nu'=\nu-\nu_{M}$ to deduce that $P|_{M^{c}\times\R}$ is the Increasing Supermartingale Transport from $\mu'$ to $\nu'$. In particular,
  \begin{equation}\label{eq:auxReductionM}
    \shadow{\mu_x - \mu_M}{\nu - \nu_M^P} = \nu^{P}_{x} - \nu_M^P.
  \end{equation}
  Observing that~\eqref{eq:equalOnM} implies $\nu_M^P = \nu_M^{\rP} = \shadow{\mu_M}{\nu}$, the additivity of the shadow (Proposition~\ref{pr:shadowAdditive}) shows that
  $\nu_x^\rP = \shadow{\mu_x}{\nu} 
    = \nu_M^P + \shadow{\mu_x - \mu_M}{\nu - \nu_M^P}$
  which equals $\nu^{P}_{x}$ by~\eqref{eq:auxReductionM}. As $x \notin M$ was arbitrary, the proof is complete. %
\end{proof}

\begin{proof}[Proof of Theorem~\ref{th:geomCharactRP}\,(ii).]
  It will be convenient to reverse the notation with respect to the preceding proof: given $x\in\R$, we set $\mu_{x}:=\mu|_{[x,\infty)}$ and let~$\nu^{P}_{x}$ be the second marginal of $P|_{[x,\infty)\times \R}$.
  Again, we assume for contradiction that there exists $x\in\Gamma^{1}$ such that $\nu^{P}_{x}\neq\nu^{\lP}_{x}$, so that the signed measure 
  $
    \sigma:=\nu^{\lP}_{x}-\nu^{P}_{x}
  $
  is nontrivial and we can find $a\in\supp(\sigma^{+})$ and $a<b$ as in Lemma~\ref{le:signedSupp}. Similarly as in~\eqref{eq:xnApprox}, there exist
  \begin{equation}\label{eq:xnApproxL}
    (x_{n},a_{n})\in\Gamma \quad \mbox{with $x_{n}<x$ and $a_{n}\to a$}.
  \end{equation}
  Moreover, $P\in\cS(\mu,\nu)$ implies that $\nu^{P}_{x}\in\casts{\mu_{x}}{\nu}$ and hence, by minimality,
  \begin{equation}\label{eq:auxOrderL}
    \nu^{\lP}_{x}\leq_{cd} \nu^{P}_{x}.
  \end{equation}
     
  \emph{Case 1a: $x\in M$ and $a\leq x$.} We first show that 
  \begin{equation}\label{eq:nuPConcL1}
    \nu^{P}_{x}\quad\mbox{is concentrated on }[a,\infty).
  \end{equation}
  Indeed, let $t\in\Gamma^{1}$ be such that $t>x$. Suppose that $\Gamma_{t}\cap (-\infty,a)\neq\emptyset$, where $\Gamma_{t}:=\{y\in\R:\, (t,y)\in\Gamma\}$. If $t\in M$, nondegeneracy yields that $\Gamma_{t}\cap [t,\infty)\neq\emptyset$ and since $a\leq x<t$, \eqref{eq:xnApproxL} contradicts the second-order right-monotonicity of $\Gamma$. Hence, $t\notin M$. Since $x\in M$, nondegeneracy also yields that $\Gamma_{x}\cap [x,\infty)\neq\emptyset$. But now $\Gamma_{t}\cap (-\infty,a)\neq\emptyset$ and $a\leq x$ contradict first-order left-monotonicity as $t\notin M$. As a result, $\Gamma_{t}\cap (-\infty,a)=\emptyset$. To extend this to $t=x$, note that in this case we have $t\in M$. Thus, if $\Gamma_{t}\cap (-\infty,a)\neq\emptyset$, the nondegeneracy of Definition~\ref{de:nondegenerate}\,(ii) and~\eqref{eq:xnApproxL} contradict second-order right-monotonicity. We have shown that $\Gamma_{t}\cap (-\infty,a)=\emptyset$ for all $t\geq x$, and~\eqref{eq:nuPConcL1} follows. In view of~\eqref{eq:auxOrderL} and Lemma~\ref{le:supportsOrdered}, we conclude that
  \begin{equation}\label{eq:nuPConcL2}
    \nu^{\lP}_{x} \mbox{ is concentrated on }[a,\infty)\quad\mbox{and}\quad \nu^{\lP}_{x}(\{a\})\leq \nu^{P}_{x}(\{a\}).
  \end{equation}
  Since $(b-y)^{+}$ is convex and decreasing, \eqref{eq:auxOrderL}, \eqref{eq:nuPConcL1} and~\eqref{eq:nuPConcL2} then yield
  $\int (b-y)^{+}\1_{[a,\infty)}(y)\,\nu^{\lP}_{x}(dy) \leq \int (b-y)^{+}\1_{[a,\infty)}(y)\,\nu^{P}_{x}(dy)$
  which contradicts the choice of~$a$ and~$b$; cf.\ Lemma~\ref{le:signedSupp}.
 
 \vspace{.5em}
 
 \emph{Case 1b: $x\in M$ and $x<a$.} Let $t\geq a$ and suppose that $\Gamma_{t}\cap (-\infty,a)\neq\emptyset$. 
  If $t\in M$, nondegeneracy yields that $\Gamma_{t}\cap (t,\infty)\neq\emptyset$ and since $x<a\leq t$, \eqref{eq:xnApproxL} contradicts the second-order right-monotonicity of $\Gamma$. Hence, $t\notin M$, but then  $\Gamma_{t}\cap (-\infty,a)\neq\emptyset$ and \eqref{eq:xnApproxL} contradict first-order left-monotonicity. As a result, $\Gamma_{t}\cap (-\infty,a)=\emptyset$ for all $t\geq a$ and hence $\nu^{P}_{a}$ is concentrated on $[a,\infty)$. Since 
  \begin{equation}\label{eq:auxOrderL2}
    \nu^{\lP}_{a} \leq_{cd}\nu^{P}_{a}
  \end{equation}  
   can be argued as in~\eqref{eq:auxOrderL}, Lemma~\ref{le:supportsOrdered} then yields that
  \begin{equation}\label{eq:nuPConcL3}
    \nu^{P}_{a}, \; \nu^{\lP}_{a} \mbox{ are concentrated on }[a,\infty)\quad\mbox{and}\quad \nu^{\lP}_{a}(\{a\})\leq \nu^{P}_{a}(\{a\}).
  \end{equation}  
  Next, we show that symmetrically,
  \begin{equation}\label{eq:nuPConcL4}
    \nu^{P}_{x}-\nu^{P}_{a}, \; \nu^{\lP}_{x}-\nu^{\lP}_{a} \mbox{ are concentrated on }(-\infty,a]
  \end{equation}  
  \begin{equation}\label{eq:nuPConcL5}
    \mbox{and}\quad (\nu^{\lP}_{x}-\nu^{\lP}_{a})(\{a\})\leq (\nu^{P}_{x}-\nu^{P}_{x})(\{a\}).
  \end{equation}      
Indeed, let $t\in\Gamma^{1}$ be such that $x\leq t<a$ and suppose that $\Gamma_{t}\cap (a,\infty)\neq\emptyset$. Since $\Gamma_{t}\cap (-\infty,t]\neq\emptyset$ by nondegeneracy, \eqref{eq:xnApproxL} contradicts second-order right-monotonicity. Thus, $\Gamma_{t}\cap (a,\infty)=\emptyset$ and $\nu^{P}_{x}-\nu^{P}_{a}$ is concentrated on $(-\infty,a]$. In order to conclude~\eqref{eq:nuPConcL4} and~\eqref{eq:nuPConcL5} via Lemma~\ref{le:supportsOrdered}, it remains to show that $\nu^{\lP}_{x}-\nu^{\lP}_{a}\leq_{c}\nu^{P}_{x}-\nu^{P}_{a}$. Indeed, let again $t\in\Gamma^{1}$ be such that $x\leq t<a$. If $t\notin M$, then $\Gamma_{t}\cap (-\infty,t]\neq\emptyset$ and \eqref{eq:xnApproxL} contradict first-order left-monotonicity; thus $t\in M$. As a result, $P|_{[x,a)\times \R}$ is a martingale and $\bary(\nu^{P}_{x}-\nu^{P}_{a})=\bary(\mu_{x}-\mu_{a})$. Hence, we only have to show that
  \begin{equation}\label{eq:auxDiffOrder}
    \nu^{\lP}_{x}-\nu^{\lP}_{a}\leq_{cd}\nu^{P}_{x}-\nu^{P}_{a}.
  \end{equation}
  Using that $\nu^{P}_{x}-\nu^{P}_{a}$ is concentrated on $(-\infty,a]$ as well as~\eqref{eq:nuPConcL3}, we have
  $
    \nu_x^P - \nu_a^P = (\nu_x^P - \nu_a^P)|_{(-\infty,a]} \leq (\nu - \nu_a^P)|_{(-\infty,a]} \leq (\nu - \nu_a^\lP)|_{[a,\infty)} \leq \nu - \nu_a^\lP.
  $
  On the other hand, $\mu|_{[x,a)} \leq_{cd} \nu_x^P - \nu_a^P$ by the supermartingale property of~$P$, and thus~\eqref{eq:auxDiffOrder} follows from the minimality of $\lP$ and Proposition~\ref{pr:shadowAdditive}. This completes the proof of~\eqref{eq:nuPConcL4} and~\eqref{eq:nuPConcL5}. 
  
  Finally, we apply~\eqref{eq:auxOrderL2}--\eqref{eq:nuPConcL5} to find that
  $\int (b-y)^{+}  \1_{[a,\infty)}(y)\, \nu^{\lP}_{x}(dy) 
     =  \int (b-y)^{+}\1_{[a,\infty)}(y)\, (\nu^{\lP}_{x}-\nu^{\lP}_{a})(dy) +  \int (b-y)^{+}\1_{[a,\infty)}(y)\, \nu^{\lP}_{a}(dy)$ is equal to 
     $(b-a) (\nu^{\lP}_{x}-\nu^{\lP}_{a})(\{a\}) +  \int (b-y)^{+}\, \nu^{\lP}_{a}(dy) 
    \leq (b-a) (\nu^{P}_{x}-\nu^{P}_{a})(\{a\}) +  \int (b-y)^{+}\, \nu^{P}_{a}(dy) 
     = \int (b-y)^{+}\1_{[a,\infty)}(y)\, \nu^{P}_{x}(dy)
  $
  which again contradicts the choice of~$a$ and~$b$.
  
  \vspace{.5em}
    
  \emph{Case 2: $x\notin M$.} Define again $y_{x}=\inf\Gamma_{x}$; note that $y_{x}\leq x$ by nondegeneracy. Let $t\in\Gamma^{1}$ be such that $t<x$. If $\Gamma_{t}\cap (y_{x},\infty)\neq \emptyset$, then as $x\notin M$, the definition of $y_{x}$ yields a contradiction to first-order left-monotonicity. On the other hand, let $x< t$ and assume that $\Gamma_{t}\cap (-\infty,y_{x})\neq \emptyset$. If $t\notin M$, the construction of $y_{x}$ again contradicts first-order left-monotonicity; thus $t\in M$. But then nondegeneracy shows that $\Gamma_{t}\cap[t,\infty)\neq\emptyset$ and the definition of $y_{x}$ yields a contradiction to second-order right-monotonicity. Clearly, $\Gamma_{x}\subseteq [y_{x},\infty)$, and we have established that $P$ must be concentrated on the union of 
  $
    (-\infty,x)\times (-\infty,y_{x}]$ and $[x,\infty)\times [y_{x},\infty).
  $
  Since $P\in\cS(\mu,\nu)$, this implies that
  $
    \nu^{P}_{x} = \nu|_{(y_{x},\infty)} + k\delta_{y_{x}}$, where $k:=\mu([x,\infty))- \nu((y_{x},\infty)).
  $
  This is the minimal element of $\llbracket \mu|_{[x,\infty)},\nu\rrbracket$, and thus $\nu^{P}_{x}=\nu^{\lP}_{x}$, completing the proof. We remark that the proof is shorter than in~(i) due to the asymmetry of the supermartingale constraint.
\end{proof}

\section{Regularity of Spence--Mirrlees Functions}\label{se:regularityOfSpenceMirrlees}

A supermartingale Spence--Mirrlees function $f$ need not be (semi)continu\-ous. For instance, if $f(x,y)=\varphi(x)\psi(y)$ for a strictly increasing function $\varphi$ and a strictly convex and decreasing function $\psi$, then $f$ is supermartingale Spence--Mirrlees but clearly $\varphi$ need not be upper or lower semicontinuous. In general, $f$ may have a continuum of various types of discontinuities.

However, we show in Proposition~\ref{pr:spenceMirrlessCont} below that a measurable second-order Spence--Mirrlees function is automatically continuous for a finer topology on $\R^{2}$, and this topology will be coarse enough to preserve the weak compactness of $\cS(\mu,\nu)$. Thus, we can still deduce the existence of optimal transports (Lemma~\ref{le:primalExistence}) for upper semicontinuous reward functions $f$, and in particular for supermartingale Spence--Mirrlees functions. That will allow us to apply the monotonicity principle of Theorem~\ref{th:monotonicityPrinciple}.

Before stating these results, we introduce a relaxed version of the Spence--Mirrlees conditions of Definition~\ref{de:spenceMirrlees}, where increase and convexity are required in the non-strict sense---we have reserved the shorter name for the object that appears more frequently.

\begin{definition}\label{de:relaxedSpenceMirrlees}
  We call $f : \R^2 \to \R$ \emph{relaxed first-order Spence--Mirrlees} if
  $$
    f(x_{2},\cdot) - f(x_{1},\cdot)\quad\mbox{is increasing for all } x_{1}<x_{2},
  $$
  \emph{relaxed second-order Spence--Mirrlees} if
  $$
    f(x_{2},\cdot) - f(x_{1},\cdot)\quad\mbox{is convex for all } x_{1}<x_{2},
  $$
  and \emph{relaxed supermartingale Spence--Mirrlees} if $f$ is relaxed second-order Spence--Mirrlees and $-f$ is relaxed first-order Spence--Mirrlees.
\end{definition}

\begin{proposition}\label{pr:spenceMirrlessCont}
  Let $f:\R^{2}\to\R$ be Borel and relaxed second-order Spence--Mirrlees. There is a Polish topology $\tau$ on $\R$ such that $f$ is $\tau\otimes\tau$-continuous. Moreover, $\tau$ refines the Euclidean topology and induces the same Borel sets.
\end{proposition}

\begin{proof}
  We begin by constructing the functions $f_{n}$; the topology will be defined in the last step.
  
  \emph{Step 1: Regularity in $y$.} 
  We first suppose that $f$ vanishes along the $y$-axis,
  \begin{equation}\label{eq:contApproxZero}
    f(0,y)=0,\quad y\in\R.
  \end{equation}
  Under this hypothesis, the second-order Spence--Mirrlees condition implies
  \begin{equation}\label{eq:contApproxGeom}
    \begin{cases}
      f(x,\cdot)\mbox{ is convex},&  x\geq0,\\
      f(x,\cdot)\mbox{ is concave},& x\leq0.
    \end{cases}
  \end{equation}
  Therefore,  $y\mapsto f(x,y)$ admits a finite left derivative $\partial_{y}f(x,0)$ at $y=0$. We impose the further hypothesis that
  \begin{equation}\label{eq:contApproxZeroDeriv}
    \partial_{y}f(x,0)=0,\quad x\in\R.
  \end{equation}  
  Since $y\mapsto f(x,y)$ is convex or concave, its restriction to a compact interval $K_{m}=[-m,m]$ is Lipschitz continuous with some optimal Lipschitz constant $\Lip (f(x,\cdot)|_{K_{m}})<\infty$. More precisely, ~\eqref{eq:contApproxGeom} and \eqref{eq:contApproxZeroDeriv} imply that the optimal constant is the supremum of the absolute slopes of the tangents at the endpoints $y=\pm m$. The second-order Spence--Mirrlees condition implies that the absolute slopes are increasing in $|x|$; in particular, 
  \begin{equation}\label{eq:uniformLipschitzConst}
    \sup_{x\in K_{m}} \Lip (f(x,\cdot)|_{K_{m}})  = \sup_{x =\pm m} \Lip (f(x,\cdot)|_{K_{m}}) <\infty.
  \end{equation}  
    
   \emph{ Step 2: Approximation.}
   Fix $n\in\N$, let $y^{n}_{k}=2^{-n}k$ for $k\in\Z$ and let $f_{n}(x,\cdot)$ be the continuous, piecewise affine approximation to $f(x,\cdot)$ along this grid; that is, for $y^{n}_{k}\leq y< y^{n}_{k+1}$ we define
  \begin{equation}\label{eq:contApproxFn}
    f_{n}(x,y) = \lambda f(x,y^{n}_{k}) + (1-\lambda)f(x,y^{n}_{k+1}), \quad \lambda := 2^{n}(y^{n}_{k+1}-y).
  \end{equation}
  We then have $|f_{n}(x,y)-f(x,y)|\leq 2^{-n}L$ for all $y\in K_{m}$ if $L$ is a Lipschitz constant for $f(x,\cdot)$ on $K_{m}$. In view of~\eqref{eq:uniformLipschitzConst}, this shows that 
  \begin{equation}\label{eq:contApproxUnifConv}
   f_{n}\to f\quad \mbox{uniformly on $K_{m}\times K_{m}$},\quad \mbox{for all $m\in\N$}.
  \end{equation}
  
  \emph{Step 3: Refining the Topology.} Next, we introduce the topology $\tau$. The basic idea here is that if $\varphi$ is a real function with a single discontinuity at $y_{0}\in\R$, we can change the topology on $\R$ by declaring $y_{0}$ an isolated point and then $\varphi$ becomes continuous. More generally, \cite[Theorem~13.11, Lemma~13.3]{Kechris.95} show that given a countable family of Borel functions on $\R$, there exists a Polish topology $\tau\subseteq \cB(\R)$ which renders these functions continuous and refines the Euclidean topology. In particular, we can find $\tau$ such that $f(\cdot,y^{n}_{k})$ is $\tau$-continuous for all $n,k$. As $\tau$ refines the Euclidean topology, it readily follows that the functions $f_{n}$ defined in~\eqref{eq:contApproxFn} are $\tau\otimes\tau$-continuous. But now~\eqref{eq:contApproxUnifConv} yields that $f$ is continuous as well (note that since $\tau$ refines the Euclidean topology, any $\tau\otimes\tau$-neighborhood contains a bounded one).
  
  It remains to remove the hypotheses~\eqref{eq:contApproxZero} and~\eqref{eq:contApproxZeroDeriv}. The above shows that the claim holds for 
  $
    \tilde{f}(x,y):=f(x,y)-f(0,y) - (\partial_{y} [f(x,y)-f(0,y)]_{y=0}) y;
  $
  note that $\tilde{f}$ is again relaxed second-order Spence--Mirrlees. We can further refine~$\tau$ such that the two Borel functions subtracted on the right-hand side are $\tau$-continuous, and then the result for $f$ follows.
\end{proof}

The preceding result leads to the existence of optimal transports.

\begin{lemma}\label{le:primalExistence}
  Let $\mu\leq_{cd}\nu$ and let $\tau$ be a Polish topology on $\R$ which refines the Euclidean topology and induces the same Borel sets. Moreover, let $f: \R^{2}\to \overline{\R}$ be upper semicontinuous for the product topology $\tau\otimes\tau$ and suppose that $f^{+}$ is $\cS(\mu,\nu)$-uniformly integrable; i.e.,
  \begin{equation}\label{eq:fUI}
    \lim_{N\to\infty}\sup_{P\in\cS(\mu,\nu)} P(f^{+}\1_{f^{+}>N})=0.
  \end{equation}
  Then, $\bS_{\mu,\nu}(f)<\infty$ and there exists an optimal $P\in\cS(\mu,\nu)$ for $\bS_{\mu,\nu}(f)$.

  Condition~\eqref{eq:fUI} is satisfied in particular if $f(x,y)\leq a(x)+b(y)$ for some functions $a\in L^{1}(\mu)$ and $b\in L^{1}(\nu)$.
\end{lemma}

\begin{proof}
  Standard arguments show that $\cS(\mu,\nu)$ is compact in the usual topology of weak convergence as induced by the Euclidean metric. However, the weak topology on $\cS(\mu,\nu)$ induced by $\tau\otimes\tau$ does not depend on the choice of the Polish topology  $\tau$ as long as $\sigma(\tau)=\cB(\R)$; this follows from \cite[Lemma~2.3]{BeiglbockPratelli.12}. Thus, $\cS(\mu,\nu)$ is still weakly compact relative to $\tau\otimes\tau$.
  
  Under the additional condition that $f$ is bounded from above, the mapping $P\mapsto P(f)$ is upper semicontinuous by \cite[Lemma~4.3]{Villani.09}. Applying this result to $f\wedge N$ and using~\eqref{eq:fUI}, the same extends to $f$ as in the lemma, and the claim follows.
\end{proof}

We remark that compactness of $\cS(\mu,\nu)$ may fail if non-product topologies are considered on $\R^{2}$, so that the use of $\tau\otimes\tau$ is crucial. The following corollary also improves the existing results in the martingale transport case that occurs when $\mu\leq_{c}\nu$, so we state that case separately.

\begin{corollary}\label{co:optTranspCharact}
  Let $\mu\leq_{cd}\nu$ be probability measures and let $f : \R^2 \to \R$ be Borel and relaxed supermartingale Spence--Mirrlees. Suppose that there exist $a\in L^{1}(\mu)$, $b\in L^{1}(\nu)$ such that 
  $
    f(x,y)\geq a(x) + b(y)$ for all $x,y\in\R
  $
  and that $f^{+}$ is $\cS(\mu,\nu)$-uniformly integrable; cf.~\eqref{eq:fUI}. Then, $\bS_{\mu,\nu}(f)<\infty$ and $\rP\in\cS(\mu,\nu)$ is an optimizer. If $f$ is supermartingale Spence--Mirrlees, the optimizer is unique.
  The analogue holds for $\lP$ if instead $-f$ is (relaxed) supermartingale Spence--Mirrlees.
  
  If $\mu\leq_{c}\nu$, the same result holds with supermartingale Spence--Mirrlees  replaced by second-order Spence--Mirrlees, and then $\rP$ (resp.\ $\lP$) coincides with the Left-Curtain (Right-Curtain) coupling of~\cite{BeiglbockJuillet.12}.
\end{corollary}

\begin{proof}
  Let $f$ be supermartingale Spence--Mirrlees (in the strict sense). Under the stated integrability condition, Proposition~\ref{pr:spenceMirrlessCont} and Lemma~\ref{le:primalExistence} show that $\bS_{\mu,\nu}(f)<\infty$ and that an optimizer $P\in\cS(\mu,\nu)$ exists. Now, the monotonicity principle of Theorem~\ref{th:monotonicityPrinciple} and Remark~\ref{rk:lowerBoundDualityMonPrinciple} provide sets $(\Gamma,M)\in\cB(\R^{2})\times\cB(\R)$ such that $P$ is concentrated on $\Gamma$, $P|_{M\times\R}$ is a martingale and the assertion of Theorem~\ref{th:monotonicityPrinciple}\,(iii) holds. In view of Remark~\ref{rk:nondegWlog}, we may assume that $\Gamma$ is nondegenerate by passing to a subset of full $P$-measure. Proposition~\ref{pr:spenceImpliesGammaMonotone} implies that $(\Gamma,M)$ is first-order right monotone and second-order left-monotone, and then Theorem~\ref{th:geomCharactRP} yields that $P=\rP$.
  
  If $f$ is relaxed supermartingale Spence--Mirrlees, let $g$ be as in Remark~\ref{rk:smoothSMfunctions} and note that for each $n\in\N$, the function $f_{n}=f+(1/n) g$ is supermartingale Spence--Mirrlees in the strict sense. Since $f_{n}$ satisfies the stated integrability conditions, the above shows that $\rP$ is the unique optimizer for $f_{n}$. Suppose that there exists $P_{*}\in\cS(\mu,\nu)$ such that $P_{*}(f)>\rP(f)$. Then, as monotone convergence yields $P_{*}(f_{n})\to P_{*}(f)$ and $\rP(f_{n})\to \rP(f)$, it follows that $P_{*}(f_{n})>\rP(f_{n})$ for $n$ large enough, contradicting the optimality of $\rP$.
  
  The argument for $\lP$ is similar. The proofs for the martingale case are the same: when $M=\R$, the first-order monotonicity condition is vacuous.
\end{proof}

Finally, we also have the converse of Theorem~\ref{th:geomCharactRP} which completes the proofs for the main results as stated in the Introduction.

\begin{corollary}\label{co:geomCharactRPConverse}
  Let $\mu\leq_{cd}\nu$ be probability measures and let $\rP$ be the associated Increasing Supermartingale Transport. There exists a nondegenerate pair $(\Gamma,M)\in\cB(\R^{2})\times\cB(\R)$ which is first-order right-monotone and second-order left-monotone such that $\rP$ is concentrated on $\Gamma$ and $\rP|_{M\times\R}$ is a martingale.
  The analogue, exchanging left and right,  holds for~$\lP$.
\end{corollary}

\begin{proof}
  Let $g$ be a supermartingale Spence--Mirrlees function as in Remark~\ref{rk:smoothSMfunctions}. We know from Corollary~\ref{co:optTranspCharact} that $\rP$ is the unique optimal transport for $g$, and the existence of $(\Gamma,M)$ follows as in the proof of Corollary~\ref{co:optTranspCharact}.
\end{proof}

\begin{remark}\label{rk:propertiesFromIntro}
  Corollary~\ref{co:geomCharactRPConverse} shows, in particular, that the no-crossing properties of $\rP$ and $\lP$ as stated in the Introduction are true for general marginals. The preservation of order mentioned in Figure~\ref{fi:simulationDecr} follows from the two monotonicity properties and nondegeneracy, and together with Remark~\ref{rk:leftClosed}, the corollary also yields that $\rP$ has at most one transition from martingale kernels to proper supermartingale ones.
  
  A martingale transport with second-order left-monotone support is the Left-Curtain coupling of its marginals and if the first marginal has no atoms, each kernel of this transport is concentrated on two points~\cite{BeiglbockJuillet.12}. Moreover, an arbitrary transport with first-order right-monotone support is the antitone coupling and if the first marginal has no atoms, its kernels are deterministic \cite[Section~3.1]{RachevRuschendorf.98a}. As a result, if $\mu$ is diffuse,
  \begin{enumerate}
    \item $\rP|_{M^{c}\times\R}$ is of Monge-type,
    \item $\rP|_{M\times\R}$ is concentrated on the union of two graphs.
  \end{enumerate}
  The analogue holds for $\lP$, with the Right-Curtain and Hoeffding--Fr\'echet couplings.
\end{remark}

\section{Counterexamples}\label{se:counterex}

\subsection{Duality Theory}\label{se:counterexDuality}

In the Introduction and the body of the text, we have claimed that certain relaxations cannot be avoided.
In \cite{BeiglbockNutzTouzi.15}, we have already stated several examples related to the duality theory for the case of martingale transport. Bearing in mind that this is a special case of the supermartingale transport problem at hand, these examples still apply:
If the inequality defining the dual elements is stated in the classical sense as
$
  \varphi(x)+\psi(y)+h(x)(y-x)\geq f(x,y)$ for all $(x,y)\in\R^{2}
$
rather than the quasi-sure sense, a duality gap may occur; cf.\ \cite[Example~8.1]{BeiglbockNutzTouzi.15}.
A duality gap may also occur if integrability of dual elements is required in the usual sense; i.e., $\varphi\in L^{1}(\mu)$, or if $f$ has no lower bound, see \cite[Examples~8.4,\,8.5]{BeiglbockNutzTouzi.15}.

Next, let us substantiate two claims made in the body of the text. Recall that the set $\cD^{ci,pw}_{\mu,\nu}(f)$ was defined with nonnegative functions $h$, whereas for $\cD_{\mu,\nu}(f)$ nonnegativity is required only on the proper portion of the state space (Definitions~\ref{de:dualDomainPointw} and~\ref{de:globalIntegrability}). We shall show below that this is necessary.

\begin{enumerate}
  \item  The requirement that the dual elements $(\varphi,\psi,h)$ satisfy $h\geq 0$ would preclude existence of dual optimizers.
\end{enumerate}

Second, we have claimed that the restriction to proper pairs $\mu\leq_{cd}\nu$ in Section~\ref{se:closednessOnIrred} is necessary. While we have already seen that the proof of Proposition~\ref{pr:closednessIrred} crucially uses a nontrivial difference of the barycenters of $\mu$ and $\nu$ in order to control the slope of $\chi$, we still owe an argument that this is indeed unavoidable.
\begin{enumerate}
  \item[(ii)] The closedness property of $\cD^{ci,pw}_{\mu,\nu}(f)$ asserted in Proposition~\ref{pr:closednessIrred} fails if the (irreducible) pair $\mu\leq_{cd}\nu$ is not proper,
\end{enumerate}
and this remains true even if, in view of~(i), we were to alleviate the requirement that $h\geq0$.

  Indeed, let $c_{i}=i^{-3}C$, $i\geq1$, where $C>0$ is such that $\sum c_{i}=1$, and define 
  $
    \mu := \sum_{i\geq1} c_{i} \delta_{i}$ and $\nu := \frac13 \sum_{i\geq1} c_{i} (\delta_{i-1} + \delta_{i} + \delta_{i+1}).
  $
  Moreover, let $f(x,y)= \1_{x\neq y}$.
  Following \cite[Examples~8.4,\,8.5]{BeiglbockNutzTouzi.15} we find that $\mu\leq_{cd}\nu$ is irreducible and
  $$
     P :=  \sum_{i\geq1} c_{i}\, \delta_{i} \otimes \frac13(\delta_{i-1} + \delta_{i} + \delta_{i+1}) \,\in \cS(\mu,\nu)
  $$  
  is a primal optimizer. Clearly, $\bary(\mu)=\bary(\nu)$; i.e., the pair is not proper. Let $(\varphi,\psi,h)$ be a dual optimizer. Even if we are flexible about the precise definition of the dual domain, a minimal requirement to avoid a duality gap is that
  $\varphi(x) + \psi(y) + h(x)(y-x) = f(x,y)$ $P\as$ and hence
  $$
    \varphi(x) + \psi(y) + h(x)(y-x) = f(x,y), \quad (x,y)\in \N\times\N_{0},\; y\in \{x-1,x,x+1\}.
  $$  
  It follows that
  $\varphi(x) + \psi(x-1) - h(x) =1$ and $\varphi(x) + \psi(x+1) + h(x) =1$ and  
    $\varphi(x) + \psi(x) = 0$ for $x\in\N$, and all solutions of this system satisfy 
  $$
    \varphi(x) =-x^{2}+bx+c, \quad
    \psi(x) = x^{2} - bx - c, \quad
    h(x) = -2x + b
  $$
  for $x\in\N$, where $b,c\in\R$ are arbitrary constants. While any such triplet defines a dual optimizer in the sense of the body of the paper, we see that there is no solution satisfying $h\geq0$, which was our claim in~(i).
  
  To argue (ii), suppose for contradiction that the closedness property of $\cD^{ci,pw}_{\mu,\nu}(f)$ asserted in Proposition~\ref{pr:closednessIrred} were true even though $\mu\leq_{cd}\nu$ is not proper. Then, following the proofs in the body of the paper shows that the analogue of Proposition~\ref{pr:dualityIrred} would hold as well; i.e.,  there is no duality gap and there exists a dual optimizer in $\cD^{ci,pw}_{\mu,\nu}(f)$. We have seen that this is not the case with the requirement that $h\geq0$, but it fails even if this is dropped. Indeed, consider again a triplet $(\varphi,\psi,h)$ satisfying the above system of equations. If $(\varphi,\psi,h)\in \cD^{ci,pw}_{\mu,\nu}(f)$, then in particular there exists a concave and increasing moderator $\chi$ such that $\varphi - \chi \in L^{1}(\mu)$. Noting that $\mu$ has an infinite second moment and that $\varphi^{-}(x)$ has quadratic growth as $x\to\infty$ along the integers, it follows that $\chi^{-}(x)$ must have superlinear growth as $x\to\infty$. But then $\chi$ can certainly not be increasing, and we have reached the desired contradiction.

\subsection{Two Couplings that are not Canonical}\label{se:noncanonicalCouplings}

As mentioned in the Introduction, it is natural to ask if reward functions~$f$ that are first- and second-order Spence-Mirrlees are also maximized by a common supermartingale transport---i.e., $f_{xy}>0,f_{xyy}>0$ if $f$ is smooth, rather than the mixed signs that were considered in the preceding sections (see also Example~\ref{ex:decompositionNoncanonicalCase}). However, it turns out that two functions $f^{1},f^{2}$ satisfying these Spence--Mirrlees conditions may have different optimizers, even if the optimizer is unique for each $f^{i}$. This is shown in Example~\ref{ex:noncan1}. The same is true when $-f^{i}$ are first- and second-order Spence-Mirrlees, as shown by Example~\ref{ex:noncan2}; we confine ourselves to numerical counterexamples.

\begin{example}\label{ex:noncan1}
	Let $\mu$ and $\nu$ be uniformly distributed on $\{-1,0,1\}$ and $\{-4,-2.5,2\}$, respectively; then $\mu \leq_{cd} \nu$. We consider the reward functions $f^1(x,y) = e^xe^y$ and $f^2(x,y) = e^xe^y + 4xy$  which satisfy $f^i_{xy} > 0$ and $f^i_{xyy} > 0$. The corresponding optimal transports can be obtained with an LP-solver; they are unique and given by
	\begin{align*}
	  \pi^{1} &= \tfrac{5}{18}\delta_{(-1,-4)} + \tfrac{1}{18}\delta_{(-1,-2.5)} + \tfrac{5}{18}\delta_{(0,-2.5)} + \tfrac{1}{18}\delta_{(0,2)} + \tfrac{1}{18}\delta_{(1,-4)} + \tfrac{5}{18}\delta_{(1,2)},\\
	 \pi^{2} &= \tfrac{1}{3}\delta_{(-1,-4)} + \tfrac{7}{27}\delta_{(0,-2.5)} + \tfrac{2}{27}\delta_{(0,2)} + \tfrac{2}{27}\delta_{(1,-2.5)} + \tfrac{7}{27}\delta_{(1,2)}.
	\end{align*}
	Their supports are shown in Figure~\ref{fi:noncan1}. The transports are first- and second-order left-monotone with $M=\{1\}$, but the supports are different.

	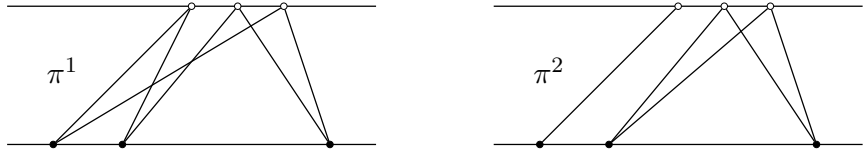
\begin{figure}[ht]
	\begin{center}
	\begin{tabular}{ccc}
	\resizebox{5cm}{!}{
	\begin{tikzpicture}[every node/.style={draw,circle,inner sep=0pt,minimum size=4pt}]
	\draw[thick] (-5,0) -- (3,0);\draw[thick] (-5,-3) -- (3,-3);
	\node[fill=white] (x0) at (-1,0) {};
	\node[fill=white] (x1) at (0.0,0) {};
	\node[fill=white] (x2) at (1,0) {};
	\node[fill=black] (y0) at (-4,-3) {};
	\node[fill=black] (y1) at (-2.5,-3) {};
	\node[fill=black] (y2) at (2,-3) {};
	\draw[thick] (x0) -- (y0);
	\draw[thick] (x0) -- (y1);
	\draw[thick] (x1) -- (y1);
	\draw[thick] (x1) -- (y2);
	\draw[thick] (x2) -- (y0);
	\draw[thick] (x2) -- (y2);
  \node[draw=none,fill=none] at (-3.8,-1.4) {\scalebox{1.7}{$\pi^{1}$}};
	\end{tikzpicture}}
	& $\qquad$ &
	\resizebox{5cm}{!}{
	\begin{tikzpicture}[every node/.style={draw,circle,inner sep=0pt,minimum size=4pt}]
	\draw[thick] (-5,0) -- (3,0);\draw[thick] (-5,-3) -- (3,-3);
	\node[fill=white] (x0) at (-1,0) {};
	\node[fill=white] (x1) at (0.0,0) {};
	\node[fill=white] (x2) at (1,0) {};
	\node[fill=black] (y0) at (-4,-3) {};
	\node[fill=black] (y1) at (-2.5,-3) {};
	\node[fill=black] (y2) at (2,-3) {};
	\draw[thick] (x0) -- (y0);
	\draw[thick] (x1) -- (y1);
	\draw[thick] (x1) -- (y2);
	\draw[thick] (x2) -- (y1);
	\draw[thick] (x2) -- (y2);
		\node[draw=none,fill=none] at (-3.8,-1.4) {\scalebox{1.7}{$\pi^{2}$}};
	\end{tikzpicture}}
	\end{tabular}
	\end{center}
	\vspace{-1em}\caption{The optimal transports from	Example \ref{ex:noncan1}}
	\label{fi:noncan1}
	\end{figure}
	
\end{example}

\begin{example}\label{ex:noncan2}
Let $\mu$ and $\nu$ be uniformly distributed on $\{-1,0,1\}$ and $\{-4,-2.5,0.5\}$, respectively; then again $\mu \leq_{cd} \nu$. We consider the reward functions $f^1(x,y) = -e^xe^y$ and $f^2(x,y) = -e^xe^y - 4xy$ which are the negatives of the functions in Example~\ref{ex:noncan1}; they satisfy $f^i_{xy} < 0$ and $f^i_{xyy} < 0$. The corresponding (unique) optimal transports are given by
	\begin{align*}
	  \pi^{1} &= \tfrac{1}{9}\delta_{(-1,-4)} + \tfrac{2}{9}\delta_{(-1,0.5)} + \tfrac{2}{9}\delta_{(0,-2.5)} + \tfrac{1}{9}\delta_{(0,0.5)} + \tfrac{2}{9}\delta_{(1,-4)} + \tfrac{1}{9}\delta_{(1,-2.5)},\\
	 \pi^{2} &= \tfrac{1}{6}\delta_{(-1,-2.5)} + \tfrac{1}{6}\delta_{(-1,0.5)} + \tfrac{1}{6}\delta_{(0,-2.5)} + \tfrac{1}{6}\delta_{(0,0.5)} + \tfrac{1}{3}\delta_{(1,-4)}.
	\end{align*}
	Their supports are shown in Figure~\ref{fi:noncan2}.  The transports are first- and second-order right-monotone with $M=\{-1\}$, but the supports are different.
	
	\begin{figure}[ht]
	\begin{center}
	\begin{tabular}{ccc}
	\resizebox{5cm}{!}{
	\begin{tikzpicture}[every node/.style={draw,circle,inner sep=0pt,minimum size=4pt}]
	\draw[thick] (-5,0) -- (2,0);\draw[thick] (-5,-3) -- (2,-3);
	\node[fill=white] (x0) at (-1,0) {};
	\node[fill=white] (x1) at (0.0,0) {};
	\node[fill=white] (x2) at (1,0) {};
	\node[fill=black] (y0) at (-4,-3) {};
	\node[fill=black] (y1) at (-2.5,-3) {};
	\node[fill=black] (y2) at (0.5,-3) {};
	\draw[thick] (x0) -- (y0);
	\draw[thick] (x0) -- (y2);
	\draw[thick] (x1) -- (y1);
	\draw[thick] (x1) -- (y2);
	\draw[thick] (x2) -- (y0);
	\draw[thick] (x2) -- (y1);
	\node[draw=none,fill=none] at (-3.8,-1.4) {\scalebox{1.5}{$\pi^{1}$}};
	\end{tikzpicture}}
	& $\qquad$ &
	\resizebox{5cm}{!}{
	\begin{tikzpicture}[every node/.style={draw,circle,inner sep=0pt,minimum size=4pt}]
	\draw[thick] (-5,0) -- (2,0);\draw[thick] (-5,-3) -- (2,-3);
	\node[fill=white] (x0) at (-1,0) {};
	\node[fill=white] (x1) at (0.0,0) {};
	\node[fill=white] (x2) at (1,0) {};
	\node[fill=black] (y0) at (-4,-3) {};
	\node[fill=black] (y1) at (-2.5,-3) {};
	\node[fill=black] (y2) at (0.5,-3) {};
	\draw[thick] (x0) -- (y1);
	\draw[thick] (x0) -- (y2);
	\draw[thick] (x1) -- (y1);
	\draw[thick] (x1) -- (y2);
	\draw[thick] (x2) -- (y0);
		\node[draw=none,fill=none] at (-3.8,-1.4) {\scalebox{1.5}{$\pi^{2}$}};
	\end{tikzpicture}}
	\end{tabular}
	\end{center}
	\vspace{-1em}\caption{The optimal transports from Example \ref{ex:noncan2}}
	\label{fi:noncan2}
	\end{figure}
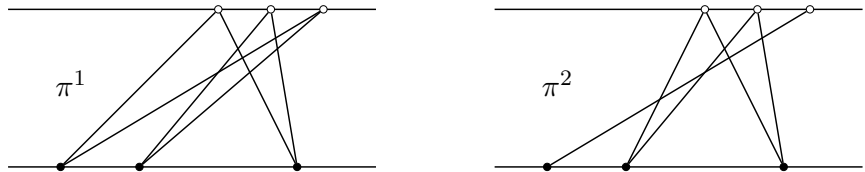
\end{example}

\newcommand{\dummy}[1]{}

\end{document}

%% file: simulationIncr2_pic.tex
\resizebox{3.8in}{!}{
\begin{tikzpicture}[every node/.style={draw,circle,fill=black,inner sep=0pt,minimum size=0.5pt}]

\draw (-4.5,-2.9) -- (-4.5,-3.1) node[fill=none,draw=none,below] {$x^*$};
\draw [decorate,decoration = {brace,amplitude=10pt}, yshift=2pt, xshift=0pt] (-1.5,0.0) -- (1.5,0.0) node [black,midway,yshift=15pt,fill=none,draw=none] {$M_0$};

\draw[thick] (-6,0) -- (10,0);\draw[thick] (-6,-3) -- (10,-3);
\node (x0) at (6.0,0) {};
\node (x2) at (6.06,0) {};
\node (x4) at (6.12,0) {};
\node (x6) at (6.18,0) {};
\node (x8) at (6.24,0) {};
\node (x10) at (6.3,0) {};
\node (x12) at (6.36,0) {};
\node (x14) at (6.42,0) {};
\node (x16) at (6.48,0) {};
\node (x18) at (6.54,0) {};
\node (x20) at (6.6,0) {};
\node (x22) at (6.66,0) {};
\node (x24) at (6.72,0) {};
\node (x26) at (6.78,0) {};
\node (x28) at (6.84,0) {};
\node (x30) at (6.9,0) {};
\node (x32) at (6.96,0) {};
\node (x34) at (7.02,0) {};
\node (x36) at (7.08,0) {};
\node (x38) at (7.14,0) {};
\node (x40) at (7.2,0) {};
\node (x42) at (7.26,0) {};
\node (x44) at (7.32,0) {};
\node (x46) at (7.38,0) {};
\node (x48) at (7.44,0) {};
\node (x50) at (7.5,0) {};
\node (x52) at (7.56,0) {};
\node (x54) at (7.62,0) {};
\node (x56) at (7.68,0) {};
\node (x58) at (7.74,0) {};
\node (x60) at (7.8,0) {};
\node (x62) at (7.86,0) {};
\node (x64) at (7.92,0) {};
\node (x66) at (7.98,0) {};
\node (x68) at (8.04,0) {};
\node (x70) at (8.1,0) {};
\node (x72) at (8.16,0) {};
\node (x74) at (8.22,0) {};
\node (x76) at (8.28,0) {};
\node (x78) at (8.34,0) {};
\node (x80) at (8.4,0) {};
\node (x82) at (8.46,0) {};
\node (x84) at (8.52,0) {};
\node (x86) at (8.58,0) {};
\node (x88) at (8.64,0) {};
\node (x90) at (8.7,0) {};
\node (x92) at (8.76,0) {};
\node (x94) at (8.82,0) {};
\node (x96) at (8.88,0) {};
\node (x98) at (8.94,0) {};
\node (x100) at (9.0,0) {};
\node (y0) at (-4.5,-3) {};
\node (y2) at (-4.44,-3) {};
\node (y4) at (-4.38,-3) {};
\node (y6) at (-4.32,-3) {};
\node (y8) at (-4.26,-3) {};
\node (y10) at (-4.2,-3) {};
\node (y12) at (-4.14,-3) {};
\node (y14) at (-4.08,-3) {};
\node (y16) at (-4.02,-3) {};
\node (y18) at (-3.96,-3) {};
\node (y20) at (-3.9,-3) {};
\node (y22) at (-3.84,-3) {};
\node (y24) at (-3.78,-3) {};
\node (y26) at (-3.72,-3) {};
\node (y28) at (-3.66,-3) {};
\node (y30) at (-3.6,-3) {};
\node (y32) at (-3.54,-3) {};
\node (y34) at (-3.48,-3) {};
\node (y36) at (-3.42,-3) {};
\node (y38) at (-3.36,-3) {};
\node (y40) at (-3.3,-3) {};
\node (y42) at (-3.24,-3) {};
\node (y44) at (-3.18,-3) {};
\node (y46) at (-3.12,-3) {};
\node (y48) at (-3.06,-3) {};
\node (y50) at (-3.0,-3) {};
\node (y52) at (3.03,-3) {};
\node (y54) at (3.09,-3) {};
\node (y56) at (3.15,-3) {};
\node (y58) at (3.21,-3) {};
\node (y60) at (3.27,-3) {};
\node (y62) at (3.33,-3) {};
\node (y64) at (3.39,-3) {};
\node (y66) at (3.45,-3) {};
\node (y68) at (3.51,-3) {};
\node (y70) at (3.57,-3) {};
\node (y72) at (3.63,-3) {};
\node (y74) at (3.69,-3) {};
\node (y76) at (3.75,-3) {};
\node (y78) at (3.81,-3) {};
\node (y80) at (3.87,-3) {};
\node (y82) at (3.93,-3) {};
\node (y84) at (3.99,-3) {};
\node (y86) at (4.05,-3) {};
\node (y88) at (4.11,-3) {};
\node (y90) at (4.17,-3) {};
\node (y92) at (4.23,-3) {};
\node (y94) at (4.29,-3) {};
\node (y96) at (4.35,-3) {};
\node (y98) at (4.41,-3) {};
\node (y100) at (4.47,-3) {};
\draw[color=gray] (x0) -- (y100);
\draw[color=gray] (x2) -- (y98);
\draw[color=gray] (x4) -- (y96);
\draw[color=gray] (x6) -- (y94);
\draw[color=gray] (x8) -- (y92);
\draw[color=gray] (x10) -- (y90);
\draw[color=gray] (x12) -- (y88);
\draw[color=gray] (x14) -- (y86);
\draw[color=gray] (x16) -- (y84);
\draw[color=gray] (x18) -- (y82);
\draw[color=gray] (x20) -- (y80);
\draw[color=gray] (x22) -- (y78);
\draw[color=gray] (x24) -- (y76);
\draw[color=gray] (x26) -- (y74);
\draw[color=gray] (x28) -- (y72);
\draw[color=gray] (x30) -- (y70);
\draw[color=gray] (x32) -- (y68);
\draw[color=gray] (x34) -- (y66);
\draw[color=gray] (x36) -- (y64);
\draw[color=gray] (x38) -- (y62);
\draw[color=gray] (x40) -- (y60);
\draw[color=gray] (x42) -- (y58);
\draw[color=gray] (x44) -- (y56);
\draw[color=gray] (x46) -- (y54);
\draw[color=gray] (x48) -- (y52);
\draw[color=gray] (x50) -- (y50);
\draw[color=gray] (x52) -- (y48);
\draw[color=gray] (x54) -- (y46);
\draw[color=gray] (x56) -- (y44);
\draw[color=gray] (x58) -- (y42);
\draw[color=gray] (x60) -- (y40);
\draw[color=gray] (x62) -- (y38);
\draw[color=gray] (x64) -- (y36);
\draw[color=gray] (x66) -- (y34);
\draw[color=gray] (x68) -- (y32);
\draw[color=gray] (x70) -- (y30);
\draw[color=gray] (x72) -- (y28);
\draw[color=gray] (x74) -- (y26);
\draw[color=gray] (x76) -- (y24);
\draw[color=gray] (x78) -- (y22);
\draw[color=gray] (x80) -- (y20);
\draw[color=gray] (x82) -- (y18);
\draw[color=gray] (x84) -- (y16);
\draw[color=gray] (x86) -- (y14);
\draw[color=gray] (x88) -- (y12);
\draw[color=gray] (x90) -- (y10);
\draw[color=gray] (x92) -- (y8);
\draw[color=gray] (x94) -- (y6);
\draw[color=gray] (x96) -- (y4);
\draw[color=gray] (x98) -- (y2);
\draw[color=gray] (x100) -- (y0);
\node (x300) at (-1.5,0) {};
\node (x302) at (-1.44,0) {};
\node (x304) at (-1.38,0) {};
\node (x306) at (-1.32,0) {};
\node (x308) at (-1.26,0) {};
\node (x310) at (-1.2,0) {};
\node (x312) at (-1.14,0) {};
\node (x314) at (-1.08,0) {};
\node (x316) at (-1.02,0) {};
\node (x318) at (-0.96,0) {};
\node (x320) at (-0.9,0) {};
\node (x322) at (-0.84,0) {};
\node (x324) at (-0.78,0) {};
\node (x326) at (-0.72,0) {};
\node (x328) at (-0.66,0) {};
\node (x330) at (-0.6,0) {};
\node (x332) at (-0.54,0) {};
\node (x334) at (-0.48,0) {};
\node (x336) at (-0.42,0) {};
\node (x338) at (-0.36,0) {};
\node (x340) at (-0.3,0) {};
\node (x342) at (-0.24,0) {};
\node (x344) at (-0.18,0) {};
\node (x346) at (-0.12,0) {};
\node (x348) at (-0.06,0) {};
\node (x350) at (0.0,0) {};
\node (x352) at (0.06,0) {};
\node (x354) at (0.12,0) {};
\node (x356) at (0.18,0) {};
\node (x358) at (0.24,0) {};
\node (x360) at (0.3,0) {};
\node (x362) at (0.36,0) {};
\node (x364) at (0.42,0) {};
\node (x366) at (0.48,0) {};
\node (x368) at (0.54,0) {};
\node (x370) at (0.6,0) {};
\node (x372) at (0.66,0) {};
\node (x374) at (0.72,0) {};
\node (x376) at (0.78,0) {};
\node (x378) at (0.84,0) {};
\node (x380) at (0.9,0) {};
\node (x382) at (0.96,0) {};
\node (x384) at (1.02,0) {};
\node (x386) at (1.08,0) {};
\node (x388) at (1.14,0) {};
\node (x390) at (1.2,0) {};
\node (x392) at (1.26,0) {};
\node (x394) at (1.32,0) {};
\node (x396) at (1.38,0) {};
\node (x398) at (1.44,0) {};
\node (x400) at (1.5,0) {};
\node (y300) at (-3.0,-3) {};
\node (y302) at (-2.94,-3) {};
\node (y304) at (-2.88,-3) {};
\node (y306) at (-2.82,-3) {};
\node (y308) at (-2.76,-3) {};
\node (y310) at (-2.7,-3) {};
\node (y312) at (-2.64,-3) {};
\node (y314) at (-2.58,-3) {};
\node (y316) at (-2.52,-3) {};
\node (y318) at (-2.46,-3) {};
\node (y320) at (-2.4,-3) {};
\node (y322) at (-2.34,-3) {};
\node (y324) at (-2.28,-3) {};
\node (y326) at (-2.22,-3) {};
\node (y328) at (-2.16,-3) {};
\node (y330) at (-2.1,-3) {};
\node (y332) at (-2.04,-3) {};
\node (y334) at (-1.98,-3) {};
\node (y336) at (-1.92,-3) {};
\node (y338) at (-1.86,-3) {};
\node (y340) at (-1.8,-3) {};
\node (y342) at (-1.74,-3) {};
\node (y344) at (-1.68,-3) {};
\node (y346) at (-1.62,-3) {};
\node (y348) at (-1.56,-3) {};
\node (y350) at (-1.5,-3) {};
\node (y352) at (-1.44,-3) {};
\node (y354) at (-1.38,-3) {};
\node (y356) at (-1.32,-3) {};
\node (y358) at (-1.26,-3) {};
\node (y360) at (-1.2,-3) {};
\node (y362) at (-1.14,-3) {};
\node (y364) at (-1.08,-3) {};
\node (y366) at (-1.02,-3) {};
\node (y368) at (-0.96,-3) {};
\node (y370) at (-0.9,-3) {};
\node (y372) at (-0.84,-3) {};
\node (y374) at (-0.78,-3) {};
\node (y376) at (-0.72,-3) {};
\node (y378) at (-0.66,-3) {};
\node (y380) at (-0.6,-3) {};
\node (y382) at (-0.54,-3) {};
\node (y384) at (-0.48,-3) {};
\node (y386) at (-0.42,-3) {};
\node (y388) at (-0.36,-3) {};
\node (y390) at (-0.3,-3) {};
\node (y392) at (-0.24,-3) {};
\node (y394) at (-0.18,-3) {};
\node (y396) at (-0.12,-3) {};
\node (y398) at (-0.06,-3) {};
\node (y400) at (0.0,-3) {};
\node (y402) at (0.06,-3) {};
\node (y404) at (0.12,-3) {};
\node (y406) at (0.18,-3) {};
\node (y408) at (0.24,-3) {};
\node (y410) at (0.3,-3) {};
\node (y412) at (0.36,-3) {};
\node (y414) at (0.42,-3) {};
\node (y416) at (0.48,-3) {};
\node (y418) at (0.54,-3) {};
\node (y420) at (0.6,-3) {};
\node (y422) at (0.66,-3) {};
\node (y424) at (0.72,-3) {};
\node (y426) at (0.78,-3) {};
\node (y428) at (0.84,-3) {};
\node (y430) at (0.9,-3) {};
\node (y432) at (0.96,-3) {};
\node (y434) at (1.02,-3) {};
\node (y436) at (1.08,-3) {};
\node (y438) at (1.14,-3) {};
\node (y440) at (1.2,-3) {};
\node (y442) at (1.26,-3) {};
\node (y444) at (1.32,-3) {};
\node (y446) at (1.38,-3) {};
\node (y448) at (1.44,-3) {};
\node (y450) at (1.5,-3) {};
\node (y452) at (1.56,-3) {};
\node (y454) at (1.62,-3) {};
\node (y456) at (1.68,-3) {};
\node (y458) at (1.74,-3) {};
\node (y460) at (1.8,-3) {};
\node (y462) at (1.86,-3) {};
\node (y464) at (1.92,-3) {};
\node (y466) at (1.98,-3) {};
\node (y468) at (2.04,-3) {};
\node (y470) at (2.1,-3) {};
\node (y472) at (2.16,-3) {};
\node (y474) at (2.22,-3) {};
\node (y476) at (2.28,-3) {};
\node (y478) at (2.34,-3) {};
\node (y480) at (2.4,-3) {};
\node (y482) at (2.46,-3) {};
\node (y484) at (2.52,-3) {};
\node (y486) at (2.58,-3) {};
\node (y488) at (2.64,-3) {};
\node (y490) at (2.7,-3) {};
\node (y492) at (2.76,-3) {};
\node (y494) at (2.82,-3) {};
\node (y496) at (2.88,-3) {};
\node (y498) at (2.94,-3) {};
\node (y500) at (3.0,-3) {};
\draw (x300) -- (y350);
\draw (x302) -- (y348);
\draw (x302) -- (y352);
\draw (x302) -- (y354);
\draw (x304) -- (y348);
\draw (x304) -- (y356);
\draw (x306) -- (y346);
\draw (x306) -- (y358);
\draw (x306) -- (y360);
\draw (x308) -- (y346);
\draw (x308) -- (y362);
\draw (x310) -- (y344);
\draw (x310) -- (y364);
\draw (x310) -- (y366);
\draw (x312) -- (y344);
\draw (x312) -- (y368);
\draw (x314) -- (y342);
\draw (x314) -- (y370);
\draw (x314) -- (y372);
\draw (x316) -- (y342);
\draw (x316) -- (y374);
\draw (x318) -- (y340);
\draw (x318) -- (y376);
\draw (x318) -- (y378);
\draw (x320) -- (y340);
\draw (x320) -- (y380);
\draw (x322) -- (y338);
\draw (x322) -- (y382);
\draw (x322) -- (y384);
\draw (x324) -- (y338);
\draw (x324) -- (y386);
\draw (x326) -- (y336);
\draw (x326) -- (y388);
\draw (x326) -- (y390);
\draw (x328) -- (y336);
\draw (x328) -- (y392);
\draw (x330) -- (y334);
\draw (x330) -- (y394);
\draw (x330) -- (y396);
\draw (x332) -- (y334);
\draw (x332) -- (y398);
\draw (x334) -- (y332);
\draw (x334) -- (y400);
\draw (x334) -- (y402);
\draw (x336) -- (y332);
\draw (x336) -- (y404);
\draw (x338) -- (y330);
\draw (x338) -- (y406);
\draw (x338) -- (y408);
\draw (x340) -- (y330);
\draw (x340) -- (y410);
\draw (x342) -- (y328);
\draw (x342) -- (y412);
\draw (x342) -- (y414);
\draw (x344) -- (y328);
\draw (x344) -- (y416);
\draw (x346) -- (y326);
\draw (x346) -- (y418);
\draw (x346) -- (y420);
\draw (x348) -- (y326);
\draw (x348) -- (y422);
\draw (x350) -- (y324);
\draw (x350) -- (y424);
\draw (x350) -- (y426);
\draw (x352) -- (y324);
\draw (x352) -- (y428);
\draw (x354) -- (y322);
\draw (x354) -- (y430);
\draw (x354) -- (y432);
\draw (x356) -- (y322);
\draw (x356) -- (y434);
\draw (x358) -- (y320);
\draw (x358) -- (y436);
\draw (x358) -- (y438);
\draw (x360) -- (y320);
\draw (x360) -- (y440);
\draw (x362) -- (y318);
\draw (x362) -- (y442);
\draw (x362) -- (y444);
\draw (x364) -- (y318);
\draw (x364) -- (y446);
\draw (x366) -- (y316);
\draw (x366) -- (y448);
\draw (x366) -- (y450);
\draw (x368) -- (y316);
\draw (x368) -- (y452);
\draw (x370) -- (y314);
\draw (x370) -- (y454);
\draw (x370) -- (y456);
\draw (x372) -- (y314);
\draw (x372) -- (y458);
\draw (x374) -- (y312);
\draw (x374) -- (y460);
\draw (x374) -- (y462);
\draw (x376) -- (y312);
\draw (x376) -- (y464);
\draw (x378) -- (y310);
\draw (x378) -- (y466);
\draw (x378) -- (y468);
\draw (x380) -- (y310);
\draw (x380) -- (y470);
\draw (x382) -- (y308);
\draw (x382) -- (y472);
\draw (x382) -- (y474);
\draw (x384) -- (y308);
\draw (x384) -- (y476);
\draw (x386) -- (y306);
\draw (x386) -- (y478);
\draw (x386) -- (y480);
\draw (x388) -- (y306);
\draw (x388) -- (y482);
\draw (x390) -- (y304);
\draw (x390) -- (y484);
\draw (x390) -- (y486);
\draw (x392) -- (y304);
\draw (x392) -- (y488);
\draw (x394) -- (y302);
\draw (x394) -- (y490);
\draw (x394) -- (y492);
\draw (x396) -- (y302);
\draw (x396) -- (y494);
\draw (x398) -- (y300);
\draw (x398) -- (y496);
\draw (x398) -- (y498);
\draw (x400) -- (y300);
\draw (x400) -- (y500);
\end{tikzpicture}}

%% file: simulationDecr2_pic.tex
\resizebox{3.8in}{!}{
\begin{tikzpicture}[every node/.style={draw,circle,fill=black,inner sep=0pt,minimum size=0.5pt}]

\draw (0.0,-2.9) -- (0.0,-3.1) node[fill=none,draw=none,below] {\footnotesize $x^*$};
\draw [decorate,decoration = {brace,amplitude=10pt}, yshift=2pt, xshift=0pt] (3.2,0.0) -- (5.1,0.0) node [black,midway,yshift=15pt,fill=none,draw=none] {\footnotesize $M_0$};
\draw [decorate,decoration = {brace,amplitude=10pt}, yshift=2pt, xshift=0pt] (-6.0,0.0) -- (0.0,0.0) node [black,midway,yshift=15pt,fill=none,draw=none] {\footnotesize $M_1$};

\draw[thick] (-6,0) -- (6,0);\draw[thick] (-6,-3) -- (6,-3);\node (x0) at (-4.5,0) {};
\node (x1) at (-4.47,0) {};
\node (x2) at (-4.44,0) {};
\node (x3) at (-4.41,0) {};
\node (x4) at (-4.38,0) {};
\node (x5) at (-4.35,0) {};
\node (x6) at (-4.32,0) {};
\node (x7) at (-4.29,0) {};
\node (x8) at (-4.26,0) {};
\node (x9) at (-4.23,0) {};
\node (x10) at (-4.2,0) {};
\node (x11) at (-4.17,0) {};
\node (x12) at (-4.14,0) {};
\node (x13) at (-4.11,0) {};
\node (x14) at (-4.08,0) {};
\node (x15) at (-4.05,0) {};
\node (x16) at (-4.02,0) {};
\node (x17) at (-3.99,0) {};
\node (x18) at (-3.96,0) {};
\node (x19) at (-3.93,0) {};
\node (x20) at (-3.9,0) {};
\node (x21) at (-3.87,0) {};
\node (x22) at (-3.84,0) {};
\node (x23) at (-3.81,0) {};
\node (x24) at (-3.78,0) {};
\node (x25) at (-3.75,0) {};
\node (x26) at (-3.72,0) {};
\node (x27) at (-3.69,0) {};
\node (x28) at (-3.66,0) {};
\node (x29) at (-3.63,0) {};
\node (x30) at (-3.6,0) {};
\node (x31) at (-3.57,0) {};
\node (x32) at (-3.54,0) {};
\node (x33) at (-3.51,0) {};
\node (x34) at (-3.48,0) {};
\node (x35) at (-3.45,0) {};
\node (x36) at (-3.42,0) {};
\node (x37) at (-3.39,0) {};
\node (x38) at (-3.36,0) {};
\node (x39) at (-3.33,0) {};
\node (x40) at (-3.3,0) {};
\node (x41) at (-3.27,0) {};
\node (x42) at (-3.24,0) {};
\node (x43) at (-3.21,0) {};
\node (x44) at (-3.18,0) {};
\node (x45) at (-3.15,0) {};
\node (x46) at (-3.12,0) {};
\node (x47) at (-3.09,0) {};
\node (x48) at (-3.06,0) {};
\node (x49) at (-3.03,0) {};
\node (x50) at (-3.0,0) {};
\node (x51) at (-2.97,0) {};
\node (x52) at (-2.94,0) {};
\node (x53) at (-2.91,0) {};
\node (x54) at (-2.88,0) {};
\node (x55) at (-2.85,0) {};
\node (x56) at (-2.82,0) {};
\node (x57) at (-2.79,0) {};
\node (x58) at (-2.76,0) {};
\node (x59) at (-2.73,0) {};
\node (x60) at (-2.7,0) {};
\node (x61) at (-2.67,0) {};
\node (x62) at (-2.64,0) {};
\node (x63) at (-2.61,0) {};
\node (x64) at (-2.58,0) {};
\node (x65) at (-2.55,0) {};
\node (x66) at (-2.52,0) {};
\node (x67) at (-2.49,0) {};
\node (x68) at (-2.46,0) {};
\node (x69) at (-2.43,0) {};
\node (x70) at (-2.4,0) {};
\node (x71) at (-2.37,0) {};
\node (x72) at (-2.34,0) {};
\node (x73) at (-2.31,0) {};
\node (x74) at (-2.28,0) {};
\node (x75) at (-2.25,0) {};
\node (x76) at (-2.22,0) {};
\node (x77) at (-2.19,0) {};
\node (x78) at (-2.16,0) {};
\node (x79) at (-2.13,0) {};
\node (x80) at (-2.1,0) {};
\node (x81) at (-2.07,0) {};
\node (x82) at (-2.04,0) {};
\node (x83) at (-2.01,0) {};
\node (x84) at (-1.98,0) {};
\node (x85) at (-1.95,0) {};
\node (x86) at (-1.92,0) {};
\node (x87) at (-1.89,0) {};
\node (x88) at (-1.86,0) {};
\node (x89) at (-1.83,0) {};
\node (x90) at (-1.8,0) {};
\node (x91) at (-1.77,0) {};
\node (x92) at (-1.74,0) {};
\node (x93) at (-1.71,0) {};
\node (x94) at (-1.68,0) {};
\node (x95) at (-1.65,0) {};
\node (x96) at (-1.62,0) {};
\node (x97) at (-1.59,0) {};
\node (x98) at (-1.56,0) {};
\node (x99) at (-1.53,0) {};
\node (x100) at (-1.5,0) {};
\node (x101) at (2.1,0) {};
\node (x102) at (2.13,0) {};
\node (x103) at (2.16,0) {};
\node (x104) at (2.19,0) {};
\node (x105) at (2.22,0) {};
\node (x106) at (2.25,0) {};
\node (x107) at (2.28,0) {};
\node (x108) at (2.31,0) {};
\node (x109) at (2.34,0) {};
\node (x110) at (2.37,0) {};
\node (x111) at (2.4,0) {};
\node (x112) at (2.43,0) {};
\node (x113) at (2.46,0) {};
\node (x114) at (2.49,0) {};
\node (x115) at (2.52,0) {};
\node (x116) at (2.55,0) {};
\node (x117) at (2.58,0) {};
\node (x118) at (2.61,0) {};
\node (x119) at (2.64,0) {};
\node (x120) at (2.67,0) {};
\node (x121) at (2.7,0) {};
\node (x122) at (2.73,0) {};
\node (x123) at (2.76,0) {};
\node (x124) at (2.79,0) {};
\node (x125) at (2.82,0) {};
\node (x126) at (2.85,0) {};
\node (x127) at (2.88,0) {};
\node (x128) at (2.91,0) {};
\node (x129) at (2.94,0) {};
\node (x130) at (2.97,0) {};
\node (x131) at (3.0,0) {};
\node (x132) at (3.03,0) {};
\node (x133) at (3.06,0) {};
\node (x134) at (3.09,0) {};
\node (x135) at (3.12,0) {};
\node (x136) at (3.15,0) {};
\node (x137) at (3.18,0) {};
\node (x138) at (3.21,0) {};
\node (x139) at (3.24,0) {};
\node (x140) at (3.27,0) {};
\node (x141) at (3.3,0) {};
\node (x142) at (3.33,0) {};
\node (x143) at (3.36,0) {};
\node (x144) at (3.39,0) {};
\node (x145) at (3.42,0) {};
\node (x146) at (3.45,0) {};
\node (x147) at (3.48,0) {};
\node (x148) at (3.51,0) {};
\node (x149) at (3.54,0) {};
\node (x150) at (3.57,0) {};
\node (x151) at (3.6,0) {};
\node (x152) at (3.63,0) {};
\node (x153) at (3.66,0) {};
\node (x154) at (3.69,0) {};
\node (x155) at (3.72,0) {};
\node (x156) at (3.75,0) {};
\node (x157) at (3.78,0) {};
\node (x158) at (3.81,0) {};
\node (x159) at (3.84,0) {};
\node (x160) at (3.87,0) {};
\node (x161) at (3.9,0) {};
\node (x162) at (3.93,0) {};
\node (x163) at (3.96,0) {};
\node (x164) at (3.99,0) {};
\node (x165) at (4.02,0) {};
\node (x166) at (4.05,0) {};
\node (x167) at (4.08,0) {};
\node (x168) at (4.11,0) {};
\node (x169) at (4.14,0) {};
\node (x170) at (4.17,0) {};
\node (x171) at (4.2,0) {};
\node (x172) at (4.23,0) {};
\node (x173) at (4.26,0) {};
\node (x174) at (4.29,0) {};
\node (x175) at (4.32,0) {};
\node (x176) at (4.35,0) {};
\node (x177) at (4.38,0) {};
\node (x178) at (4.41,0) {};
\node (x179) at (4.44,0) {};
\node (x180) at (4.47,0) {};
\node (x181) at (4.5,0) {};
\node (x182) at (4.53,0) {};
\node (x183) at (4.56,0) {};
\node (x184) at (4.59,0) {};
\node (x185) at (4.62,0) {};
\node (x186) at (4.65,0) {};
\node (x187) at (4.68,0) {};
\node (x188) at (4.71,0) {};
\node (x189) at (4.74,0) {};
\node (x190) at (4.77,0) {};
\node (x191) at (4.8,0) {};
\node (x192) at (4.83,0) {};
\node (x193) at (4.86,0) {};
\node (x194) at (4.89,0) {};
\node (x195) at (4.92,0) {};
\node (x196) at (4.95,0) {};
\node (x197) at (4.98,0) {};
\node (x198) at (5.01,0) {};
\node (x199) at (5.04,0) {};
\node (x200) at (5.07,0) {};
\node (x201) at (5.1,0) {};
\node (y0) at (-6.0,-3) {};
\node (y1) at (-5.97,-3) {};
\node (y2) at (-5.94,-3) {};
\node (y3) at (-5.91,-3) {};
\node (y4) at (-5.88,-3) {};
\node (y5) at (-5.85,-3) {};
\node (y6) at (-5.82,-3) {};
\node (y7) at (-5.79,-3) {};
\node (y8) at (-5.76,-3) {};
\node (y9) at (-5.73,-3) {};
\node (y10) at (-5.7,-3) {};
\node (y11) at (-5.67,-3) {};
\node (y12) at (-5.64,-3) {};
\node (y13) at (-5.61,-3) {};
\node (y14) at (-5.58,-3) {};
\node (y15) at (-5.55,-3) {};
\node (y16) at (-5.52,-3) {};
\node (y17) at (-5.49,-3) {};
\node (y18) at (-5.46,-3) {};
\node (y19) at (-5.43,-3) {};
\node (y20) at (-5.4,-3) {};
\node (y21) at (-5.37,-3) {};
\node (y22) at (-5.34,-3) {};
\node (y23) at (-5.31,-3) {};
\node (y24) at (-5.28,-3) {};
\node (y25) at (-5.25,-3) {};
\node (y26) at (-5.22,-3) {};
\node (y27) at (-5.19,-3) {};
\node (y28) at (-5.16,-3) {};
\node (y29) at (-5.13,-3) {};
\node (y30) at (-5.1,-3) {};
\node (y31) at (-5.07,-3) {};
\node (y32) at (-5.04,-3) {};
\node (y33) at (-5.01,-3) {};
\node (y34) at (-4.98,-3) {};
\node (y35) at (-4.95,-3) {};
\node (y36) at (-4.92,-3) {};
\node (y37) at (-4.89,-3) {};
\node (y38) at (-4.86,-3) {};
\node (y39) at (-4.83,-3) {};
\node (y40) at (-4.8,-3) {};
\node (y41) at (-4.77,-3) {};
\node (y42) at (-4.74,-3) {};
\node (y43) at (-4.71,-3) {};
\node (y44) at (-4.68,-3) {};
\node (y45) at (-4.65,-3) {};
\node (y46) at (-4.62,-3) {};
\node (y47) at (-4.59,-3) {};
\node (y48) at (-4.56,-3) {};
\node (y49) at (-4.53,-3) {};
\node (y50) at (-4.5,-3) {};
\node (y51) at (-4.47,-3) {};
\node (y52) at (-4.44,-3) {};
\node (y53) at (-4.41,-3) {};
\node (y54) at (-4.38,-3) {};
\node (y55) at (-4.35,-3) {};
\node (y56) at (-4.32,-3) {};
\node (y57) at (-4.29,-3) {};
\node (y58) at (-4.26,-3) {};
\node (y59) at (-4.23,-3) {};
\node (y60) at (-4.2,-3) {};
\node (y61) at (-4.17,-3) {};
\node (y62) at (-4.14,-3) {};
\node (y63) at (-4.11,-3) {};
\node (y64) at (-4.08,-3) {};
\node (y65) at (-4.05,-3) {};
\node (y66) at (-4.02,-3) {};
\node (y67) at (-3.99,-3) {};
\node (y68) at (-3.96,-3) {};
\node (y69) at (-3.93,-3) {};
\node (y70) at (-3.9,-3) {};
\node (y71) at (-3.87,-3) {};
\node (y72) at (-3.84,-3) {};
\node (y73) at (-3.81,-3) {};
\node (y74) at (-3.78,-3) {};
\node (y75) at (-3.75,-3) {};
\node (y76) at (-3.72,-3) {};
\node (y77) at (-3.69,-3) {};
\node (y78) at (-3.66,-3) {};
\node (y79) at (-3.63,-3) {};
\node (y80) at (-3.6,-3) {};
\node (y81) at (-3.57,-3) {};
\node (y82) at (-3.54,-3) {};
\node (y83) at (-3.51,-3) {};
\node (y84) at (-3.48,-3) {};
\node (y85) at (-3.45,-3) {};
\node (y86) at (-3.42,-3) {};
\node (y87) at (-3.39,-3) {};
\node (y88) at (-3.36,-3) {};
\node (y89) at (-3.33,-3) {};
\node (y90) at (-3.3,-3) {};
\node (y91) at (-3.27,-3) {};
\node (y92) at (-3.24,-3) {};
\node (y93) at (-3.21,-3) {};
\node (y94) at (-3.18,-3) {};
\node (y95) at (-3.15,-3) {};
\node (y96) at (-3.12,-3) {};
\node (y97) at (-3.09,-3) {};
\node (y98) at (-3.06,-3) {};
\node (y99) at (-3.03,-3) {};
\node (y100) at (-3.0,-3) {};
\node (y101) at (-2.97,-3) {};
\node (y102) at (-2.94,-3) {};
\node (y103) at (-2.91,-3) {};
\node (y104) at (-2.88,-3) {};
\node (y105) at (-2.85,-3) {};
\node (y106) at (-2.82,-3) {};
\node (y107) at (-2.79,-3) {};
\node (y108) at (-2.76,-3) {};
\node (y109) at (-2.73,-3) {};
\node (y110) at (-2.7,-3) {};
\node (y111) at (-2.67,-3) {};
\node (y112) at (-2.64,-3) {};
\node (y113) at (-2.61,-3) {};
\node (y114) at (-2.58,-3) {};
\node (y115) at (-2.55,-3) {};
\node (y116) at (-2.52,-3) {};
\node (y117) at (-2.49,-3) {};
\node (y118) at (-2.46,-3) {};
\node (y119) at (-2.43,-3) {};
\node (y120) at (-2.4,-3) {};
\node (y121) at (-2.37,-3) {};
\node (y122) at (-2.34,-3) {};
\node (y123) at (-2.31,-3) {};
\node (y124) at (-2.28,-3) {};
\node (y125) at (-2.25,-3) {};
\node (y126) at (-2.22,-3) {};
\node (y127) at (-2.19,-3) {};
\node (y128) at (-2.16,-3) {};
\node (y129) at (-2.13,-3) {};
\node (y130) at (-2.1,-3) {};
\node (y131) at (-2.07,-3) {};
\node (y132) at (-2.04,-3) {};
\node (y133) at (-2.01,-3) {};
\node (y134) at (-1.98,-3) {};
\node (y135) at (-1.95,-3) {};
\node (y136) at (-1.92,-3) {};
\node (y137) at (-1.89,-3) {};
\node (y138) at (-1.86,-3) {};
\node (y139) at (-1.83,-3) {};
\node (y140) at (-1.8,-3) {};
\node (y141) at (-1.77,-3) {};
\node (y142) at (-1.74,-3) {};
\node (y143) at (-1.71,-3) {};
\node (y144) at (-1.68,-3) {};
\node (y145) at (-1.65,-3) {};
\node (y146) at (-1.62,-3) {};
\node (y147) at (-1.59,-3) {};
\node (y148) at (-1.56,-3) {};
\node (y149) at (-1.53,-3) {};
\node (y150) at (-1.5,-3) {};
\node (y151) at (-1.47,-3) {};
\node (y152) at (-1.44,-3) {};
\node (y153) at (-1.41,-3) {};
\node (y154) at (-1.38,-3) {};
\node (y155) at (-1.35,-3) {};
\node (y156) at (-1.32,-3) {};
\node (y157) at (-1.29,-3) {};
\node (y158) at (-1.26,-3) {};
\node (y159) at (-1.23,-3) {};
\node (y160) at (-1.2,-3) {};
\node (y161) at (-1.17,-3) {};
\node (y162) at (-1.14,-3) {};
\node (y163) at (-1.11,-3) {};
\node (y164) at (-1.08,-3) {};
\node (y165) at (-1.05,-3) {};
\node (y166) at (-1.02,-3) {};
\node (y167) at (-0.99,-3) {};
\node (y168) at (-0.96,-3) {};
\node (y169) at (-0.93,-3) {};
\node (y170) at (-0.9,-3) {};
\node (y171) at (-0.87,-3) {};
\node (y172) at (-0.84,-3) {};
\node (y173) at (-0.81,-3) {};
\node (y174) at (-0.78,-3) {};
\node (y175) at (-0.75,-3) {};
\node (y176) at (-0.72,-3) {};
\node (y177) at (-0.69,-3) {};
\node (y178) at (-0.66,-3) {};
\node (y179) at (-0.63,-3) {};
\node (y180) at (-0.6,-3) {};
\node (y181) at (-0.57,-3) {};
\node (y182) at (-0.54,-3) {};
\node (y183) at (-0.51,-3) {};
\node (y184) at (-0.48,-3) {};
\node (y185) at (-0.45,-3) {};
\node (y186) at (-0.42,-3) {};
\node (y187) at (-0.39,-3) {};
\node (y188) at (-0.36,-3) {};
\node (y189) at (-0.33,-3) {};
\node (y190) at (-0.3,-3) {};
\node (y191) at (-0.27,-3) {};
\node (y192) at (-0.24,-3) {};
\node (y193) at (-0.21,-3) {};
\node (y194) at (-0.18,-3) {};
\node (y195) at (-0.15,-3) {};
\node (y196) at (-0.12,-3) {};
\node (y197) at (-0.09,-3) {};
\node (y198) at (-0.06,-3) {};
\node (y199) at (-0.03,-3) {};
\node (y200) at (0.0,-3) {};
\node (y201) at (0.0,-3) {};
\node (y202) at (0.03,-3) {};
\node (y203) at (0.06,-3) {};
\node (y204) at (0.09,-3) {};
\node (y205) at (0.12,-3) {};
\node (y206) at (0.15,-3) {};
\node (y207) at (0.18,-3) {};
\node (y208) at (0.21,-3) {};
\node (y209) at (0.24,-3) {};
\node (y210) at (0.27,-3) {};
\node (y211) at (0.3,-3) {};
\node (y212) at (0.33,-3) {};
\node (y213) at (0.36,-3) {};
\node (y214) at (0.39,-3) {};
\node (y215) at (0.42,-3) {};
\node (y216) at (0.45,-3) {};
\node (y217) at (0.48,-3) {};
\node (y218) at (0.51,-3) {};
\node (y219) at (0.54,-3) {};
\node (y220) at (0.57,-3) {};
\node (y221) at (0.6,-3) {};
\node (y222) at (0.63,-3) {};
\node (y223) at (0.66,-3) {};
\node (y224) at (0.69,-3) {};
\node (y225) at (0.72,-3) {};
\node (y226) at (0.75,-3) {};
\node (y227) at (0.78,-3) {};
\node (y228) at (0.81,-3) {};
\node (y229) at (0.84,-3) {};
\node (y230) at (0.87,-3) {};
\node (y231) at (0.9,-3) {};
\node (y232) at (0.93,-3) {};
\node (y233) at (0.96,-3) {};
\node (y234) at (0.99,-3) {};
\node (y235) at (1.02,-3) {};
\node (y236) at (1.05,-3) {};
\node (y237) at (1.08,-3) {};
\node (y238) at (1.11,-3) {};
\node (y239) at (1.14,-3) {};
\node (y240) at (1.17,-3) {};
\node (y241) at (1.2,-3) {};
\node (y242) at (1.23,-3) {};
\node (y243) at (1.26,-3) {};
\node (y244) at (1.29,-3) {};
\node (y245) at (1.32,-3) {};
\node (y246) at (1.35,-3) {};
\node (y247) at (1.38,-3) {};
\node (y248) at (1.41,-3) {};
\node (y249) at (1.44,-3) {};
\node (y250) at (1.47,-3) {};
\node (y251) at (1.5,-3) {};
\node (y252) at (1.53,-3) {};
\node (y253) at (1.56,-3) {};
\node (y254) at (1.59,-3) {};
\node (y255) at (1.62,-3) {};
\node (y256) at (1.65,-3) {};
\node (y257) at (1.68,-3) {};
\node (y258) at (1.71,-3) {};
\node (y259) at (1.74,-3) {};
\node (y260) at (1.77,-3) {};
\node (y261) at (1.8,-3) {};
\node (y262) at (1.83,-3) {};
\node (y263) at (1.86,-3) {};
\node (y264) at (1.89,-3) {};
\node (y265) at (1.92,-3) {};
\node (y266) at (1.95,-3) {};
\node (y267) at (1.98,-3) {};
\node (y268) at (2.01,-3) {};
\node (y269) at (2.04,-3) {};
\node (y270) at (2.07,-3) {};
\node (y271) at (2.1,-3) {};
\node (y272) at (2.13,-3) {};
\node (y273) at (2.16,-3) {};
\node (y274) at (2.19,-3) {};
\node (y275) at (2.22,-3) {};
\node (y276) at (2.25,-3) {};
\node (y277) at (2.28,-3) {};
\node (y278) at (2.31,-3) {};
\node (y279) at (2.34,-3) {};
\node (y280) at (2.37,-3) {};
\node (y281) at (2.4,-3) {};
\node (y282) at (2.43,-3) {};
\node (y283) at (2.46,-3) {};
\node (y284) at (2.49,-3) {};
\node (y285) at (2.52,-3) {};
\node (y286) at (2.55,-3) {};
\node (y287) at (2.58,-3) {};
\node (y288) at (2.61,-3) {};
\node (y289) at (2.64,-3) {};
\node (y290) at (2.67,-3) {};
\node (y291) at (2.7,-3) {};
\node (y292) at (2.73,-3) {};
\node (y293) at (2.76,-3) {};
\node (y294) at (2.79,-3) {};
\node (y295) at (2.82,-3) {};
\node (y296) at (2.85,-3) {};
\node (y297) at (2.88,-3) {};
\node (y298) at (2.91,-3) {};
\node (y299) at (2.94,-3) {};
\node (y300) at (2.97,-3) {};
\node (y301) at (3.0,-3) {};
\node (y302) at (3.03,-3) {};
\node (y303) at (3.06,-3) {};
\node (y304) at (3.09,-3) {};
\node (y305) at (3.12,-3) {};
\node (y306) at (3.15,-3) {};
\node (y307) at (3.18,-3) {};
\node (y308) at (3.21,-3) {};
\node (y309) at (3.24,-3) {};
\node (y310) at (3.27,-3) {};
\node (y311) at (3.3,-3) {};
\node (y312) at (3.33,-3) {};
\node (y313) at (3.36,-3) {};
\node (y314) at (3.39,-3) {};
\node (y315) at (3.42,-3) {};
\node (y316) at (3.45,-3) {};
\node (y317) at (3.48,-3) {};
\node (y318) at (3.51,-3) {};
\node (y319) at (3.54,-3) {};
\node (y320) at (3.57,-3) {};
\node (y321) at (3.6,-3) {};
\node (y322) at (3.63,-3) {};
\node (y323) at (3.66,-3) {};
\node (y324) at (3.69,-3) {};
\node (y325) at (3.72,-3) {};
\node (y326) at (3.75,-3) {};
\node (y327) at (3.78,-3) {};
\node (y328) at (3.81,-3) {};
\node (y329) at (3.84,-3) {};
\node (y330) at (3.87,-3) {};
\node (y331) at (3.9,-3) {};
\node (y332) at (3.93,-3) {};
\node (y333) at (3.96,-3) {};
\node (y334) at (3.99,-3) {};
\node (y335) at (4.02,-3) {};
\node (y336) at (4.05,-3) {};
\node (y337) at (4.08,-3) {};
\node (y338) at (4.11,-3) {};
\node (y339) at (4.14,-3) {};
\node (y340) at (4.17,-3) {};
\node (y341) at (4.2,-3) {};
\node (y342) at (4.23,-3) {};
\node (y343) at (4.26,-3) {};
\node (y344) at (4.29,-3) {};
\node (y345) at (4.32,-3) {};
\node (y346) at (4.35,-3) {};
\node (y347) at (4.38,-3) {};
\node (y348) at (4.41,-3) {};
\node (y349) at (4.44,-3) {};
\node (y350) at (4.47,-3) {};
\node (y351) at (4.5,-3) {};
\node (y352) at (4.53,-3) {};
\node (y353) at (4.56,-3) {};
\node (y354) at (4.59,-3) {};
\node (y355) at (4.62,-3) {};
\node (y356) at (4.65,-3) {};
\node (y357) at (4.68,-3) {};
\node (y358) at (4.71,-3) {};
\node (y359) at (4.74,-3) {};
\node (y360) at (4.77,-3) {};
\node (y361) at (4.8,-3) {};
\node (y362) at (4.83,-3) {};
\node (y363) at (4.86,-3) {};
\node (y364) at (4.89,-3) {};
\node (y365) at (4.92,-3) {};
\node (y366) at (4.95,-3) {};
\node (y367) at (4.98,-3) {};
\node (y368) at (5.01,-3) {};
\node (y369) at (5.04,-3) {};
\node (y370) at (5.07,-3) {};
\node (y371) at (5.1,-3) {};
\node (y372) at (5.13,-3) {};
\node (y373) at (5.16,-3) {};
\node (y374) at (5.19,-3) {};
\node (y375) at (5.22,-3) {};
\node (y376) at (5.25,-3) {};
\node (y377) at (5.28,-3) {};
\node (y378) at (5.31,-3) {};
\node (y379) at (5.34,-3) {};
\node (y380) at (5.37,-3) {};
\node (y381) at (5.4,-3) {};
\node (y382) at (5.43,-3) {};
\node (y383) at (5.46,-3) {};
\node (y384) at (5.49,-3) {};
\node (y385) at (5.52,-3) {};
\node (y386) at (5.55,-3) {};
\node (y387) at (5.58,-3) {};
\node (y388) at (5.61,-3) {};
\node (y389) at (5.64,-3) {};
\node (y390) at (5.67,-3) {};
\node (y391) at (5.7,-3) {};
\node (y392) at (5.73,-3) {};
\node (y393) at (5.76,-3) {};
\node (y394) at (5.79,-3) {};
\node (y395) at (5.82,-3) {};
\node (y396) at (5.85,-3) {};
\node (y397) at (5.88,-3) {};
\node (y398) at (5.91,-3) {};
\node (y399) at (5.94,-3) {};
\node (y400) at (5.97,-3) {};
\node (y401) at (6.0,-3) {};
\draw[very thin] (x0) -- (y0);
\draw[very thin] (x0) -- (y1);
\draw[very thin] (x0) -- (y200);
\draw[very thin] (x0) -- (y201);
\draw[very thin] (x1) -- (y1);
\draw[very thin] (x1) -- (y2);
\draw[very thin] (x1) -- (y200);
\draw[very thin] (x1) -- (y201);
\draw[very thin] (x2) -- (y2);
\draw[very thin] (x2) -- (y3);
\draw[very thin] (x2) -- (y4);
\draw[very thin] (x2) -- (y199);
\draw[very thin] (x2) -- (y200);
\draw[very thin] (x2) -- (y201);
\draw[very thin] (x3) -- (y4);
\draw[very thin] (x3) -- (y5);
\draw[very thin] (x3) -- (y199);
\draw[very thin] (x4) -- (y5);
\draw[very thin] (x4) -- (y6);
\draw[very thin] (x4) -- (y7);
\draw[very thin] (x4) -- (y198);
\draw[very thin] (x4) -- (y199);
\draw[very thin] (x5) -- (y7);
\draw[very thin] (x5) -- (y8);
\draw[very thin] (x5) -- (y198);
\draw[very thin] (x6) -- (y8);
\draw[very thin] (x6) -- (y9);
\draw[very thin] (x6) -- (y10);
\draw[very thin] (x6) -- (y197);
\draw[very thin] (x6) -- (y198);
\draw[very thin] (x7) -- (y10);
\draw[very thin] (x7) -- (y11);
\draw[very thin] (x7) -- (y197);
\draw[very thin] (x8) -- (y11);
\draw[very thin] (x8) -- (y12);
\draw[very thin] (x8) -- (y13);
\draw[very thin] (x8) -- (y196);
\draw[very thin] (x8) -- (y197);
\draw[very thin] (x9) -- (y13);
\draw[very thin] (x9) -- (y14);
\draw[very thin] (x9) -- (y196);
\draw[very thin] (x10) -- (y14);
\draw[very thin] (x10) -- (y15);
\draw[very thin] (x10) -- (y16);
\draw[very thin] (x10) -- (y195);
\draw[very thin] (x10) -- (y196);
\draw[very thin] (x11) -- (y16);
\draw[very thin] (x11) -- (y17);
\draw[very thin] (x11) -- (y195);
\draw[very thin] (x12) -- (y17);
\draw[very thin] (x12) -- (y18);
\draw[very thin] (x12) -- (y19);
\draw[very thin] (x12) -- (y194);
\draw[very thin] (x12) -- (y195);
\draw[very thin] (x13) -- (y19);
\draw[very thin] (x13) -- (y20);
\draw[very thin] (x13) -- (y194);
\draw[very thin] (x14) -- (y20);
\draw[very thin] (x14) -- (y21);
\draw[very thin] (x14) -- (y22);
\draw[very thin] (x14) -- (y193);
\draw[very thin] (x14) -- (y194);
\draw[very thin] (x15) -- (y22);
\draw[very thin] (x15) -- (y23);
\draw[very thin] (x15) -- (y193);
\draw[very thin] (x16) -- (y23);
\draw[very thin] (x16) -- (y24);
\draw[very thin] (x16) -- (y25);
\draw[very thin] (x16) -- (y192);
\draw[very thin] (x16) -- (y193);
\draw[very thin] (x17) -- (y25);
\draw[very thin] (x17) -- (y26);
\draw[very thin] (x17) -- (y192);
\draw[very thin] (x18) -- (y26);
\draw[very thin] (x18) -- (y27);
\draw[very thin] (x18) -- (y28);
\draw[very thin] (x18) -- (y191);
\draw[very thin] (x18) -- (y192);
\draw[very thin] (x19) -- (y28);
\draw[very thin] (x19) -- (y29);
\draw[very thin] (x19) -- (y191);
\draw[very thin] (x20) -- (y29);
\draw[very thin] (x20) -- (y30);
\draw[very thin] (x20) -- (y31);
\draw[very thin] (x20) -- (y190);
\draw[very thin] (x20) -- (y191);
\draw[very thin] (x21) -- (y31);
\draw[very thin] (x21) -- (y32);
\draw[very thin] (x21) -- (y190);
\draw[very thin] (x22) -- (y32);
\draw[very thin] (x22) -- (y33);
\draw[very thin] (x22) -- (y34);
\draw[very thin] (x22) -- (y189);
\draw[very thin] (x22) -- (y190);
\draw[very thin] (x23) -- (y34);
\draw[very thin] (x23) -- (y35);
\draw[very thin] (x23) -- (y189);
\draw[very thin] (x24) -- (y35);
\draw[very thin] (x24) -- (y36);
\draw[very thin] (x24) -- (y37);
\draw[very thin] (x24) -- (y188);
\draw[very thin] (x24) -- (y189);
\draw[very thin] (x25) -- (y37);
\draw[very thin] (x25) -- (y38);
\draw[very thin] (x25) -- (y188);
\draw[very thin] (x26) -- (y38);
\draw[very thin] (x26) -- (y39);
\draw[very thin] (x26) -- (y40);
\draw[very thin] (x26) -- (y187);
\draw[very thin] (x26) -- (y188);
\draw[very thin] (x27) -- (y40);
\draw[very thin] (x27) -- (y41);
\draw[very thin] (x27) -- (y187);
\draw[very thin] (x28) -- (y41);
\draw[very thin] (x28) -- (y42);
\draw[very thin] (x28) -- (y43);
\draw[very thin] (x28) -- (y186);
\draw[very thin] (x28) -- (y187);
\draw[very thin] (x29) -- (y43);
\draw[very thin] (x29) -- (y44);
\draw[very thin] (x29) -- (y186);
\draw[very thin] (x30) -- (y44);
\draw[very thin] (x30) -- (y45);
\draw[very thin] (x30) -- (y46);
\draw[very thin] (x30) -- (y185);
\draw[very thin] (x30) -- (y186);
\draw[very thin] (x31) -- (y46);
\draw[very thin] (x31) -- (y47);
\draw[very thin] (x31) -- (y185);
\draw[very thin] (x32) -- (y47);
\draw[very thin] (x32) -- (y48);
\draw[very thin] (x32) -- (y49);
\draw[very thin] (x32) -- (y184);
\draw[very thin] (x32) -- (y185);
\draw[very thin] (x33) -- (y49);
\draw[very thin] (x33) -- (y50);
\draw[very thin] (x33) -- (y184);
\draw[very thin] (x34) -- (y50);
\draw[very thin] (x34) -- (y51);
\draw[very thin] (x34) -- (y52);
\draw[very thin] (x34) -- (y183);
\draw[very thin] (x34) -- (y184);
\draw[very thin] (x35) -- (y52);
\draw[very thin] (x35) -- (y53);
\draw[very thin] (x35) -- (y183);
\draw[very thin] (x36) -- (y53);
\draw[very thin] (x36) -- (y54);
\draw[very thin] (x36) -- (y55);
\draw[very thin] (x36) -- (y182);
\draw[very thin] (x36) -- (y183);
\draw[very thin] (x37) -- (y55);
\draw[very thin] (x37) -- (y56);
\draw[very thin] (x37) -- (y182);
\draw[very thin] (x38) -- (y56);
\draw[very thin] (x38) -- (y57);
\draw[very thin] (x38) -- (y58);
\draw[very thin] (x38) -- (y181);
\draw[very thin] (x38) -- (y182);
\draw[very thin] (x39) -- (y58);
\draw[very thin] (x39) -- (y59);
\draw[very thin] (x39) -- (y181);
\draw[very thin] (x40) -- (y59);
\draw[very thin] (x40) -- (y60);
\draw[very thin] (x40) -- (y61);
\draw[very thin] (x40) -- (y180);
\draw[very thin] (x40) -- (y181);
\draw[very thin] (x41) -- (y61);
\draw[very thin] (x41) -- (y62);
\draw[very thin] (x41) -- (y180);
\draw[very thin] (x42) -- (y62);
\draw[very thin] (x42) -- (y63);
\draw[very thin] (x42) -- (y64);
\draw[very thin] (x42) -- (y179);
\draw[very thin] (x42) -- (y180);
\draw[very thin] (x43) -- (y64);
\draw[very thin] (x43) -- (y65);
\draw[very thin] (x43) -- (y179);
\draw[very thin] (x44) -- (y65);
\draw[very thin] (x44) -- (y66);
\draw[very thin] (x44) -- (y67);
\draw[very thin] (x44) -- (y178);
\draw[very thin] (x44) -- (y179);
\draw[very thin] (x45) -- (y67);
\draw[very thin] (x45) -- (y68);
\draw[very thin] (x45) -- (y178);
\draw[very thin] (x46) -- (y68);
\draw[very thin] (x46) -- (y69);
\draw[very thin] (x46) -- (y70);
\draw[very thin] (x46) -- (y177);
\draw[very thin] (x46) -- (y178);
\draw[very thin] (x47) -- (y70);
\draw[very thin] (x47) -- (y71);
\draw[very thin] (x47) -- (y177);
\draw[very thin] (x48) -- (y71);
\draw[very thin] (x48) -- (y72);
\draw[very thin] (x48) -- (y73);
\draw[very thin] (x48) -- (y176);
\draw[very thin] (x48) -- (y177);
\draw[very thin] (x49) -- (y73);
\draw[very thin] (x49) -- (y74);
\draw[very thin] (x49) -- (y176);
\draw[very thin] (x50) -- (y74);
\draw[very thin] (x50) -- (y75);
\draw[very thin] (x50) -- (y76);
\draw[very thin] (x50) -- (y175);
\draw[very thin] (x50) -- (y176);
\draw[very thin] (x51) -- (y76);
\draw[very thin] (x51) -- (y77);
\draw[very thin] (x51) -- (y175);
\draw[very thin] (x52) -- (y77);
\draw[very thin] (x52) -- (y78);
\draw[very thin] (x52) -- (y79);
\draw[very thin] (x52) -- (y174);
\draw[very thin] (x52) -- (y175);
\draw[very thin] (x53) -- (y79);
\draw[very thin] (x53) -- (y80);
\draw[very thin] (x53) -- (y174);
\draw[very thin] (x54) -- (y80);
\draw[very thin] (x54) -- (y81);
\draw[very thin] (x54) -- (y82);
\draw[very thin] (x54) -- (y173);
\draw[very thin] (x54) -- (y174);
\draw[very thin] (x55) -- (y82);
\draw[very thin] (x55) -- (y83);
\draw[very thin] (x55) -- (y173);
\draw[very thin] (x56) -- (y83);
\draw[very thin] (x56) -- (y84);
\draw[very thin] (x56) -- (y85);
\draw[very thin] (x56) -- (y172);
\draw[very thin] (x56) -- (y173);
\draw[very thin] (x57) -- (y85);
\draw[very thin] (x57) -- (y86);
\draw[very thin] (x57) -- (y172);
\draw[very thin] (x58) -- (y86);
\draw[very thin] (x58) -- (y87);
\draw[very thin] (x58) -- (y88);
\draw[very thin] (x58) -- (y171);
\draw[very thin] (x58) -- (y172);
\draw[very thin] (x59) -- (y88);
\draw[very thin] (x59) -- (y89);
\draw[very thin] (x59) -- (y171);
\draw[very thin] (x60) -- (y89);
\draw[very thin] (x60) -- (y90);
\draw[very thin] (x60) -- (y91);
\draw[very thin] (x60) -- (y170);
\draw[very thin] (x60) -- (y171);
\draw[very thin] (x61) -- (y91);
\draw[very thin] (x61) -- (y92);
\draw[very thin] (x61) -- (y170);
\draw[very thin] (x62) -- (y92);
\draw[very thin] (x62) -- (y93);
\draw[very thin] (x62) -- (y94);
\draw[very thin] (x62) -- (y169);
\draw[very thin] (x62) -- (y170);
\draw[very thin] (x63) -- (y94);
\draw[very thin] (x63) -- (y95);
\draw[very thin] (x63) -- (y169);
\draw[very thin] (x64) -- (y95);
\draw[very thin] (x64) -- (y96);
\draw[very thin] (x64) -- (y97);
\draw[very thin] (x64) -- (y168);
\draw[very thin] (x64) -- (y169);
\draw[very thin] (x65) -- (y97);
\draw[very thin] (x65) -- (y98);
\draw[very thin] (x65) -- (y168);
\draw[very thin] (x66) -- (y98);
\draw[very thin] (x66) -- (y99);
\draw[very thin] (x66) -- (y100);
\draw[very thin] (x66) -- (y167);
\draw[very thin] (x66) -- (y168);
\draw[very thin] (x67) -- (y100);
\draw[very thin] (x67) -- (y101);
\draw[very thin] (x67) -- (y167);
\draw[very thin] (x68) -- (y101);
\draw[very thin] (x68) -- (y102);
\draw[very thin] (x68) -- (y103);
\draw[very thin] (x68) -- (y166);
\draw[very thin] (x68) -- (y167);
\draw[very thin] (x69) -- (y103);
\draw[very thin] (x69) -- (y104);
\draw[very thin] (x69) -- (y166);
\draw[very thin] (x70) -- (y104);
\draw[very thin] (x70) -- (y105);
\draw[very thin] (x70) -- (y106);
\draw[very thin] (x70) -- (y165);
\draw[very thin] (x70) -- (y166);
\draw[very thin] (x71) -- (y106);
\draw[very thin] (x71) -- (y107);
\draw[very thin] (x71) -- (y165);
\draw[very thin] (x72) -- (y107);
\draw[very thin] (x72) -- (y108);
\draw[very thin] (x72) -- (y109);
\draw[very thin] (x72) -- (y164);
\draw[very thin] (x72) -- (y165);
\draw[very thin] (x73) -- (y109);
\draw[very thin] (x73) -- (y110);
\draw[very thin] (x73) -- (y164);
\draw[very thin] (x74) -- (y110);
\draw[very thin] (x74) -- (y111);
\draw[very thin] (x74) -- (y112);
\draw[very thin] (x74) -- (y163);
\draw[very thin] (x74) -- (y164);
\draw[very thin] (x75) -- (y112);
\draw[very thin] (x75) -- (y113);
\draw[very thin] (x75) -- (y163);
\draw[very thin] (x76) -- (y113);
\draw[very thin] (x76) -- (y114);
\draw[very thin] (x76) -- (y115);
\draw[very thin] (x76) -- (y162);
\draw[very thin] (x76) -- (y163);
\draw[very thin] (x77) -- (y115);
\draw[very thin] (x77) -- (y116);
\draw[very thin] (x77) -- (y162);
\draw[very thin] (x78) -- (y116);
\draw[very thin] (x78) -- (y117);
\draw[very thin] (x78) -- (y118);
\draw[very thin] (x78) -- (y161);
\draw[very thin] (x78) -- (y162);
\draw[very thin] (x79) -- (y118);
\draw[very thin] (x79) -- (y119);
\draw[very thin] (x79) -- (y161);
\draw[very thin] (x80) -- (y119);
\draw[very thin] (x80) -- (y120);
\draw[very thin] (x80) -- (y121);
\draw[very thin] (x80) -- (y160);
\draw[very thin] (x80) -- (y161);
\draw[very thin] (x81) -- (y121);
\draw[very thin] (x81) -- (y122);
\draw[very thin] (x81) -- (y160);
\draw[very thin] (x82) -- (y122);
\draw[very thin] (x82) -- (y123);
\draw[very thin] (x82) -- (y124);
\draw[very thin] (x82) -- (y159);
\draw[very thin] (x82) -- (y160);
\draw[very thin] (x83) -- (y124);
\draw[very thin] (x83) -- (y125);
\draw[very thin] (x83) -- (y159);
\draw[very thin] (x84) -- (y125);
\draw[very thin] (x84) -- (y126);
\draw[very thin] (x84) -- (y127);
\draw[very thin] (x84) -- (y158);
\draw[very thin] (x84) -- (y159);
\draw[very thin] (x85) -- (y127);
\draw[very thin] (x85) -- (y128);
\draw[very thin] (x85) -- (y158);
\draw[very thin] (x86) -- (y128);
\draw[very thin] (x86) -- (y129);
\draw[very thin] (x86) -- (y130);
\draw[very thin] (x86) -- (y157);
\draw[very thin] (x86) -- (y158);
\draw[very thin] (x87) -- (y130);
\draw[very thin] (x87) -- (y131);
\draw[very thin] (x87) -- (y157);
\draw[very thin] (x88) -- (y131);
\draw[very thin] (x88) -- (y132);
\draw[very thin] (x88) -- (y133);
\draw[very thin] (x88) -- (y156);
\draw[very thin] (x88) -- (y157);
\draw[very thin] (x89) -- (y133);
\draw[very thin] (x89) -- (y134);
\draw[very thin] (x89) -- (y156);
\draw[very thin] (x90) -- (y134);
\draw[very thin] (x90) -- (y135);
\draw[very thin] (x90) -- (y136);
\draw[very thin] (x90) -- (y155);
\draw[very thin] (x90) -- (y156);
\draw[very thin] (x91) -- (y136);
\draw[very thin] (x91) -- (y137);
\draw[very thin] (x91) -- (y155);
\draw[very thin] (x92) -- (y137);
\draw[very thin] (x92) -- (y138);
\draw[very thin] (x92) -- (y139);
\draw[very thin] (x92) -- (y154);
\draw[very thin] (x92) -- (y155);
\draw[very thin] (x93) -- (y139);
\draw[very thin] (x93) -- (y140);
\draw[very thin] (x93) -- (y154);
\draw[very thin] (x94) -- (y140);
\draw[very thin] (x94) -- (y141);
\draw[very thin] (x94) -- (y142);
\draw[very thin] (x94) -- (y153);
\draw[very thin] (x94) -- (y154);
\draw[very thin] (x95) -- (y142);
\draw[very thin] (x95) -- (y143);
\draw[very thin] (x95) -- (y153);
\draw[very thin] (x96) -- (y143);
\draw[very thin] (x96) -- (y144);
\draw[very thin] (x96) -- (y145);
\draw[very thin] (x96) -- (y152);
\draw[very thin] (x96) -- (y153);
\draw[very thin] (x97) -- (y145);
\draw[very thin] (x97) -- (y146);
\draw[very thin] (x97) -- (y152);
\draw[very thin] (x98) -- (y146);
\draw[very thin] (x98) -- (y147);
\draw[very thin] (x98) -- (y148);
\draw[very thin] (x98) -- (y151);
\draw[very thin] (x98) -- (y152);
\draw[very thin] (x99) -- (y148);
\draw[very thin] (x99) -- (y149);
\draw[very thin] (x99) -- (y151);
\draw[very thin] (x100) -- (y149);
\draw[very thin] (x100) -- (y150);
\draw[very thin] (x100) -- (y151);
\draw[very thin,color=gray] (x101) -- (y200);
\draw[very thin,color=gray] (x101) -- (y201);
\draw[very thin,color=gray] (x101) -- (y202);
\draw[very thin,color=gray] (x102) -- (y202);
\draw[very thin,color=gray] (x102) -- (y203);
\draw[very thin,color=gray] (x102) -- (y204);
\draw[very thin,color=gray] (x103) -- (y204);
\draw[very thin,color=gray] (x103) -- (y205);
\draw[very thin,color=gray] (x103) -- (y206);
\draw[very thin,color=gray] (x104) -- (y206);
\draw[very thin,color=gray] (x104) -- (y207);
\draw[very thin,color=gray] (x104) -- (y208);
\draw[very thin,color=gray] (x105) -- (y208);
\draw[very thin,color=gray] (x105) -- (y209);
\draw[very thin,color=gray] (x105) -- (y210);
\draw[very thin,color=gray] (x106) -- (y210);
\draw[very thin,color=gray] (x106) -- (y211);
\draw[very thin,color=gray] (x106) -- (y212);
\draw[very thin,color=gray] (x107) -- (y212);
\draw[very thin,color=gray] (x107) -- (y213);
\draw[very thin,color=gray] (x107) -- (y214);
\draw[very thin,color=gray] (x108) -- (y214);
\draw[very thin,color=gray] (x108) -- (y215);
\draw[very thin,color=gray] (x108) -- (y216);
\draw[very thin,color=gray] (x109) -- (y216);
\draw[very thin,color=gray] (x109) -- (y217);
\draw[very thin,color=gray] (x109) -- (y218);
\draw[very thin,color=gray] (x110) -- (y218);
\draw[very thin,color=gray] (x110) -- (y219);
\draw[very thin,color=gray] (x110) -- (y220);
\draw[very thin,color=gray] (x111) -- (y220);
\draw[very thin,color=gray] (x111) -- (y221);
\draw[very thin,color=gray] (x111) -- (y222);
\draw[very thin,color=gray] (x112) -- (y222);
\draw[very thin,color=gray] (x112) -- (y223);
\draw[very thin,color=gray] (x112) -- (y224);
\draw[very thin,color=gray] (x113) -- (y224);
\draw[very thin,color=gray] (x113) -- (y225);
\draw[very thin,color=gray] (x113) -- (y226);
\draw[very thin,color=gray] (x114) -- (y226);
\draw[very thin,color=gray] (x114) -- (y227);
\draw[very thin,color=gray] (x114) -- (y228);
\draw[very thin,color=gray] (x115) -- (y228);
\draw[very thin,color=gray] (x115) -- (y229);
\draw[very thin,color=gray] (x115) -- (y230);
\draw[very thin,color=gray] (x116) -- (y230);
\draw[very thin,color=gray] (x116) -- (y231);
\draw[very thin,color=gray] (x116) -- (y232);
\draw[very thin,color=gray] (x117) -- (y232);
\draw[very thin,color=gray] (x117) -- (y233);
\draw[very thin,color=gray] (x117) -- (y234);
\draw[very thin,color=gray] (x118) -- (y234);
\draw[very thin,color=gray] (x118) -- (y235);
\draw[very thin,color=gray] (x118) -- (y236);
\draw[very thin,color=gray] (x119) -- (y236);
\draw[very thin,color=gray] (x119) -- (y237);
\draw[very thin,color=gray] (x119) -- (y238);
\draw[very thin,color=gray] (x120) -- (y238);
\draw[very thin,color=gray] (x120) -- (y239);
\draw[very thin,color=gray] (x120) -- (y240);
\draw[very thin,color=gray] (x121) -- (y240);
\draw[very thin,color=gray] (x121) -- (y241);
\draw[very thin,color=gray] (x121) -- (y242);
\draw[very thin,color=gray] (x122) -- (y242);
\draw[very thin,color=gray] (x122) -- (y243);
\draw[very thin,color=gray] (x122) -- (y244);
\draw[very thin,color=gray] (x123) -- (y244);
\draw[very thin,color=gray] (x123) -- (y245);
\draw[very thin,color=gray] (x123) -- (y246);
\draw[very thin,color=gray] (x124) -- (y246);
\draw[very thin,color=gray] (x124) -- (y247);
\draw[very thin,color=gray] (x124) -- (y248);
\draw[very thin,color=gray] (x125) -- (y248);
\draw[very thin,color=gray] (x125) -- (y249);
\draw[very thin,color=gray] (x125) -- (y250);
\draw[very thin,color=gray] (x126) -- (y250);
\draw[very thin,color=gray] (x126) -- (y251);
\draw[very thin,color=gray] (x126) -- (y252);
\draw[very thin,color=gray] (x127) -- (y252);
\draw[very thin,color=gray] (x127) -- (y253);
\draw[very thin,color=gray] (x127) -- (y254);
\draw[very thin,color=gray] (x128) -- (y254);
\draw[very thin,color=gray] (x128) -- (y255);
\draw[very thin,color=gray] (x128) -- (y256);
\draw[very thin,color=gray] (x129) -- (y256);
\draw[very thin,color=gray] (x129) -- (y257);
\draw[very thin,color=gray] (x129) -- (y258);
\draw[very thin,color=gray] (x130) -- (y258);
\draw[very thin,color=gray] (x130) -- (y259);
\draw[very thin,color=gray] (x130) -- (y260);
\draw[very thin,color=gray] (x131) -- (y260);
\draw[very thin,color=gray] (x131) -- (y261);
\draw[very thin,color=gray] (x131) -- (y262);
\draw[very thin,color=gray] (x132) -- (y262);
\draw[very thin,color=gray] (x132) -- (y263);
\draw[very thin,color=gray] (x132) -- (y264);
\draw[very thin,color=gray] (x133) -- (y264);
\draw[very thin,color=gray] (x133) -- (y265);
\draw[very thin,color=gray] (x133) -- (y266);
\draw[very thin,color=gray] (x134) -- (y266);
\draw[very thin,color=gray] (x134) -- (y267);
\draw[very thin,color=gray] (x134) -- (y268);
\draw[very thin,color=gray] (x135) -- (y268);
\draw[very thin,color=gray] (x135) -- (y269);
\draw[very thin,color=gray] (x135) -- (y270);
\draw[very thin,color=gray] (x136) -- (y270);
\draw[very thin,color=gray] (x136) -- (y271);
\draw[very thin,color=gray] (x136) -- (y272);
\draw[very thin,color=gray] (x137) -- (y272);
\draw[very thin,color=gray] (x137) -- (y273);
\draw[very thin,color=gray] (x137) -- (y274);
\draw[very thin,color=gray] (x138) -- (y274);
\draw[very thin,color=gray] (x138) -- (y275);
\draw[very thin,color=gray] (x138) -- (y276);
\draw[very thin,color=gray] (x139) -- (y276);
\draw[very thin,color=gray] (x139) -- (y277);
\draw[very thin,color=gray] (x139) -- (y278);
\draw[very thin] (x140) -- (y278);
\draw[very thin] (x140) -- (y279);
\draw[very thin] (x140) -- (y280);
\draw[very thin] (x141) -- (y280);
\draw[very thin] (x141) -- (y281);
\draw[very thin] (x141) -- (y282);
\draw[very thin] (x141) -- (y401);
\draw[very thin] (x142) -- (y282);
\draw[very thin] (x142) -- (y283);
\draw[very thin] (x142) -- (y401);
\draw[very thin] (x143) -- (y283);
\draw[very thin] (x143) -- (y284);
\draw[very thin] (x143) -- (y285);
\draw[very thin] (x143) -- (y400);
\draw[very thin] (x143) -- (y401);
\draw[very thin] (x144) -- (y285);
\draw[very thin] (x144) -- (y286);
\draw[very thin] (x144) -- (y400);
\draw[very thin] (x145) -- (y286);
\draw[very thin] (x145) -- (y287);
\draw[very thin] (x145) -- (y288);
\draw[very thin] (x145) -- (y399);
\draw[very thin] (x145) -- (y400);
\draw[very thin] (x146) -- (y288);
\draw[very thin] (x146) -- (y289);
\draw[very thin] (x146) -- (y399);
\draw[very thin] (x147) -- (y289);
\draw[very thin] (x147) -- (y290);
\draw[very thin] (x147) -- (y291);
\draw[very thin] (x147) -- (y398);
\draw[very thin] (x147) -- (y399);
\draw[very thin] (x148) -- (y291);
\draw[very thin] (x148) -- (y292);
\draw[very thin] (x148) -- (y398);
\draw[very thin] (x149) -- (y292);
\draw[very thin] (x149) -- (y293);
\draw[very thin] (x149) -- (y294);
\draw[very thin] (x149) -- (y397);
\draw[very thin] (x149) -- (y398);
\draw[very thin] (x150) -- (y294);
\draw[very thin] (x150) -- (y295);
\draw[very thin] (x150) -- (y397);
\draw[very thin] (x151) -- (y295);
\draw[very thin] (x151) -- (y296);
\draw[very thin] (x151) -- (y297);
\draw[very thin] (x151) -- (y396);
\draw[very thin] (x151) -- (y397);
\draw[very thin] (x152) -- (y297);
\draw[very thin] (x152) -- (y298);
\draw[very thin] (x152) -- (y396);
\draw[very thin] (x153) -- (y298);
\draw[very thin] (x153) -- (y299);
\draw[very thin] (x153) -- (y300);
\draw[very thin] (x153) -- (y395);
\draw[very thin] (x153) -- (y396);
\draw[very thin] (x154) -- (y300);
\draw[very thin] (x154) -- (y301);
\draw[very thin] (x154) -- (y395);
\draw[very thin] (x155) -- (y301);
\draw[very thin] (x155) -- (y302);
\draw[very thin] (x155) -- (y303);
\draw[very thin] (x155) -- (y394);
\draw[very thin] (x155) -- (y395);
\draw[very thin] (x156) -- (y303);
\draw[very thin] (x156) -- (y304);
\draw[very thin] (x156) -- (y394);
\draw[very thin] (x157) -- (y304);
\draw[very thin] (x157) -- (y305);
\draw[very thin] (x157) -- (y306);
\draw[very thin] (x157) -- (y393);
\draw[very thin] (x157) -- (y394);
\draw[very thin] (x158) -- (y306);
\draw[very thin] (x158) -- (y307);
\draw[very thin] (x158) -- (y393);
\draw[very thin] (x159) -- (y307);
\draw[very thin] (x159) -- (y308);
\draw[very thin] (x159) -- (y309);
\draw[very thin] (x159) -- (y392);
\draw[very thin] (x159) -- (y393);
\draw[very thin] (x160) -- (y309);
\draw[very thin] (x160) -- (y310);
\draw[very thin] (x160) -- (y392);
\draw[very thin] (x161) -- (y310);
\draw[very thin] (x161) -- (y311);
\draw[very thin] (x161) -- (y312);
\draw[very thin] (x161) -- (y391);
\draw[very thin] (x161) -- (y392);
\draw[very thin] (x162) -- (y312);
\draw[very thin] (x162) -- (y313);
\draw[very thin] (x162) -- (y391);
\draw[very thin] (x163) -- (y313);
\draw[very thin] (x163) -- (y314);
\draw[very thin] (x163) -- (y315);
\draw[very thin] (x163) -- (y390);
\draw[very thin] (x163) -- (y391);
\draw[very thin] (x164) -- (y315);
\draw[very thin] (x164) -- (y316);
\draw[very thin] (x164) -- (y390);
\draw[very thin] (x165) -- (y316);
\draw[very thin] (x165) -- (y317);
\draw[very thin] (x165) -- (y318);
\draw[very thin] (x165) -- (y389);
\draw[very thin] (x165) -- (y390);
\draw[very thin] (x166) -- (y318);
\draw[very thin] (x166) -- (y319);
\draw[very thin] (x166) -- (y389);
\draw[very thin] (x167) -- (y319);
\draw[very thin] (x167) -- (y320);
\draw[very thin] (x167) -- (y321);
\draw[very thin] (x167) -- (y388);
\draw[very thin] (x167) -- (y389);
\draw[very thin] (x168) -- (y321);
\draw[very thin] (x168) -- (y322);
\draw[very thin] (x168) -- (y388);
\draw[very thin] (x169) -- (y322);
\draw[very thin] (x169) -- (y323);
\draw[very thin] (x169) -- (y324);
\draw[very thin] (x169) -- (y387);
\draw[very thin] (x169) -- (y388);
\draw[very thin] (x170) -- (y324);
\draw[very thin] (x170) -- (y325);
\draw[very thin] (x170) -- (y387);
\draw[very thin] (x171) -- (y325);
\draw[very thin] (x171) -- (y326);
\draw[very thin] (x171) -- (y327);
\draw[very thin] (x171) -- (y386);
\draw[very thin] (x171) -- (y387);
\draw[very thin] (x172) -- (y327);
\draw[very thin] (x172) -- (y328);
\draw[very thin] (x172) -- (y386);
\draw[very thin] (x173) -- (y328);
\draw[very thin] (x173) -- (y329);
\draw[very thin] (x173) -- (y330);
\draw[very thin] (x173) -- (y385);
\draw[very thin] (x173) -- (y386);
\draw[very thin] (x174) -- (y330);
\draw[very thin] (x174) -- (y331);
\draw[very thin] (x174) -- (y385);
\draw[very thin] (x175) -- (y331);
\draw[very thin] (x175) -- (y332);
\draw[very thin] (x175) -- (y333);
\draw[very thin] (x175) -- (y384);
\draw[very thin] (x175) -- (y385);
\draw[very thin] (x176) -- (y333);
\draw[very thin] (x176) -- (y334);
\draw[very thin] (x176) -- (y384);
\draw[very thin] (x177) -- (y334);
\draw[very thin] (x177) -- (y335);
\draw[very thin] (x177) -- (y336);
\draw[very thin] (x177) -- (y383);
\draw[very thin] (x177) -- (y384);
\draw[very thin] (x178) -- (y336);
\draw[very thin] (x178) -- (y337);
\draw[very thin] (x178) -- (y383);
\draw[very thin] (x179) -- (y337);
\draw[very thin] (x179) -- (y338);
\draw[very thin] (x179) -- (y339);
\draw[very thin] (x179) -- (y382);
\draw[very thin] (x179) -- (y383);
\draw[very thin] (x180) -- (y339);
\draw[very thin] (x180) -- (y340);
\draw[very thin] (x180) -- (y382);
\draw[very thin] (x181) -- (y340);
\draw[very thin] (x181) -- (y341);
\draw[very thin] (x181) -- (y342);
\draw[very thin] (x181) -- (y381);
\draw[very thin] (x181) -- (y382);
\draw[very thin] (x182) -- (y342);
\draw[very thin] (x182) -- (y343);
\draw[very thin] (x182) -- (y381);
\draw[very thin] (x183) -- (y343);
\draw[very thin] (x183) -- (y344);
\draw[very thin] (x183) -- (y345);
\draw[very thin] (x183) -- (y380);
\draw[very thin] (x183) -- (y381);
\draw[very thin] (x184) -- (y345);
\draw[very thin] (x184) -- (y346);
\draw[very thin] (x184) -- (y380);
\draw[very thin] (x185) -- (y346);
\draw[very thin] (x185) -- (y347);
\draw[very thin] (x185) -- (y348);
\draw[very thin] (x185) -- (y379);
\draw[very thin] (x185) -- (y380);
\draw[very thin] (x186) -- (y348);
\draw[very thin] (x186) -- (y349);
\draw[very thin] (x186) -- (y379);
\draw[very thin] (x187) -- (y349);
\draw[very thin] (x187) -- (y350);
\draw[very thin] (x187) -- (y351);
\draw[very thin] (x187) -- (y378);
\draw[very thin] (x187) -- (y379);
\draw[very thin] (x188) -- (y351);
\draw[very thin] (x188) -- (y352);
\draw[very thin] (x188) -- (y378);
\draw[very thin] (x189) -- (y352);
\draw[very thin] (x189) -- (y353);
\draw[very thin] (x189) -- (y354);
\draw[very thin] (x189) -- (y377);
\draw[very thin] (x189) -- (y378);
\draw[very thin] (x190) -- (y354);
\draw[very thin] (x190) -- (y355);
\draw[very thin] (x190) -- (y377);
\draw[very thin] (x191) -- (y355);
\draw[very thin] (x191) -- (y356);
\draw[very thin] (x191) -- (y357);
\draw[very thin] (x191) -- (y376);
\draw[very thin] (x191) -- (y377);
\draw[very thin] (x192) -- (y357);
\draw[very thin] (x192) -- (y358);
\draw[very thin] (x192) -- (y376);
\draw[very thin] (x193) -- (y358);
\draw[very thin] (x193) -- (y359);
\draw[very thin] (x193) -- (y360);
\draw[very thin] (x193) -- (y375);
\draw[very thin] (x193) -- (y376);
\draw[very thin] (x194) -- (y360);
\draw[very thin] (x194) -- (y361);
\draw[very thin] (x194) -- (y375);
\draw[very thin] (x195) -- (y361);
\draw[very thin] (x195) -- (y362);
\draw[very thin] (x195) -- (y363);
\draw[very thin] (x195) -- (y374);
\draw[very thin] (x195) -- (y375);
\draw[very thin] (x196) -- (y363);
\draw[very thin] (x196) -- (y364);
\draw[very thin] (x196) -- (y374);
\draw[very thin] (x197) -- (y364);
\draw[very thin] (x197) -- (y365);
\draw[very thin] (x197) -- (y366);
\draw[very thin] (x197) -- (y373);
\draw[very thin] (x197) -- (y374);
\draw[very thin] (x198) -- (y366);
\draw[very thin] (x198) -- (y367);
\draw[very thin] (x198) -- (y373);
\draw[very thin] (x199) -- (y367);
\draw[very thin] (x199) -- (y368);
\draw[very thin] (x199) -- (y369);
\draw[very thin] (x199) -- (y372);
\draw[very thin] (x199) -- (y373);
\draw[very thin] (x200) -- (y369);
\draw[very thin] (x200) -- (y370);
\draw[very thin] (x200) -- (y372);
\draw[very thin] (x201) -- (y370);
\draw[very thin] (x201) -- (y371);
\draw[very thin] (x201) -- (y372);
\end{tikzpicture}
}